\documentclass{article}
\usepackage{cite,verbatim}
\usepackage{mathtools}
\usepackage{amssymb,amsmath,mathrsfs,graphicx}
\usepackage{fancyhdr}
\usepackage{float}
\usepackage{tikz}
\usetikzlibrary{matrix,arrows}
\usepackage{graphicx}
\usepackage{epstopdf}
\newtheorem{precor}{{\bf Corollary}}

\newenvironment{cor}{\begin{precor}{\hspace{-0.5
em}{\bf.\ }}}{\end{precor}}
\newtheorem{precon}{{\bf Conjecture}}

\newtheorem{predefin}{{\bf Definition}}

\newenvironment{defin}[1]{\begin{predefin}{\hspace{-0.5
em}{\bf.\ }}{\rm
#1}\hfill{$\blacktriangleleft$}}{\end{predefin}}
\newtheorem{preexm}{{\bf Example}}

\newenvironment{exm}[1]{\begin{preexm}{\hspace{-0.5
em}{\bf.\ }}{\rm #1}\hfill{$\blacktriangleright$}}{\end{preexm}}

\newtheorem{preappl}{{\bf Application}}

\newtheorem{prelem}{{\bf Lemma}}

\newenvironment{lem}{\begin{prelem}{\hspace{-0.5
em}{\bf.\ }}}{\end{prelem}}
\newtheorem{preproof}{{\bf Proof.\ }}

\newenvironment{proof}[1]{\begin{preproof}{\rm
#1}\hfill{$\blacksquare$}}{\end{preproof}}
\newtheorem{presproof}{{\bf Sketch of Proof.\ }}

\newtheorem{prethm}{{\bf Theorem}}

\newenvironment{thm}{\begin{prethm}{\hspace{-0.5
em}{\bf.\ }}}{\end{prethm}}
\newtheorem{prealphthm}{{\bf Theorem}}

\newenvironment{alphthm}{\begin{prealphthm}{\hspace{-0.5
em}{\bf.\ }}}{\end{prealphthm}}
\newtheorem{prepro}{{\bf Proposition}}

\newenvironment{pro}{\begin{prepro}{\hspace{-0.5
em}{\bf.\ }}}{\end{prepro}}
\newtheorem{preprb}{{\bf Problem}}

\newenvironment{prb}{\begin{preprb}{\hspace{-0.5
em}{\bf.\ }}}{\end{preprb}}
\def\symtw{symmetric}

\def\conct[#1,#2]{\mbox {${#1} \leftrightarrow {#2}$}}
\def\dconct[#1,#2]{\mbox {${#1} \rightarrow {#2}$}}
\def\deg[#1,#2]{\mbox {$d_{_{#1}}(#2)$}}
\def\mindeg[#1]{\mbox {$\delta_{_{#1}}$}}
\def\maxdeg[#1]{\mbox {$\Delta_{_{#1}}$}}
\def\outdeg[#1,#2]{\mbox {$d_{_{#1}}^{^+}(#2)$}}
\def\minoutdeg[#1]{\mbox {$\delta_{_{#1}}^{^+}$}}
\def\maxoutdeg[#1]{\mbox {$\Delta_{_{#1}}^{^+}$}}
\def\indeg[#1,#2]{\mbox {$d_{_{#1}}^{^-}(#2)$}}
\def\minindeg[#1]{\mbox {$\delta_{_{#1}}^{^-}$}}
\def\maxindeg[#1]{\mbox {$\Delta_{_{#1}}^{^-}$}}
\def\isdef{\mbox {$\ \stackrel{\rm def}{=} \ $}}
\def\dre[#1,#2,#3]{\mbox {${\cal E}_{_{#3}}(#1,#2)$}}
\def\pdre[#1,#2,#3]{\mbox {${\cal P}_{_{#3}}(#1,#2)$}}
\def\var[#1,#2]{\mbox {${\rm Var}_{_{#1}}(#2)$}}
\def\ls[#1]{\mbox {$\xi^{^{#1}}$}}
\def\onvhom[#1,#2]{\mbox {${\rm Hom^{v}}(#1,#2)$}}
\def\onehom[#1,#2]{\mbox {${\rm Hom^{e}}(#1,#2)$}}
\def\core[#1]{\mbox {$#1^{^{\bullet}}$}}
\def\cay[#1,#2]{\mbox {${\rm Cay}({#1},{#2})$}}
\def\cays[#1,#2]{\mbox {${\rm Cay_{s}}({#1},{#2})$}}
\def\dirc[#1]{\mbox {$\stackrel{\rightarrow}{C}_{_{#1}}$}}
\def\cycl[#1]{\mbox {${\bf Z}_{_{#1}}$}}

\def\sdg[#1]{\mbox {$\stackrel{\leftrightarrow}{#1}$}}
\def\ghom{\mbox {${\rm Hom_{_{\Gamma}}}$}}
\def\hom{\mbox {${\rm Hom}$}}
\def\lhom{\mbox {${\rm Hom_{_{\ell}}}$}}

\def\chom{\mbox {${\rm Hom_{_{\Gamma,m}}}$}}

\def\slhom{\mbox {${\rm Hom^{^{s}}_{_{\ell}}}$}}
\def\schom{\mbox {${\rm Hom^{^{s}}_{_{\Gamma,m}}}$}}
\def\homeq {\approx}
\def\bx {{\bf x}}
\def\by {{\bf y}}

\def\bz {{\bf z}}
\def\bu {{\bf u}}
\def\bv {{\bf v}}
\def\bw {{\bf w}}
\def\J {{\epsilon}}

\long\def\Perm[#1]{\mbox{$\lfloor #1 \rceil$}}
\long\def\nat[#1]{\mbox{$\widehat{1...#1}$}}
\def\ident {\mbox{$\textit{id}$}}
\def\Ident {\mbox{$\textbf{id}$}}
\newcommand{\matr}[1]{\mathrm{#1}} 
\newcommand{\cyl}[1]{\boxtimes_{_{#1}}} 
\newcommand{\scyl}[1]{\boxtimes^{^{s}}_{_{#1}}} 

 \def\authora{Amir Daneshgar}
\def\authorb{Mohsen Hejrati\footnote{M.~Hejrati's contributions to this article
were obtained when he was a BSc. student at Sharif University of Technology.}}
\def\authorc{Meysam Madani}

\title{Cylindrical Graph Construction\\ (definition and basic properties)}
\author{
{ \authora}\\
{ \authorb}\\
{ \authorc}\\[3mm]
{\it Department of Mathematical Sciences} \\
{\it Sharif University of Technology} \\}

\begin{document}

\maketitle
\begin{abstract}
\noindent In this article we introduce the {\it cylindrical construction} for graphs and investigate its basic properties.
We state a main result claiming a weak tensor-like duality for this construction.
Details of our motivations and applications of the construction will appear elsewhere.
\end{abstract}
\section{Introduction}
This article is about a tensor-hom like duality for a specific graph construction that will be called the {\it cylindrical construction}. Strictly speaking, the construction can be described as replacement of edges of a graph by some other graph(s)
(here called cylinders), where many different variants of this construction can be found within the vast literature
of graph theory (e.g. Pultr templates \cite{pultr, ladjoint, adjoint} and replacement \cite{HENE, HRG}).
As our primary motivation for this construction was related to introducing a cobordism theory for graphs and was initiated
through algorithmic problems, we present a general form of this
edge-replacement construction in which there also may be some twists at each terminal end (here called bases) of these cylinders.
In Secion~\ref{sec:cylconst}
we show how this general construction and its dual almost cover most well-known graph constructions. It is also interesting to
note that the role of twists in this construction is fundamental and may give rise to new constructs and results
(e.g. see Sections~\ref{sec:cylconst} for $k$-lifts as cylindrical constructions by identity cylinders and note a recent breakthrough of A.W.Marcus, D.A.Spielman and N.Srivastava \cite{MSS} and references therein for the background).

 In Section~\ref{sec:basics} we go through the basic definitions and notations we will be needing hereafter. An important part of this framework is going through the definitions of different categories of graphs and the pushout construction in them, where we
discuss some subtleties related to the existence of pushouts in different setups in the Appendix of this article for completeness.

 In Sections~\ref{sec:cylconst} we introduce the main construction and its dual, in which we go through the details of expressing many different well-known graph constructions as a cylindrical product or its dual. In this section we have tried to avoid being formal\footnote{Note that from a categorical point of view the natural framework to describe such a construction is the
category of cospans, however, we have focused on the current literature and nomenclature of graph theory to avoid unnecessary
complexities.}
and have described the constructions in detail from the scratch using a variant of examples for clarification.

 In Sections~\ref{sec:main} and \ref{sec:cat} we state the main duality result and some categorical perspective showing the naturality of definitions.
There are related results to reductions for the graph homomorphism problem that will be considered in Section~\ref{sec:adjred}.
Other applications of this construction will appear elsewhere.

\section{Basic definitions}\label{sec:basics}
In this section we go through some basic definitions and will set the notations for what will appear in the rest of the article.
Throughout the paper, {\it sets} are denoted by capital Roman characters as $X,Y,\ldots$ while {\it graphs} are denoted by capital RightRoman ones as $\matr{G}$ or $\matr{H}$. Ordered lists are denoted as $(x_{_{1}},x_{_{2}},\ldots,x_{_{t}})$, and in a concise form by Bold small characters as $\bx$. The notation $|.|$ is used for the size of a set as in $|X|$ or for the size of a list as in $|\bx|$. Also, the greatest common divisor of two integers $n$ and $k$ is denoted by $(n,k)$.

 We use $\nat[k]$ to refer to the set $\{1,2,\ldots,k\}$, and the symbol
${\bf S_{_{k}}}$ is used to refer to the group of all $k$ permutations of this set.
The symbol $\ident$ stands for the identity of the group ${\bf S_{_{k}}}$, and the
symbol ${\Perm[{a_{_{1}},a_{_{2}},\ldots,a_{_{t}}}]}$ stands for the cycle $(a_{_{1}},a_{_{2}},\ldots,a_{_{t}})$ in this group. Also, $\Ident \simeq {\bf S_{_{1}}} \leq {\bf S_{_{k}}}$ is the trivial subgroup containing only the identity element.

 For the ordered list
$\bx=(x_{_{1}},x_{_{2}},\ldots,x_{_{k}})$ and $\gamma \in {\bf S_{_{k}}}$, the right action of $\gamma$ on indices of $\bx$ is denoted as
$\bx \gamma \isdef (x_{_{\gamma(1)}},x_{_{\gamma(2)}},\ldots,x_{_{\gamma(k)}})$, where for any function $f$,
$$f\bx \isdef (f(x_{_{1}}),f(x_{_{2}}),\ldots,f(x_{_{k}})).$$
We use the notation $\{\bx \}$
for the set $\{x_{_{1}},x_{_{2}},\ldots,x_{_{k}}\}$.
\subsection{Graphs and their categories} \label{subsec:grph}
In this paper, a graph $\matr{G}=(V(\matr{G}),E(\matr{G}),\iota_{_{\matr{G}}}: E(\matr{G}) \longrightarrow V(\matr{G}),\tau_{_{\matr{G}}}: E(\matr{G}) \longrightarrow V(\matr{G}))$
consists of a {\it finite} set of {\it vertices} $V(\matr{G})$, a {\it finite} set of {\it edges} $E(\matr{G})$ and two maps $\iota_{_{\matr{G}}}$ and $\tau_{_{\matr{G}}}$
called {\it initial} and {\it terminal} maps, respectively. Hence, our graphs are finite but are generally directed, and may contain loops or multiple edges. For any edge $e \in E(\matr{G})$, the vertex $u \isdef \iota_{_{\matr{G}}}(e)$ is called the {\it initial vertex of $e$} and may be denoted by $e^{-}$. Similarly,
the vertex $v \isdef \tau_{_{\matr{G}}}(e)$ is called the {\it terminal vertex of $e$} and may be denoted by $e^{+}$. In this setting, when there is no ambiguity, the edge $e$ is sometimes referred to by the notation $uv$. Also, we say that an edge $e$ intersects a vertex $u$ if $u=e^-$ or $u=e^+$. Hereafter, we may freely refer to an edge $uv$ when the corresponding data is clear from the context.

 The {\it reduced part} of a graph $\matr{G}$, denoted by ${\rm red}(\matr{G})$, is the graph obtained by excluding all isolated vertices of $\matr{G}$. A graph is said to be {\it reduced} if it is isomorphic to its reduced part.
Given a graph $\matr{G}$ and a subset $X \subseteq V(\matr{G})$, the (vertex) {\it induced subgraph} on the subset
$X$ is denoted by $\matr{G}[X]$.

 A (graph) {\it homomorphism} $(\sigma_{_{V}},\sigma_{_{E}})$ from a graph $\matr{G}=(V(\matr{G}),E(\matr{G}),\iota_{_{\matr{G}}},\tau_{_{\matr{G}}})$ to a graph $\matr{H}=(V(\matr{H}),E(\matr{H}),\iota_{_{\matr{H}}},\tau_{_{\matr{H}}})$
is a pair of maps
$$\sigma_{_{V}}: V(\matr{G}) \longrightarrow V(\matr{H}) \quad {\rm and} \quad \sigma_{_{E}}: E(\matr{G}) \longrightarrow E(\matr{H})$$
which are compatible with the structure
maps (see Figure~\ref{fig:graphhom}), i.e.,
\begin{equation}\label{HOMCONDITION}
\iota_{_{\matr{H}}} \circ \sigma_{_{E}} = \sigma_{_{V}} \circ \iota_{_{\matr{G}}} \quad {\rm and} \quad
\tau_{_{\matr{H}}} \circ \sigma_{_{E}} = \sigma_{_{V}} \circ \tau_{_{\matr{G}}}.
\end{equation}
Note that when there is no multiple edge, one may think of $E(\matr{G})$ as a subset of $V(\matr{G}) \times V(\matr{G})$ and one may talk about the homomorphism $\sigma_{_{V}}$ when the compatibility condition is equivalent to the following,
\begin{equation}
uv \in E(\matr{G}) \quad \Rightarrow \quad \sigma(u)\sigma(v) \in E(\matr{H}).
\end{equation}\label{eq:hom}
The set of homomorphisms from the graph $\matr{G}$ to the graph $\matr{H}$ is denoted by $\hom(\matr{G},\matr{H})$,
and ${\bf Grph}$ stands for the category of (directed) graphs and their homomorphisms. Sometimes we may write $\matr{G}\rightarrow\matr{H}$ for
$\hom(\matr{G},\matr{H}) \neq \emptyset$.

 The partial ordering on the set of all graphs for which $\matr{G} \leq \matr{H}$ if and only if $\hom(\matr{G},\matr{H}) \neq \emptyset$, is denoted by ${\bf Grph}_{_{\leq}}$. Two graphs $\matr{G}$ and $\matr{H}$ are said to be {\it homomorphically equivalent}, denoted by $\matr{G} \homeq \matr{H}$, if both sets $\hom(\matr{G},\matr{H})$ and
$\hom(\matr{H},\matr{G})$ are non-empty i.e., if $\matr{G} \leq \matr{H}$ and $\matr{H} \leq \matr{G}$.
The interested reader is encouraged to refer to \cite{HENE} for more on graphs and their homomorphisms.
In this paper we will be working in different categories of graphs and their homomorphisms, that will be defined in the sequel.
\begin{figure}[ht]
\centering{
\begin{tikzpicture}[thick,scale=0.5]
\tikzstyle{every node}=[ inner sep=0pt, minimum width=5pt]
\node (A)[ ]{$E(\matr{G})$};
\node (B)[node distance=2.5 cm,right of=A] {$V(\matr{G})$};
\node (C)[node distance=2.5 cm,below of=A] {$E(\matr{H})$};
\node (D)[node distance=2.5 cm, below of=B] {$V(\matr{H})$};
\draw[->, thick] (A) to node [auto]{$\iota_{G_{}}$} (B);
\draw[->, thick] (A) to node [auto]{$\sigma_E$} (C);
\draw[->, thick] (C) to node [above]{$\iota_{_{\matr{H}_{}}}$} (D);
\draw[->, thick] (B) to node [auto]{$\sigma_V$} (D);
\end{tikzpicture} \qquad \qquad \begin{tikzpicture}[thick,scale=0.5]
\tikzstyle{every node}=[ inner sep=0pt, minimum width=5pt]
\node (A)[ ]{$E(\matr{G})$};
\node (B)[node distance=2.5 cm,right of=A] {$V(\matr{G})$};
\node (C)[node distance=2.5 cm,below of=A] {$E(\matr{H})$};
\node (D)[node distance=2.5 cm, below of=B] {$V(\matr{H})$};
\draw[->, thick] (A) to node [auto]{$\tau_{_{\matr{G}}}$} (B);
\draw[->, thick] (A) to node [auto]{$\sigma_E$} (C);
\draw[->, thick] (C) to node [above]{$\tau_{_{\matr{H}_{}}}$} (D);
\draw[->, thick] (B) to node [auto]{$\sigma_V$} (D);
\end{tikzpicture}
}
\caption{Commutative diagrams for a graph homomorphism.}
\label{fig:graphhom}
\end{figure}
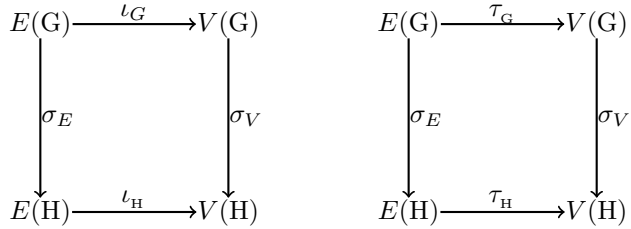
\begin{defin}{{\bf Marked graphs}\\
Let $X=\{x_{_{1}},x_{_{2}},\ldots,x_{_{k}}\}$ and $\matr{G}$ be a
set and a graph, respectively, and also, consider a one-to-one map
$\varrho: X \hookrightarrow V(\matr{G})$. Evidently, one can
consider $\varrho$ as a graph monomorphism from the empty graph
$\matr{X}$ on the vertex set $X$ to the graph $\matr{G}$, where in
this setting we interpret the situation as {\it marking} some
vertices of $\matr{G}$ by the elements of $X$. The data introduced
by the pair $(\matr{G},\varrho)$ is called a {\it marked graph} $\matr{G}$
marked by the set $X$ through the map $\varrho$. Note that (by abuse
of language) we may introduce the corresponding marked graph as
$\matr{G}(x_{_{1}},x_{_{2}},\ldots,x_{_{k}})$ when the definition of
$\varrho$ (especially its range) is clear from the context (for examples see Figures~\ref{fig:amalgam}$(a)$ and ~\ref{fig:amalgam}$(b)$).
Also, (by abuse of language) we may refer to {\it the vertex $x_{_{i}}$}
as the vertex $\varrho(x_{_{i}}) \in V(\matr{G})$. This is most
natural when $X \subseteq V(\matr{G})$ and vertices in $X$ are marked by
the corresponding elements in $V(\matr{G})$ through the identity mapping.
}\end{defin}
If $\varsigma: X \longrightarrow Y$ is a (not necessarily
one-to-one) map, then one can obtain a new marked graph
$(\matr{H},\tau: Y \longrightarrow V(\matr{H}))$ by considering the pushout of the diagram
\begin{equation}\label{diag:amalgam}
\matr{Y} \stackrel {\varsigma}{\longleftarrow} \matr{X}
\stackrel {\varrho}{\longrightarrow} \matr{G}
\end{equation}
in the category of
graphs. It is well-known that in \textbf{Grph} the pushout exists and is a
monomorphism (e.g. Figure~\ref{fig:amalgam}$(c)$ presents such a pushout (i.e. amalgam) of
marked graphs \ref{fig:amalgam}$(a)$ and \ref{fig:amalgam}$(b)$). Also, it is easy to see that the new marked graph
$(\matr{H},\tau)$ can be obtained from $(\matr{G},\varrho)$ by
identifying the vertices in each inverse-image of $\varsigma$.
Hence, again we may denote
$(\matr{H},\tau)$ as
$\matr{G}(\varsigma(x_{_{1}}),\varsigma(x_{_{2}}),\ldots,\varsigma(x_{_{k}}))$,
where we allow repetition in the list appearing in the brackets.
Note that, with this notation, one may interpret $x_{_{i}}$'s as a set
of {\it variables} in the {\it graph structure}
$\matr{G}(x_{_{1}},x_{_{2}},\ldots,x_{_{k}})$, such that when one
assigns other (new and not necessarily distinct) {\it values} to
these variables one can obtain
some other graphs (by identification of vertices) (e.g. see Figure~\ref{fig:amalgam}$(d)$).
\begin{figure}[ht]
\centering{\includegraphics[width=12cm]{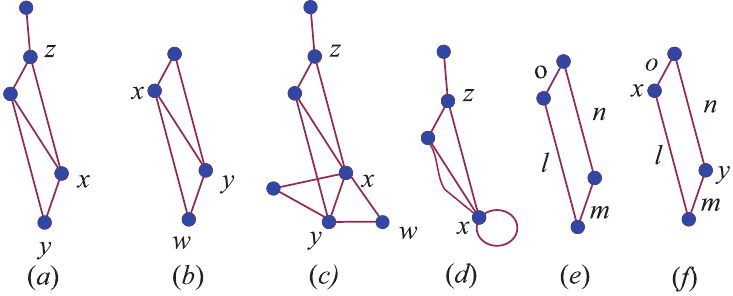}}
\caption{$(a)$ A marked graph $\matr{G}(x,y,z)$, $(b)$ A marked graph $\matr{H}(x,y,w)$, $(c)$ The amalgamation
$\matr{G}(x,y,z)+\matr{H}(x,y,w)$, $(d)$ The marked graph $\matr{G}(x,x,z)$, $(e)$ A labeled graph, $(f)$ A labeled marked graph.}
\label{fig:amalgam}
\end{figure}

 On the other hand, given two marked graphs $(\matr{G},\varrho)$
and $(\matr{H},\tau)$ with
$$X=\{x_{_{1}},x_{_{2}},\ldots,x_{_{k}}\} \quad {\rm and} \quad
Y=\{y_{_{1}},y_{_{2}},\ldots,y_{_{l}}\},$$
one can construct their
amalgam $$(\matr{G},\varrho) +_{_{\matr{S}}} (\matr{H},\tau)$$ by forming the
pushout of the following diagram,
$$\matr{H} \stackrel {\tilde{\tau}}{\longleftarrow} \matr{S}
\stackrel {\tilde{\varrho}}{\longrightarrow} \matr{G},$$ in which
$\matr{S}$ is a graph structure induced on $ X \cap Y$
indicating the way edges should be identified in this amalgam such that
one can define the extensions
$\tilde{\tau} \isdef \tau|_{_{X \cap Y}}$ and $\tilde{\varrho}
\isdef \varrho|_{_{X \cap Y}}$ as graph homomorphisms (note that pushouts exist in the category of graphs and can be constructed
as pushouts of the vertex set and the edge set naturally. For more details see the Appendix).
Following our previous notations we
may denote the new structure by
$$\matr{G}(x_{_{1}},x_{_{2}},\ldots,x_{_{k}})+\matr{H}(y_{_{1}},y_{_{2}},\ldots,y_{_{l}}).$$
If there is no confusion about the definition of mappings (in what follows the maps $\tilde{\tau}$ and $\tilde{\varrho}$ are embeddings of a
subgraph on $X \cap Y$). Note that
when $X \cap Y$ is the empty set, then the amalgam is
the {\it disjoint union} of the two marked graphs. Also, by the
universal property of the pushout diagram, the amalgam can be
considered as marked graphs marked by $X$, $Y$,
$X \cup Y$ or $X \cap Y$.

 Sometimes it is preferred to partition the list of variables in a
graph structure as,
$$\matr{G}(x_{_{1}},x_{_{2}},\ldots,x_{_{k}},
y_{_{1}},y_{_{2}},\ldots,y_{_{l}},
z_{_{1}},z_{_{2}},\ldots,z_{_{m}}).$$ In these cases we may either
use bold symbols for an ordered list of variables and write this
graph structure as $\matr{G}({\bf x};{\bf y};{\bf z})$ (if there is
no confusion about the size of the lists), or we may simply write
$$\matr{G}(x_{_{1}},x_{_{2}},\ldots,x_{_{k}};
y_{_{1}},y_{_{2}},\ldots,y_{_{l}};
z_{_{1}},z_{_{2}},\ldots,z_{_{m}}).$$ It is understood that a
repeated appearance of a graph structure in an expression as
$\matr{G}(v)+\matr{G}(v,w)$ is always considered as different
isomorphic copies of the structure marked properly by the indicated
labels (e.g. $\matr{G}(v)+\matr{G}(v,w)$ is an amalgam constructed
by two different isomorphic copies of $G$ identified on the vertex
$v$ where the
vertex $w$ in one of these copies is marked).

 By $\matr{K}_{_{k}}(v_{_1},v_{_2},\ldots,v_{_k})$ we mean a $k$-clique on
$\{v_{_1},v_{_2},\ldots,v_{_k}\}$ marked by its own set of vertices.
Specially, a single edge is denoted by $\matr{\varepsilon}(v_{_1},v_{_2})$
(i.e., $\matr{\varepsilon}(v_{_1},v_{_2}) \isdef \matr{K}_{_{2}}(v_{_1},v_{_2})$).
We use the notation $\matr{G}[( \bx )]$ for the induced subgraph on vertices of $\matr{G}$ labeled by $\bx$.
\begin{defin}{{\bf Labeled graphs}\\
An $\ell_{_{\matr{G}}}$-graph $\matr{G}$ labeled by the set $L_{_{\matr{G}}}$ consists of the data
$(\matr{G},\ell_{_{\matr{G}}}: E(\matr{G}) \longrightarrow L_{_{\matr{G}}})$, where $\matr{G}$ is a graph and $\ell_{_{\matr{G}}}$ is the labeling map (e.g. see Figure~\ref{fig:amalgam}$(e)$).
A (labeled graph) {\it homomorphism} $(\sigma_{_{V}},\sigma_{_{E}},\sigma_{_{\ell}})$ from
an $\ell_{_{\matr{G}}}$-graph $\matr{G}=(V(\matr{G}),E(\matr{G}),\ell_{_{\matr{G}}})$ to
an $\ell_{_{\matr{H}}}$-graph $\matr{H}=(V(\matr{H}),E(\matr{H}),\ell_{_{\matr{H}}})$
is a triple
$$\sigma_{_{V}}: V(\matr{G}) \longrightarrow V(\matr{H}), \quad \sigma_{_{E}}: E(\matr{G}) \longrightarrow E(\matr{H})
\quad {\rm and} \quad \sigma_{_{\ell}}: L_{_{\matr{G}}} \longrightarrow L_{_{\matr{H}}},$$
where all maps are compatible with the structure maps, i.e., conditions of Equation~\ref{HOMCONDITION} are satisfied as well as,
\begin{equation}\label{LHOMCONDITION}
\ell_{_{\matr{H}}} \circ \sigma_{_{E}} = \sigma_{_{\ell}} \circ \ell_{_{\matr{G}}}.
\end{equation}
The set of homomorphisms from the $\ell_{_{\matr{G}}}$-graph $\matr{G}$ to the $\ell_{_{\matr{H}}}$-graph $\matr{H}$ is denoted by\break $\lhom(\matr{G},\matr{H})$, and ${\bf LGrph}$ stands for the category of labeled graphs and their homomorphisms.
}\end{defin}

 It is obvious that the space of
all graphs can be identified with the class of all labeled graphs,
labeled by the one element set $\{1\}$. Hence, we treat all graphs labeled by a
set of just one element, as {\it ordinary} (i.e. not labeled) graphs.
Also, it is instructive to note that the concept of a labeled graph and the category ${\bf LGrph}$ has a central importance
in the whole world of combinatorial objects (e.g. see \cite{LAWI00}).

 Accordingly, a {\it labeled marked graph} $(\matr{G},\varrho,\ell)$, in which vertices are marked by the set $X$ and the edges are labeled by the set $L$, consists of the data $$(\matr{G},\ \varrho: X \hookrightarrow V(\matr{G}),\ \ell: E(\matr{G}) \longrightarrow L),$$
where $(\matr{G},\varrho)$ is a marked $\ell$-graph equipped with the labeling $\ell: E(\matr{G}) \longrightarrow L$ of its edges (e.g. see Figure~\ref{fig:amalgam}$(f)$).
Again, given two labeled marked graphs
$$(\matr{G},\varrho_{_{\matr{G}}}: X \hookrightarrow V(\matr{G}),\ell_{_{\matr{G}}}: E(\matr{G}) \longrightarrow L_{_{\matr{G}}}) \ {\rm and} \ (\matr{H},\varrho_{_{\matr{H}}}: Y \hookrightarrow V(\matr{H}),\ell_{_{\matr{H}}}: E(\matr{H}) \longrightarrow L_{_{\matr{H}}}),$$
one may consider these graphs as $L_{_{\matr{G}}} \cup L_{_{\matr{H}}}$-labeled marked graphs and consider the pushout of Diagram~\ref{diag:amalgam} in ${\bf LGrph}$, which exists when, intuitively,
the $L_{_{\matr{G}}} \cup L_{_{\matr{H}}}$-labeled marked graph obtained by identifying vertices in $X \cap Y$ can be constructed
by assigning the common labels of the overlapping edges in the intersection (i.e. the labels of overlapping edges must be compatible).
This graph amalgam will be denoted in a concise form (when there is no confusion) by
$$\matr{G}(\varrho_{_{\matr{G}}}(X))+\matr{H}(\varrho_{_{\matr{H}}}(Y)).$$
\begin{defin}{\label{def:symgraph}
A labeled graph $\matr{G}=(V,E,\iota_{_{\matr{G}}},\tau_{_{\matr{G}}},\ell_{_{\matr{G}}})$ is said to be {\it symmetric}
if there is a one to one correspondence between the labeled edges $uv$ and $vu$ that preserves the labeling (see Figure \ref{fig:simple}$(a)$).
If for any two vertices $u$ and $v$ there only exist two edges $uv$ and $vu$ with the same label then the graph is said to be {\it simply symmetric} (see Figure \ref{fig:simple}$(b)$). Clearly, the data available in the structure of a simply symmetric graph can be encoded in a structure on the same set of vertices for which an edge is a subset of $V$ of size two (note that for loops we correspond two directed loops to a simple loop), while in this case the new structure may be called the corresponding {\it generalized simple} graph (see Figure \ref{fig:simple}$(c)$).
\begin{figure}[ht]
\centering{\includegraphics[width=8cm]{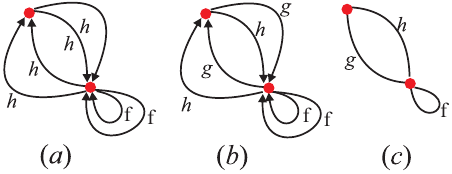}}
\caption{$(a)$ A symmetric graph, $(b)$ A simply symmetric graph, $(c)$ The corresponding simple graph.}
\label{fig:simple}
\end{figure}
Since usually in graph theory {\it simple} graphs do not have loops, we just freely use the word {\it symmetric graph} to refer to the directed or the generalized simple graph itself. Hereafter, ${\bf SymGrph}$ stands for the category of symmetric graphs
as a subcategory of ${\bf LGrph}$ where we may also denote a pair of dual directed edges by a simple edge as it is usual in graph theory.
}\end{defin}
\subsubsection{Category of cylinders and category of $(\Gamma,m)$-graphs}
\begin{defin}{\label{defin:cylinder}{\bf Cylinders}\\
A {\it cylinder} $\matr{C}(\by,\bz,\J)$ of thickness $(t,k)$ (or a
$(t,k)$-cylinder for short), with the {\it initial base} $\matr{B}^-$ and
the {\it terminal base} $\matr{B}^+$ (see Figure~\ref{fig:cylinder-twist}$(a)$), is a labeled marked graph with the following data,
\begin{itemize}
\item{A labeled marked graph $\matr{C}(\by,\bz)$.}
\item{An ordered list $\by=(y_{_{1}},y_{_{2}},\ldots,y_{_{k}})$
with $0 < k$ such that
$$\{y_{_{1}},y_{_{2}},\ldots,y_{_{k}}\} \subseteq V(\matr{C}), \quad
{\rm and}
\quad \matr{B}^- \isdef \matr{C}[y_{_{1}},y_{_{2}},\ldots,y_{_{k}}].$$}
\item{An ordered list $\bz=(z_{_{1}},z_{_{2}},\ldots,z_{_{k}})$
with $0 < k$ such that
$$\{z_{_{1}},z_{_{2}},\ldots,z_{_{k}}\} \subseteq V(\matr{C}), \quad {\rm and}
\quad \matr{B}^+ \isdef \matr{C}[z_{_{1}},z_{_{2}},\ldots,z_{_{k}}].$$}
\item A one to one partial function $\J: \nat[k] \rightarrow \nat[k]$, with
$$|{\rm domain}(\J)|=|{\rm range}(\J)|=t,$$ indicating the equality of $t$ elements of $\by$ and $\bz$, in the sense that
$$\J(i)=j \Leftrightarrow y _i =z _j.$$
\item{The mapping $y_{_{i}} \mapsto z_{_{i}} \ (1 \leq i \leq k)$ determines an isomorphism of the bases $\matr{B}^-$ and $\matr{B}^+$ as labeled graphs}.
\end{itemize}
Hereafter, we always refer to $\J$ as a relation, i.e. $\J=\{(i,\J(i)) \ : \ i \in domain(\J) \}$.
Also, note that in general a cylinder is a directed graph and also has a {\it direction} itself (say from $\matr{B}^-$ to $\matr{B}^+$). A cylinder $\matr{C}(\by,\bz,\J)$ is said to be {\it symmetric} if there exists an automorphism of $\matr{C}$ as $(\sigma_{_V}, \sigma_{_E}): \matr{C} \longrightarrow \matr{C}$ that maps $\matr{B}^-$ to $\matr{B}^+$ isomorphically as marked graphs, such that,
\begin{itemize}
\item {$\forall i\in \nat[k] \quad \sigma(y_i)=z_i, \ \sigma(z_i)=y_i$,}
\item {$\sigma (\matr{B}^-)=\matr{B}^+ ,\quad \sigma (\matr{B}^+)=\matr{B}^-$ (as labeled graphs),}
\item{$\J \subseteq\{(j,j) \ : \ j\in \nat[k]\}$.}
\end{itemize}
}\end{defin}
\begin{figure}[ht]
\centering{\includegraphics[width=10cm]{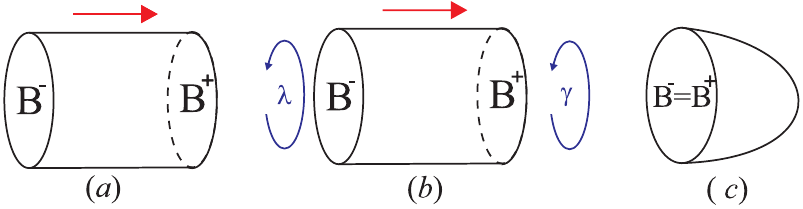}}
\caption{$(a)$ A general form of a cylinder, $(b)$ A cylinder with a twist $(\lambda , \gamma)$, $(c)$ A generalized loop cylinder.}
\label{fig:cylinder-twist}
\end{figure}
\noindent {\bf Notation.} In what follows, we may exclude variables in our notations when they are set to their
{\it default} (i.e. trivial) values. As an example,
note that $\matr{C}(\by,\bz)$ is used when $\J $ is empty (i.e. $t=0$) and we may refer to it as a $k$-cylinder.
Also, we partition the vertex set of a $(t,k)$ cylinder into the \textit{base vertices} that appear in $V(\matr{B}^-) \cup V(\matr{B}^+)$, and the \textit{inner vertices} that do not appear on either base.
A cylinder whose bases are empty graphs and has no inner vertex is called a {\it plain} cylinder.
\begin{figure}[ht]
\begin{center}
\begin{tabular}{c c}
\includegraphics[width=4cm]{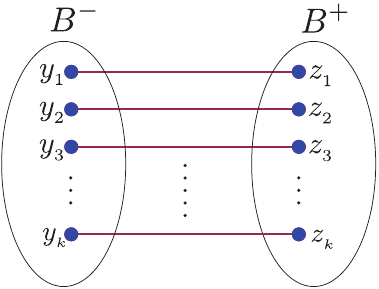} &
\includegraphics[width=4cm]{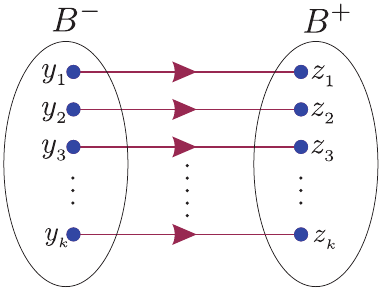}\\
\vspace{-0.5cm}\cr {$(a)$} & {$(b)$}
\end{tabular}
\end{center}
\caption{$(a)$ The identity cylinder $\matr{I}_{_{k}}(\by,\bz)$ (see Examples~\ref{exm:basiccylinders} and \ref{exm:petersen}), $(b)$ The directed identity cylinder $\vec{\matr{I}}_{_{k}}(\by,\bz)$.}
\label{fig:voltage_cylinder}
\end{figure}

 \begin{exm}{\label{exm:basiccylinders}
In this example we illustrate a couple of cylinders to clarify the definition.
\begin{itemize}
\item{{\bf The identity cylinder, $\matr{I}_{_{k}}(\by,\bz)$}

 Define the {\it symmetric identity cylinder} on $2k$ vertices as
$\matr{I}_{_{k}}(\by,\bz)$, where both bases are empty graphs on $k$ vertices and
for all $i \in \nat[k]$ the only edges are
$y_{_{i}}z_{_{i}}$ (see Figure~\ref{fig:voltage_cylinder}(a)).
Clearly, we may define $\matr{I}(y,z) \isdef \matr{I}_{_{1}}(y,z)$ standing for an edge.
Also, we use the notation $\vec{\matr{I}}_{_{k}}(\by,\bz)$ for the directed version, where edges are directed from $\matr{B}^-$ to $\matr{B}^+$(see Figure~\ref{fig:voltage_cylinder}(b)).
When the definition is clear from the context we use the same notation for both directed and symmetric forms of identity cylinder.
}
\item {\bf The $\sqcap$-cylinder}, ${\sqcap}(\by,\bz)$

 This is a $(0,2)$-cylinder (or a $2$-cylinder for short) that is depicted in Figure~\ref{fig:cylin}(a), for which
$\J = \emptyset$. This cylinder will be used in the sequel to construct the Petersen and generalized Petersen graphs (see Example~\ref{exm:petersen}).

 \begin{figure}[ht]
\begin{center}
\begin{tabular}{c c c c}
\includegraphics[width=3cm]{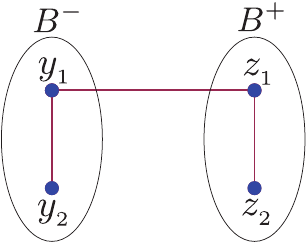} &
\includegraphics[width=3cm]{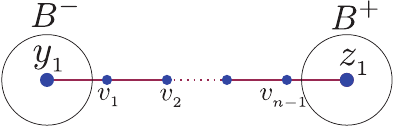} &
\includegraphics[width=3cm]{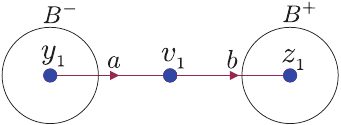}&
\includegraphics[width=3cm]{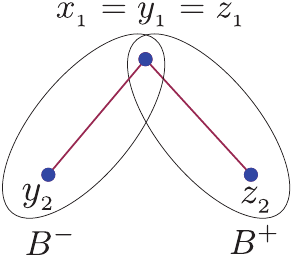} \\
\vspace{-0.5cm}\cr {$(a)$} & {$(b)$} & {$(c)$} & {$(d)$}
\end{tabular}
\end{center}
\caption{$(a)$ The $\sqcap$-cylinder, $(b)$ The path cylinder $\matr{P}_{n}$, $(c)$ The directed path cylinder
$\vec{\matr{P}}_{2}$ of length two, $(d)$ The looped line-graph cylinder.}
\label{fig:cylin}
\end{figure}

 \item {\bf The path cylinder}, $\matr{P}_{{n}}(y,z)$

 For $n > 0$ this is $1$-cylinder with $n+1$ vertices that is depicted in Figure~\ref{fig:cylin}(b), forming a path of
length $n$. This cylinder will be used in the sequel to construct subdivision and
fractional powers of a graph (see Examples~\ref{exm:exp} and \ref{exm:cylop}).
The cylinder $\matr{P}_{_0}$ is a $(1,1)$ cylinder with just one vertex
$y=z$ and no edge for which both bases are identified through $\J=\{(1,1)\}$ and have just one element. This cylinder will be used for contraction operation in graphs (see Example~\ref{exm:cylop}).
Analogously, one may define directed path cylinders $\vec{\matr{P}}_{{n}}(y,z)$ (see Figure~\ref{fig:cylin}(c)).

 \item {\bf The looped line-graph cylinder}, ${\rm \wedge}(\by,\bz)$

 This is a $(1,2)$-cylinder on $3$ vertices, depicted in Figure~\ref{fig:cylin}(d), for which
$\J = \{ (1,1) \}$, i.e. $y_{_{1}} = z_{_{1}}$.
This cylinder will be used in the sequel to construct looped line-graphs (see Example~\ref{exm:exp}).

 \item {\bf Generalized loop cylinder},

 This is a $(k,k)$-cylinder with $\J=\{(1,1),(2,2), ... ,(k,k)\}$. Any graph can be considered as a generalized loop cylinder by fixing an arbitrary subgraph of it with $k$ vertices as $\matr{B}^-=\matr{B}^+$ (see Figure~\ref{fig:cylinder-twist}$(c)$).

 \item{{\bf The deletion cylinder, $d_k(\by, \bz)$}

 This is a $(0,k)$-cylinder with $2k$ vertices and without any edges. This cylinder will be used for deletion operation in graphs. We define $d(y,z) \isdef d_{_{1}}(y,z)$.
}
\item{ {\bf The contraction cylinder},

 This is a particular case of generalized loop cylinder that will be used for the contraction operation in graphs. The contraction cylinder is an empty $(k,k)$ cylinder without any inner vertex. For example $\matr{P}_{_0}$ is a
contraction cylinder.
}
\item{ {\bf Fiber gadget} \cite{fiber,fiber1},

 A graph $\matr{M}$ containing two copies $W^{^a}$ and $W^{^b}$ of an indexed set $W^{^*}$, is an \textit{fiber gadget}. If we set $\matr{C}=\matr{M}$, $B^-=W^{^a}$ and $B^+=W^{^b}$, we can see that the fibre gadget is a cylinder, where we denote it by $\matr{C}(\matr{M})$.
}
\item{ {\bf Pultr templates} \cite{pultr,adjoint, ladjoint},

 A \textit{Pultr template} is a quadruple $\tau=(\matr{P}, \matr{Q}, \eta_{_1}, \eta_{_2})$, where $\matr{P}, \matr{Q}$
are digraphs and $\eta_{_1}, \eta_{_2}: \matr{P} \rightarrow \matr{Q}$ are graph homomorphisms and $Q$ admits an
automorphism $\zeta$ such that \linebreak $\zeta \circ \eta_{_{1}}=\eta_{_{2}}$ and $\zeta \circ \eta_{_{2}}=\eta_{_{1}}$. Note that if
$\eta_{_1}, \eta_{_2}$ are one to one graph embeddings (i.e. $\eta_{_1}(\matr{P}) \simeq \eta_{_2}(\matr{P})
\simeq \matr{P}$), then $\matr{Q}$ has the structure of a $(t,k)$-cylinder with $k=|\matr{P}|$ while
$\J$ and $t$ are determined by the overlaps in the range of the maps $\eta_{_1}, \eta_{_2}$.

 } \end{itemize}
}\end{exm}

 \noindent {\bf Notation.}
Given two $(t,k)$-cylinders $\matr{C}(\by,\bz)$ and $\matr{D}(\by,\bz)$, the $(t,k)$-cylinder $\matr{C}(\by,\bz)+\matr{D}(\by,\bz)$
is the amalgam obtained by identifying bases. Also, the $(2t,2k)$-cylinder $\matr{C} \Cup \matr{D} $ is the cylinder obtained by naturally
taking the disjoint union of these structures (e.g. Figure~\ref{fig:+-}).

 \begin{figure}[ht]
\centering{\includegraphics[width=8cm]{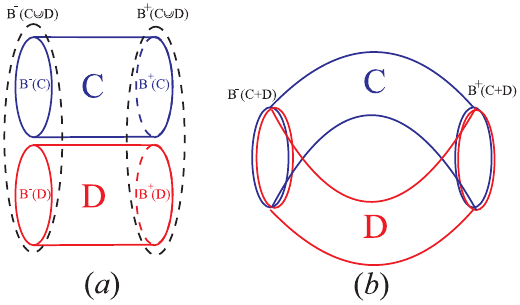}}
\caption{$(a)$ The cylinder $\matr{C} \Cup \matr{D}$, $(b)$ The cylinder $\matr{C}+\matr{D}$.}
\label{fig:+-}
\end{figure}
\begin{defin}{\label{def:twist}{\bf Twist of a cylinder}

 If $|\by| =k$, then $(\gamma ,\lambda) \in {\bf S_{_{k}} \times \bf S_{_{k}}}$ is said to be a {\it twist} of a cylinder $\matr{C}(\by,\bz, \J)$, if $\gamma$ and $\lambda $ are \textit{graph} automorphisms of bases $\matr{B}^-$ or $\matr{B}^+$ (see Figure \ref{fig:cylinder-twist}$(b)$ and note that the bases $\matr{B}^-$ and $\matr{B}^+$ are isomorphic by definition).

 A {\it twisted cylinder} $\matr{C}(\by ',\bz ', (\gamma, \lambda)\J)$, as the result of action of a twist
$(\gamma,\lambda)$ on $\matr{C}(\by,\bz, \J)$, is a cylinder for which we have
$\by '=\by \gamma$, $\bz '=\bz \lambda$ and $(\gamma, \lambda)\J \isdef\{(\gamma(i), \lambda(j)) \ : \ (i,j) \in \J \}$.

 Note that by definition, a twist is an element of $Aut(\matr{B}^-)\times Aut(\matr{B}^-)$ and vice versa.
The group
$Aut(\matr{B}^-)\times Aut(\matr{B}^-)$ is called the {\it twist group} of the cylinder $\matr{C}$.
}\end{defin}

 \noindent\textbf{Notation.} We say that a twist $(\gamma, \delta)$ is \symtw, if $\gamma=\delta$.
\begin{figure}[ht]
\begin{center}
\begin{tabular}{c c c c}
\includegraphics[width=3cm]{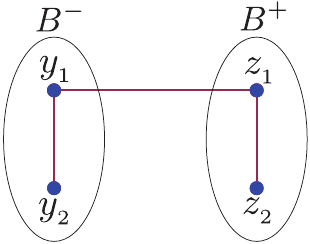} & \includegraphics[width=3cm]{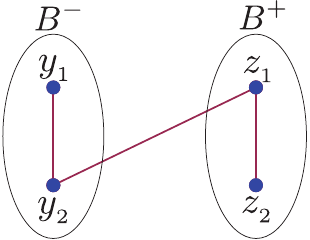} &\includegraphics[width=3cm]{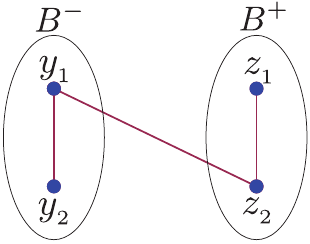} & \includegraphics[width=3cm]{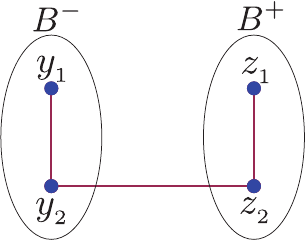} \\
\vspace{-0.5cm}\cr {$(a)$} & {$(b)$} & {$(c)$} & {$(d)$}
\end{tabular}
\end{center}
\caption{Twists of the $\sqcap$-cylinder (see Example~\ref{exm:twists}).}
\label{fig:twist}
\end{figure}

 \begin{exm}{\label{exm:twists}
\textbf{Twists of ${\sqcap}(\by,\bz)$}

 Setting $\pi_{_{2}}\isdef \Perm[{2, 1}] $, each element of ${\bf S_{_{2}} \times \bf S_{_{2}}} \simeq \{ (\ident,\ident), (\ident,\pi_{_{2}}), (\pi_{_{2}}, \ident), (\pi_{_{2}},\pi_{_{2}}) \}$ satisfies the conditions of
Definition~\ref{def:twist}, and consequently, is a twist of the $\sqcap$-cylinder. In what follows we present the action of twists of this type on ${\sqcap}(\by,\bz)$ as depicted in Figure~\ref{fig:cylin}(a).
\begin{itemize}
\item Figure~\ref{fig:twist}$(a)$: The twist $(\ident ,\ident)$ on $\sqcap$-cylinder,
\item Figure~\ref{fig:twist}$(b)$: The twist $(\pi_{_2} ,\ident)$ on $\sqcap$-cylinder,
\item Figure~\ref{fig:twist}$(c)$: The twist $(\ident ,\pi_{_2})$ on $\sqcap$-cylinder,
\item Figure~\ref{fig:twist}$(d)$: The twist $(\pi_{_2} ,\pi_{_2})$ on $\sqcap$-cylinder.
\end{itemize}
}\end{exm}
\begin{defin}{{\bf A $\Gamma$-coherent set of cylinders}

 Given integers $t$ and $k$ and a subgroup $\Gamma \leq {\bf S_{_{k}}}$,
then a set of $(t,k)$-cylinders
$$\matr{C} =\{\matr{C}^{^{j}}(\by^{^{j}},\bz^{^{j}},\J^{^{j}}) \ \ | \ j \in \nat[m]\}$$
is said to be {\it $\Gamma$-coherent}, if
\begin{itemize}
\item{For all $j,r \in \nat[m]$, the mapping $y^{^j}_i \mapsto y^{^r}_i \ (i \in \nat[k] )$ from $\matr{B}^-(\matr{C}^{^{j}})$ to $\matr{B}^-(\matr{C}^{^{r}})$ induces an isomorphism of these bases as labeled graphs,}
\item{ $\Gamma \leq Aut(\matr{B}^-(\matr{C}^{^1}))$.}
\end{itemize}
A $\Gamma$-coherent set of cylinders is said to be {\it symmetric} if each one of its cylinders is a symmetric cylinder.
}\end{defin}
\noindent\textbf{Notation.} When $m=1$, we use the notation to $\Gamma$-cylinder for $\matr{C}^{^{1}}(\by^{^{1}},\bz^{^{1}},\J^{^{1}})$.
\begin{figure}[ht]\centering{
\begin{tabular}{c c}
\includegraphics[width=4cm]{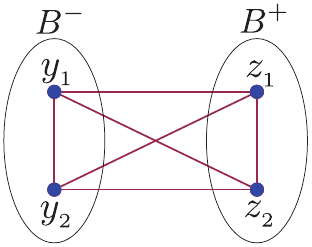} & \includegraphics[width=4cm]{./figures/peterson_cylinder.pdf}\\
\vspace{-0.5cm}\cr {$(a)$} & {$(b)$}
\end{tabular}}
\caption{A set of ${\bf S_{_{2}}}$-coherent cylinders.}
\label{fig:coherent}
\end{figure}

 \begin{exm}{\label{exm:gammacoherent}
In this example we present two coherent sets of cylinders.
\begin{itemize}
\item{
Every set of $1$-cylinders is ${\bf S_{_{1}}}$-coherent. For instance,
every set of path cylinders is ${\bf S_{_{1}}}$-coherent.}
\item{The $\sqcap$-cylinder and the cylinder shown in Figure~\ref{fig:coherent}$(a)$ form a
${\bf S_{_{2}}}$-coherent set of cylinders.}
\end{itemize}
}\end{exm}
\textbf{Remark:} If a set of $(t,k)$-cylinders is a $\textbf{S}_{_k}$-coherent set, then either all
bases are isomorphic to the complete graph $\matr{K}_{_{k}}$ or they are isomorphic to the empty graph $\overline{\matr{K}}_{_k}$.

 \begin{defin}{\label{def:gammacyl}{\bf The category of $(\Gamma,A)$-graphs }

 Given a subgroup $\Gamma \leq {\bf S_{_{k}}}$ and a finite set $A$, a $(\Gamma,A)$-graph is a
graph labeled by the set
$$L_{_{\Gamma,A}} \isdef \{(\gamma^-,a,\gamma^+) \quad | \quad a\in A \ \ ,\quad \gamma^-,\gamma^+\in \Gamma \}.$$
When $A$ is a set of $m$ elements, we usually assume that $A=\nat[m]$ unless it is stated otherwise, where in this case we
talk about $(\Gamma,m)$-graphs.

 We say that a $(\Gamma, m)$-graph has a \symtw \ labeling if all of its twists are \symtw.

 Consider an edge $e \in E(\matr{G})$ in a $(\Gamma,m)$-graph $\matr{G}$ with $e^-=u$ and $e^+=v$. Then if $\ell_{_{\matr{G}}}$ is the labeling map,
and $\ell_{_{\matr{G}}}(e)=(\gamma^-,i,\gamma^+)$, we use the following notations
\begin{equation}
\ell^{*}_{_{\matr{G}}}(e) \isdef i.
\end{equation}
\begin{equation}
\ell^{-}_{_{\matr{G}}}(e) \isdef \gamma^-, \quad \ell^{+}_{_{\matr{G}}}(e) \isdef \gamma^+.
\end{equation}
Also, if $w$ is any one of the two ends of $e$ (i.e. $w=u$ or $w=v$) and there is not any ambiguous, then we can use the notations
\begin{equation}
\ell^{u}_{_{\matr{G}}}(e) \isdef \ell^{-}_{_{\matr{G}}}(e) = \gamma^-, \quad \ell^{v}_{_{\matr{G}}}(e) \isdef \ell^{+}_{_{\matr{G}}}(e) = \gamma^+.
\end{equation}
Suppose that $e_{_1},e_{_2} \in E(\matr{G})$ are incident edges, where $e_{_1}^a=e_{_2}^b=w$ (e.g. $e_{_1}^-=e_{_2}^+$) for $a,b\in \{-,+\}$, then we define a {\it label difference function}
$\Delta_{_{\matr{G}}}^{ab}$ as
$$\Delta_{_{\matr{G}}}^{ab}(e_{_1},e_{_2}) \isdef [\ell^{a}_{_{\matr{G}}}(e_{_{1}})]^{-1}[\ell^{b}_{_{\matr{G}}}(e_{_{2}})].$$
For simple graphs, we can use the notation $\Delta_{_{\matr{G}}}^{w}(e_{_1},e_{_2})$. Of course when there is not any ambiguity we may omit the subscripts or superscripts.

 Let $\matr{G}$ and $\matr{H}$ be two $(\Gamma,m)$-graphs labeled by the maps $\ell_{_{\matr{G}}}$ and $\ell_{_{\matr{H}}}$, respectively.
Then a $(\Gamma,m)$-{\it homomorphism} from $\matr{G}$ to $\matr{H}$ is a graph homomorphism $(\sigma_{_{V}},\sigma_{_{E}}):\matr{G}\longrightarrow \matr{H}$ such that
\begin{equation}
\ell^{*}_{_{\matr{G}}}=\ell^{*}_{_{\matr{H}}} \circ \sigma_{_{E}},
\end{equation}
and for any pair of two edges $e_{_{1}},e_{_{2}} \in E(\matr{G})$ that intersect in $e_{_{1}}^a=e_{_{2}}^b$ (for $a,b\in \{-,+\}$),
$$\Delta_{_{\matr{G}}}^{ab}(e_{_1},e_{_2})=\Delta_{_{\matr{H}}}^{ab}(\sigma_E(e_{_1}),\sigma_E(e_{_2})),$$
or equivalently,
\begin{equation}
[\ell^{a}_{_{\matr{G}}}(e_{_{1}})]^{-1}[\ell^{b}_{_{\matr{G}}}(e_{_{2}})]=
[\ell^{a}_{_{\matr{H}}}
(\sigma_{_{E}}(e_{_{1}}))]^{-1}
[\ell^{b}_{_{\matr{H}}}(\sigma_{_{E}}(e_{_{2}}))].
\end{equation}
The set of homomorphisms from the $(\Gamma,m)$-graph $\matr{G}$ to the $(\Gamma,m)$-graph $\matr{H}$ is denoted by $\chom(\matr{G},\matr{H})$, and ${\bf LGrph}(\Gamma,m)$ stands for the category of $(\Gamma,m)$-graphs and their homomorphisms.
}\end{defin}
Note that when $m=1$ and $\Gamma=\Ident $ then ${\bf Grph} \simeq {\bf LGrph}({\bf S}_{_{1}},1)$.
\begin{prepro}{\label{CHOMLEMMA}
Let $\matr{G}$ and $\matr{H}$ be two $(\Gamma,m)$-graphs labeled by the maps $\ell_{_{\matr{G}}}$ and $\ell_{_{\matr{H}}}$, respectively.
Then for a labeled marked graph homomorphism $(\sigma_{_{V}},\sigma_{_{E}}) \in \chom(\matr{G},\matr{H})$ the following conditions are equivalent.
\begin{itemize}{
\item[{\rm a)}]{The pair $(\sigma_{_{V}},\sigma_{_{E}})$ is a $(\Gamma,m)$-homomorphism. In other words, for any pair of two edges $e_{_{1}},e_{_{2}} \in E(\matr{G})$ that intersect at a vertex $e_{_{1}}^a=e_{_{2}}^b$ (for $a,b\in \{-,+\}$),
$$\Delta_{_{\matr{G}}}^{ab}(e_{_1},e_{_2})=\Delta_{_{\matr{H}}}^{ab}(\sigma_E(e_{_1}),\sigma_E(e_{_2})).
$$
}
\item[{\rm b)}]{For any vertex $u \in V(\matr{G})$, there exists a unique constant $\alpha_{_{u}} \in \Gamma$ such that for any edge $e \in E(\matr{G})$
that intersects in $e_{_{1}}^a=u$ (for $a\in \{-,+\}$), we have,
$$\ell^{a}_{_{\matr{G}}}(e)=\alpha_{_{u}}[\ell^{a}_{_{\matr{H}}}(\sigma_{_{E}}(e))].$$
}
}\end{itemize}
}\end{prepro}

 \begin{proof}{(a $\Rightarrow$ b) Fix a vertex $u$ and an edge $e_{_{0}}$ that $e_{_{0}}^+=u$ (one can do the rest with the assumption of $e_{_{0}}^+=u$ in a similar way) and define,
$$\alpha_{_{u}} \isdef \ell^{+}_{_{\matr{G}}}(e_{_{0}})[\ell^{+}_{_{\matr{H}}}
(\sigma_{_{E}}(e_{_{0}}))]^{-1}.$$
Then for any other edge $e$ that $e^a=u$ for $a\in \{+,-\}$ we have,
$$
\begin{array}{ll}
\ell^{a}_{_{\matr{G}}}(e)&= \ell^{+}_{_{\matr{G}}}(e_{_{0}})
[\ell^{+}_{_{\matr{G}}}(e_{_{0}})]^{-1}\ell^{a}_{_{\matr{G}}}(e)\\
&\\
&=\ell^{+}_{_{\matr{G}}}(e_{_{0}})[\ell^{+}_{_{\matr{H}}}
(\sigma_{_{E}}(e_{_{0}}))]^{-1}[\ell^{a}_{_{\matr{H}}}
(\sigma_{_{E}}(e))]\\
&\\
&= \alpha_{_{u}}[\ell^{a}_{_{\matr{H}}}(\sigma_{_{E}}(e))].
\end{array}
$$
(b $\Rightarrow$ a) is clear by definition of the constant $ \alpha_{_{u}}$.
}\end{proof}

 \begin{defin}{
Consider a graph $\matr{G}$ labeled by the set $L_{_{\Gamma,m}} $ and the map $\ell_{_{\matr{G}}}$, along with a
vector $\alpha \in \Gamma^{|V(\matr{G})|}$. Then we define the labeled graph $\matr{G}_\alpha$ on the graph $\matr{G}$
with the labeling
$$\ell_{_{\matr{G}_{\alpha}}}(uv)=
(\alpha_{_{u}}\ell^{-}_{_{\matr{G}}}(uv),\ell^{*}_{_{\matr{G}}}(uv),\alpha_{_{v}}\ell^{+}_{_{\matr{G}}}(uv)).$$
}\end{defin}
\noindent {\bf Remark:} If $\matr{C} = \{\matr{C}^{^{j}}(\by^{^{j}},\by^{^{j}}, \J ^{^j}) \ | \ j \in \nat[m]\}$ is a $\Gamma$-coherent set of cylinders labeled by $L_{_{\Gamma,m}}$, then
$$\matr{C}_{_\alpha} \isdef \{ \matr{C}_{_\alpha}^{^{j}}(\by^{^{j}},\bz^{^{j}}, \J^{^j}) \ | \ j \in \nat[m]\}$$
is also a $\Gamma$-coherent set of cylinders.\\
The following results is a direct consequence of Proposition~\ref{CHOMLEMMA}.
\begin{cor}{\label{cor:alphashift}
Let $\matr{G}$ and $\matr{H}$ be two graphs labeled by the set $L_{_{\Gamma,m}}$ and the maps $\ell_{_{\matr{G}}}$ and $\ell_{_{\matr{H}}}$, respectively. Then, considering these graphs as objects of the category of labeled graphs ${\bf LGrph}$, and also as objects of the category of $(\Gamma,m)$-graphs ${\bf LGrph}(\Gamma,m)$,
$$\exists\ \alpha \in \Gamma^{|V(\matr{G})|} \quad \chom(\matr{G},\matr{H}) \not = \emptyset \ \ \Leftrightarrow \ \ \lhom(\matr{G},\matr{H}_{\alpha}) \not = \emptyset \ \ \Leftrightarrow \ \ \lhom(\matr{G}_{\alpha^{-1}},\matr{H}) \not = \emptyset.$$
}\end{cor}

 \begin{lem}{\label{lem:gamma}
Let $\sigma \in \chom(\matr{G},\matr{H})$, and also $\ell_{_\matr{G}}$, $\ell_{_\matr{H}}$ be the $(\Gamma, m)$-labelings of $\matr{G}$ and $\matr{H}$ respectively.
\begin{enumerate}
\item[{\rm 1.}]
If $e_{_1},e_{_2} \in E(\matr{G})$, $e_{_1}^-=e_{_2}^-$, $e_{_1}^+=e_{_2}^+$ and $\ell_{_{\matr{H}}}(\sigma_{_E}(e_{_1}))=\ell_{_{\matr{H}}}(\sigma_{_E}(e_{_2}))$ then\break $\ell_{_{\matr{G}}}(e_{_1})=\ell_{_{\matr{G}}}(e_{_2})$.
\item[{\rm 2.}]
Let $\matr{G}$ be symmetrically labeled. If $\alpha_{u}$ $($as defined in Proposition~$\ref{CHOMLEMMA})$ is a constant function on vertices of $\matr{G}$, then the image of $\sigma$ is a symmetrically labeled subgraph of $\matr{H}$.
\end{enumerate}
}\end{lem}
\begin{proof}{\
\begin{enumerate}
\item
Let $e_{_1}=e_{_2}=uv$, then by Proposition~\ref{CHOMLEMMA} for some $\alpha_u$ we have $$\ell_{_{\matr{G}}}^-(e_{_1})=\alpha_u\ell_{_{\matr{H}}}^-(\sigma_{_E}(e_{_1}))=\alpha_u\ell_{_{\matr{H}}}^-(\sigma_{_E}(e_{_2}))=\ell_{_{\matr{G}}}^-(e_{_2}).$$
The proof of $\ell_{_{\matr{G}}}^v(e_{_1})=\ell_{_{\matr{G}}}^v(e_{_2})$ is similar.
\item
Let $e=uv \in E(\matr{G})$. By the hypothesis, there exist $\lambda\in \Gamma$ and $i\in \nat[m]$
such that $\ell_{_{\matr{G}}}(e)=(\lambda, i, \lambda)$. Consequently, since $\alpha_{_u}=\alpha_{_v}$, we have
$\ell_{_{\matr{H}}}(\sigma_{_E}(e))=(\alpha_{_u}\lambda, i, \alpha_{_u}\lambda).$
\end{enumerate}
}\end{proof}

\section{The cylindrical graph construction}\label{sec:cylconst}
\subsection{The exponential construction}
Given a $\Gamma$-coherent set of $(t,k)$-cylinders
containing $m \geq 1$ cylinders, there is a canonical way of constructing an exponential graph that will be defined in the next definition.

 \begin{defin}{\label{defin:exponentialgraph}{\bf The exponential graph $[\matr{C},\matr{H}]_{_{\Gamma}}$}

 For any given labeled graph $\matr{H}$ and a $\Gamma$-coherent set of $(t,k)$-cylinders
$$\matr{C} \isdef \{ \matr{C}^{^{j}}(\by^{^{j}},\bz^{^{j}}, \J ^{^{j}}) \ \ | \ \ j \in \nat[m]\},$$
the {\it exponential graph} $[\matr{C},\matr{H}]_{_{\Gamma}}$
is defined (up to isomorphism) as a $(\Gamma,m)$-graph in the following way,
\begin{itemize}
\item[{\rm i)}]{Consider the right action of $\Gamma$ on $ V(\matr{H})^k$ as
$$\bv \gamma=(v_{_{1}},v_{_{2}},\ldots,v_{_{k}}) \gamma \isdef (v_{_{\gamma(1)}},v_{_{\gamma(2)}},\ldots,v_{_{\gamma(k)}}),$$
and the corresponding equivalence relation whose equivalence classes determine the orbits of this action, i.e.
$$\bu \sim_{_{\Gamma}} \bv \quad \Leftrightarrow \quad \exists\ \gamma \in \Gamma \ \ \ \bu\gamma=\bv.$$
Also, fix a set of representatives $U$ of these equivalence classes and let
$$V_{_{U}}([\matr{C},\matr{H}]_{_{\Gamma}}) \isdef U = \{\bu_{_1},\bu_{_2},\ldots,\bu_{_d}\} .$$
}
\item[{\rm ii)}]{There is an edge $e \isdef \bu_{_{i}}\bu_{_{k}} \in E_{_{U}}([\matr{C},\matr{H}]_{_{\Gamma}})$ with the label
$\ell(e)=(\gamma^{{\bu_{_{i}}}}, j, \gamma^{{\bu_{_{k}}}})$ if and only if there exists a homomorphism
$$(\sigma_{_{V}},\sigma_{_{E}}) \in \lhom(\matr{C}^{^j}(\by^{^{j}},\bz^{^{j}}, \J^{^{j}}),\matr{H})$$
such that,
$$\sigma_{_{V}}(\by^{^{j}}\gamma^{{\bu_{_{i}}}})=\bu_{_{i}} \quad {\rm and} \quad \sigma_{_{V}}(\bz^{^{j}}\gamma^{{\bu_{_{k}}}})=\bu_{_{k}}.$$ }
\item[{\rm iii)}]{Exclude all isolated vertices.}

 \end{itemize}

 Using Proposition~\ref{CHOMLEMMA}, one may verify that for two different sets of representatives $U$ and $W$ we have
$$\forall\ i \in \nat[d] \quad \exists\ \gamma_i \in \Gamma \quad \bu_{_{i}}\gamma_{_{i}} = \bw_{_{i}},$$
and that the map $ \sigma=(\sigma_{_{V}},\sigma_{_{E}})$ defined as
$$\forall\ i \in \nat[d] \quad \sigma_{_{V}}(\bu_{_{i}}) \isdef \bw_{_{i}},$$
and for $e=\bu_{_{i}}\bu_{_{k}}$ with $\ell(e)=(\gamma^{{\bu_{_{i}}}},j,\gamma^{{\bu_{_{k}}}})$,
$$\sigma_{_{E}}(e) \isdef \bw_{_{i}}\bw_{_{k}}, \quad {\rm with} \quad \ell(\sigma_{_{E}}(e)) =
(\gamma_{_{i}}\gamma^{{\bu_{_{i}}}},j,\gamma_{_{k}}\gamma^{{\bu_{_{k}}}}),$$
is an isomorphism of $(\Gamma,m)$-graphs, and consequently, hereafter, we assume that every
exponential graph is constructed with respect to a fixed
class of representatives $U$, and we omit the corresponding subscript since we are just dealing with such graphs up to an isomorphism.

 Also, note that in general an exponential graph is a directed graph, however, if one dealing with a $\Gamma$-coherent set of {\it symmetric} $(t,k)$-cylinders then by the analogy between symmetric graphs and their directed counterparts, where each simple edge is replaced by a pair of directed edges in different directions, for any given symmetric labeled graph $\matr{H}$ one may talk about a {\it symmetric exponential graph}, $[\matr{C},\matr{H}]^{^s}_{_{\Gamma}}$, in which all edges are simple since by the above
definition and the definition of a symmetric cylinder,
$$ {\bf u}{\bf v} \in E([\matr{C},\matr{H}]_{_{\Gamma}}) \quad \Leftrightarrow \quad {\bf v}{\bf u} \in E([\matr{C},\matr{H}]_{_{\Gamma}}). $$
Note that there may exist a directed cylinder $\matr{C}$ and a directed graph $\matr{H}$,
such that the symmetric exponential graph $[\matr{C},\matr{H}]^{^s}_{_{\Gamma}}$ is meaningful.}\end{defin}

 \noindent {\bf Notation.} The representative $\bu_{_{i}}$ of the equivalence class $[\bu]_{_{\sim_{_{\Gamma}}}}$ (where $\bu_{_{i}} = \bu \gamma$ for some $\gamma \in \Gamma$) is denoted by $\langle \bu \rangle_{_{\Gamma}}$.
Also, as before we may exclude some variables when they are set to their {\it default} values,
as the concept of a $\Gamma$-graph which stands for a $(\Gamma,1)$-graph or $[\matr{C},\matr{H}]$ which stands for $[\matr{C},\matr{H}]_{_{\Ident}}$.\\
Moreover, to clarify the figures, we may use the {\it name} of the cylinders, instead of their index
(e.g. see Figures~ \ref{fig:exponential} and \ref{fig:loopedline3}).
\begin{exm}{
\label{exm:exp}
Here we consider a couple of examples for the exponential graph construction.
\begin{itemize}
\item{{\bf The role of identity}

 It is clear that the identity cylinder satisfies the identity property $[\matr{I},\matr{H}] = \matr{H}$ in both directed and symmetric cases (see Example~\ref{exm:basiccylinders}).}

 \item{ {\bf The indicator construction} (e.g. see \cite{HENE} and references therein)

 The exponential graph $[\matr{C},\matr{H}]$
for a $1$-cylinder $\matr{C}(y,z)$ is the standard indicator construction, where the symmetric case $[\matr{C},\matr{H}]^{^s}$
is again the standard construction through the corresponding automorphism (see Definition~\ref{defin:cylinder}). }

 \begin{figure}[ht]
\begin{center}
\begin{tabular}{c c}
\includegraphics{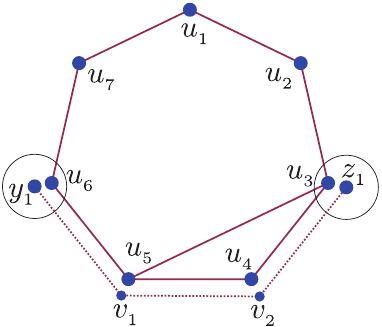} &
\includegraphics{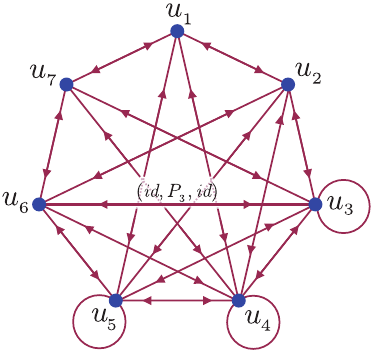}\\
\vspace{-0.5cm}\cr {$(a)$} & {$(b)$}
\end{tabular}
\end{center}
\caption{$(a)$ A base graph $\matr{G}$ and the Cylinder $\matr{P}_{{3}}$, \ $(b)$ The graph $\matr{G}^3=[\matr{P}_{{3}},\matr{G}]_{_{\Ident}}$.}
\label{fig:exponential}
\end{figure}

\item{ {\bf The $n$th power of a graph $\matr{G}$}

 It is easy to check that the $n$th power\footnote{Note that this is different from
the other definition for the $n$th power of a graph, where an edge is added when there exists a walk of length {\it less than or equal to} $n$.} $\matr{G}^n$ of a simple graph $\matr{G}$ with $|V(\matr{G})|$,
which is defined to be the graph on the vertex set $V(\matr{G}^n) \isdef V(\matr{G})$, obtained by adding an edge $uv$ if there exists a walk of length $n$ in $\matr{G}$ starting at $u$ and ending at $v$,
can be described as follows,
$$\matr{G}^n = [\matr{P}_{{n}},\matr{G}].$$
Note that $\matr{G}^n$ contains loops when $n$ is even.
An example of the power graph $\matr{G}^3$ is depicted in Figure~\ref{fig:exponential}, where
Figure~\ref{fig:exponential}$(a)$ shows the base graph $\matr{G}$ and a homomorphism of $\matr{P}_{{3}}$ to it, while
Figure~\ref{fig:exponential}$(b)$ shows the graph $\matr{G}^3$ in which the corresponding homomorphism is bolded out.}

 One can construct the directed version of graph power construction (see Figure~\ref{fig:directed}) by
$$\overrightarrow{\matr{G}}^n = [\overrightarrow{\matr{P}}_n,\overrightarrow{\matr{G}}].$$

 \item{{\bf Pultr right adjoint construction} \cite{adjoint, ladjoint},

 Given a Pultr template $\tau = (\matr{P}, \matr{Q}, \eta_{_1}, \eta_{_2})$ (see Example~\ref{exm:basiccylinders}) the \textit{central Pultr functor}
$\Gamma_{\tau}$ is a digraph functor which send any digraph $\matr{H}$ to a digraph $\Gamma_{\tau}(\matr{H})$ , where its vertices are the homomorphisms $g: \matr{P}\rightarrow \matr{H}$, and the arcs of $\Gamma_{\tau}(\matr{H})$ are
pairs $(g_{_1}, g_{_2})$ for which there exists a homomorphism $h : \matr{Q}\rightarrow \matr{H}$ such that
$$\qquad g_{_1} = h \circ \eta_{_1},\quad g_{_2} = h \circ \eta_{_2}.$$
Considering the corresponding cylinder $\matr{C}$ whose bases are isomorphic to $\matr{P}$ (see Example \ref{exm:basiccylinders}), one may verify that for one-to-one graph homomorphisms $\eta_{_1}$ and $\eta_{_2}$, we have $\Gamma_{\tau}(\matr{G}) = [\matr{C},\matr{G}]$.
}
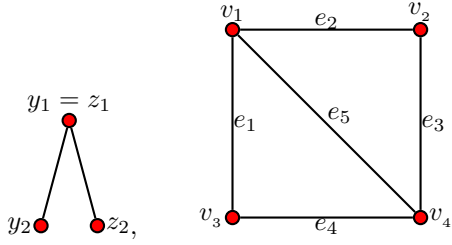
\begin{figure}[ht]
\begin{center}
\begin{tikzpicture}[thick,scale=0.3]
\tikzstyle{every node}=[ inner sep=0pt, minimum width=5pt]
\def\radius{2.5cm}
\node(e1)[circle, draw, fill=red,label=above:{$y_1=z_1$}] at (90:\radius) {} ;
\node(e2)[circle, draw, fill=red,label=right:{$z_2$}] at (300:\radius) {} ;
\node(e3)[circle, draw, fill=red,label=left:{$y_2$}] at (240:\radius) {} ;
\draw[-, thick] (e1) to node [auto]{} (e2);
\draw[-, thick] (e1) to node [auto]{} (e3);
\end{tikzpicture}, \qquad
\begin{tikzpicture}[thick,scale=0.25]
\tikzstyle{every node}=[ inner sep=0pt, minimum width=5pt]
\node (A)[draw,shape=circle,fill=red,label=$v_1$ ]{};
\node (B)[draw,node distance=2.5 cm,right of=A,shape=circle,fill=red,label=$v_{_2}$] {};
\node (C)[draw,node distance=2.5 cm,below of=A,shape=circle,fill=red,label=left:{$v_{_3}$}] {};
\node (D)[draw,node distance=2.5 cm, below of=B,shape=circle,fill=red,label=right:{$v_{_4}$}] {};
\draw[-, thick] (A) to node [auto]{$e_2$} (B);
\draw[-, thick] (A) to node [auto]{$e_1$} (C);
\draw[-, thick] (C) to node [below]{$e_4$} (D);
\draw[-, thick] (B) to node [auto]{$e_3$} (D);
\draw[-, thick] (A) to node [auto]{$e_5$} (D);
\end{tikzpicture}
\end{center}
\caption{The looped line-graph cylinder $\wedge$ and the base graph $\matr{H}$.}
\label{fig:llgexm}
\end{figure}

 \item{\bf The looped line-graph} (\cite{Abbass}, also see Proposition~\ref{pro:basics2}),

 It can be verified that the looped line-graph of a given graph $\matr{H}$ can be described as the exponential
graph $\left[ {\rm \wedge}, \matr{H}\right]_{{\bf S}_{_2}}$.
Note that in this case, there is an extra loop on each vertex of the line-graph, since for each such vertex $v$, there exists a homomorphism from the cylinder ${\rm \wedge}$ to $\matr{H}$ that maps $y$ and $z$ to $v$.

 Let us go through the details of such a construction for the graph $\matr{H}$ depicted in Figure~\ref{fig:llgexm}.
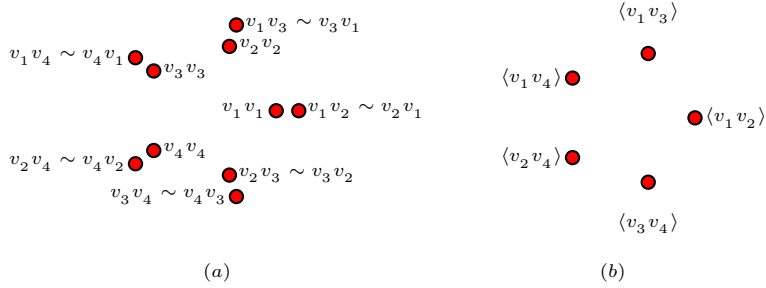
\begin{figure}[ht]
\begin{center}
\begin{scriptsize}
\begin{tabular}{cc}
\begin{tikzpicture}[thick,scale=0.3]
\tikzstyle{every node}=[circle, draw, fill=red, inner sep=0pt, minimum width=5pt]
\def\radius{4cm}
\node(e1)[label=right:{$v_{_1}v_{_2} \sim v_{_2}v_{_1}$}] at (0:\radius) { } ;
\node(e2)[label=right:{$v_{_1}v_{_3} \sim v_{_3}v_{_1}$}] at (72:\radius) {} ;
\node(e3)[label=left:{$v_{_1}v_{_4} \sim v_{_4}v_{_1}$}] at (144:\radius) {} ;
\node(e4)[label=left:{$v_{_2}v_{_4}\sim v_{_4}v_{_2}$}] at (216:\radius) {} ;
\node(e5)[label=left:{$v_{_3}v_{_4} \sim v_{_4}v_{_3}$}] at (288:\radius) {} ;
\node(e1)[label=left:{$v_{_1}v_{_1}$}] at (0:3cm) { } ;
\node(e2)[label=right:{$v_{_2}v_{_2}$}] at (72:3cm) {} ;
\node(e3)[label=right:{$v_{_3}v_{_3}$}] at (144:3cm) {} ;
\node(e4)[label=right:{$v_{_4}v_{_4}$}] at (216:3cm) {} ;
\node(e5)[label=right:{$v_{_2}v_{_3}\sim v_{_3}v_{_2}$}] at (288:3cm) {} ;
\end{tikzpicture}
&\qquad
\begin{tikzpicture}[thick,scale=0.3]
\tikzstyle{every node}=[circle, draw, fill=red, inner sep=0pt, minimum width=5pt]		
\def\radius{3cm}
\node(e1)[label=right:{$\langle v_{_1}v_{_2} \rangle$}] at (0:\radius) {} ;
\node(e2)[label=above:{$\langle v_{_1}v_{_3} \rangle$}] at (72:\radius) {} ;
\node(e3)[label=left:{$\langle v_{_1}v_{_4} \rangle$}] at (144:\radius) {} ;
\node(e4)[label=left:{$\langle v_{_2}v_{_4} \rangle$}] at (216:\radius) {} ;
\node(e5)[label=below:{$\langle v_{_3}v_{_4} \rangle$}] at (288:\radius) {} ;
\end{tikzpicture}\\
$(a)$&$(b)$
\end{tabular}
\end{scriptsize}
\end{center}
\caption{$(a)$ Arrange vertices, $(b)$ Representatives. (for Construct exponential
graph $\left[ {\rm \wedge}, \matr{H}\right]_{{\bf S}_{_2}}$).}
\label{fig:loopedline1}
\end{figure}

\begin{enumerate}
\item First we present the set of vertices corresponding to the equivalence classes (see Figure~\ref{fig:loopedline1}$(a)$).
\item Fixed a set of representatives. Exclude isolated vertices i.e. delete any pair $XY$ for which we can not find a graph homomorphism $\matr{K}_{_2}=\matr{B}^-\rightarrow \matr{H}([X,Y])$ (see Figure~\ref{fig:loopedline1}$(b)$).
\item To specify the edges,
\begin{itemize}
\item Consider a graph homomorphism for which $x_{_1} \mapsto v_{_1} , \quad x_{_2} \mapsto v_{_2},$ and $\quad y_{_1} \mapsto v_{_3}$.
Note that this homomorphism corresponds to the edge $\langle v_{_1}v_{_3}\rangle-\langle v_{_1}v_{_2}\rangle$ as depicted in Figure~\ref{fig:loopedline2}$(a)$.
\item Again, consider a graph homomorphism for which $x_{_1}\mapsto v_{_4} , \quad x_{_2} \mapsto v_{_2},$ and $ \quad y_{_1} \mapsto v_{_1}$,
and consequently, we have the edge $\langle v_{_4}v_{_1}\rangle -\langle v_{_4}v_{_2}\rangle$, however, we have to use the twist $\pi=\lfloor 1,2\rceil$ on both bases
to find the corresponding representatives as $\pi(v_{_4}v_{_1})=v_{_1}v_{_4}$ and $\pi(v_{_4}v_{_2})=v_{_2}v_{_4}$ (see Figure~\ref{fig:loopedline2}$(b)$).
\item By continuing this procedure, we obtain the loop-lined graph as depicted in Figure~\ref{fig:loopedline3}.
\end{itemize}
\end{enumerate}

 \begin{figure}[ht]
\begin{center}
\begin{scriptsize}
\begin{tabular}{cc}
\begin{tikzpicture}[thick,scale=0.3]
\tikzstyle{every node}=[inner sep=0pt, minimum width=5pt]
\def\radius{3cm}
\node(e1)[circle, draw, fill=red,label=right:{$\langle v_{_1}v_{_2} \rangle$}] at (0:\radius) {} ;
\node(e2)[circle, draw, fill=red,label=above:{$\langle v_{_1}v_{_3}\rangle$}] at (72:\radius) {} ;
\node(e3)[circle, draw, fill=red,label=left:{$\langle v_{_1}v_{_4}\rangle$}] at (144:\radius) {} ;
\node(e4)[circle, draw, fill=red,label=left:{$\langle v_{_2}v_{_4} \rangle$}] at (216:\radius) {} ;
\node(e5)[circle, draw, fill=red,label=below:{$\langle v_{_3}v_{_4} \rangle$}] at (288:\radius) {} ;
\draw[->, thick] (e2) to node [auto]{$id,\wedge , id$} (e1);
\end{tikzpicture}
\qquad & \qquad
\begin{tikzpicture}[thick,scale=0.3]
\tikzstyle{every node}=[inner sep=0pt, minimum width=5pt]
\def\radius{3cm}
\node(e1)[circle, draw, fill=red,label=right:{$\langle v_{_1}v_{_2} \rangle$}] at (0:\radius) {} ;
\node(e2)[circle, draw, fill=red,label=above:{$\langle v_{_1}v_{_3}\rangle$}] at (72:\radius) {} ;
\node(e3)[circle, draw, fill=red,label=left:{$\langle v_{_1}v_{_4}\rangle$}] at (144:\radius) {} ;
\node(e4)[circle, draw, fill=red,label=left:{$\langle v_{_2}v_{_4}\rangle$}] at (216:\radius) {} ;
\node(e5)[circle, draw, fill=red,label=below:{$\langle v_{_3}v_{_4}\rangle$}] at (288:\radius) {} ;
\draw[->, ] (e2) to node [auto]{$id,\wedge , id$} (e1);
\draw[->, thick] (e3) to node [left]{$\pi,\wedge , \pi$} (e4);
\end{tikzpicture}
\\
$(a)$&$(b)$
\end{tabular}
\end{scriptsize}
\end{center}
\caption{Building edges of the exponential
graph $\left[ {\rm \wedge}, \matr{H}\right]_{{\bf S}_{_2}}$.}
\label{fig:loopedline2}
\end{figure}
\begin{figure}[ht]
\centering{
\begin{tabular}{cc}
\begin{tikzpicture}[thick,scale=0.5]
\tikzstyle{every node}=[inner sep=0pt, minimum width=6pt]
\def\radius{4cm}
\node(ab)[circle, draw, fill=red,label=right:{$\langle v_{_1}v_{_2} \rangle$}] at (0:\radius) {} ;
\node(ac)[circle, draw, fill=red,label=above:{$\langle v_{_1}v_{_3} \rangle$}] at (72:\radius) {} ;
\node(ad)[circle, draw, fill=red,label=left:{$\langle v_{_1}v_{_4}\rangle$}] at (144:\radius) {} ;
\node(bd)[circle, draw, fill=red,label=left:{$\langle v_{_2}v_{_4} \rangle$}] at (216:\radius) {} ;
\node(cd)[circle, draw, fill=red,label=below:{$\langle v_{_3}v_{_4} \rangle$}] at (288:\radius) {} ;
\draw[<->, thick ] (ab) to node [above,sloped]{$id,\wedge , id \qquad\qquad$} (ad);
\draw[<->, thick ] (ac) to node [above,sloped]{$id,\wedge , id$} (ad);
\draw[<->, thick ] (ac) to node [above,sloped]{$id,\wedge , id$} (ab);
\draw[<->, thick ] (cd) to node [above,sloped]{$\qquad\qquad\qquad id,\wedge , \pi$} (ac);
\draw[<->, thick ] (ad) to node [above,sloped]{$\pi,\wedge , \pi \qquad$} (cd);
\draw[<->, thick ] (bd) to node [above,sloped]{$\pi ,\wedge , \pi $} (ad);
\draw[<->, thick ] (ab) to node [below,sloped]{$\qquad\qquad\qquad\pi ,\wedge , id$} (bd);
\draw[<->, thick ] (bd) to node [below,sloped]{$\pi ,\wedge , \pi$} (cd);
\draw[->, thick ,loop below,distance=2cm] (bd) to node{$id,\wedge ,id$} (bd);
\draw[->, thick ,loop right,distance=2cm] (ac) to node {$id,\wedge ,id$} (ac);
\draw[->, thick ,loop below,distance=2cm] (ab) to node {$id,\wedge ,id$} (ab);
\draw[->, thick ,loop above,distance=2cm] (ad) to node {$id,\wedge ,id$} (ad);
\draw[->, thick ,loop right,distance=2cm] (cd) to node {$id,\wedge ,id$} (cd);
\end{tikzpicture}
\end{tabular}}
\caption{The exponential graph $\left[ {\rm \wedge},\matr{H}\right]$ .}
\label{fig:loopedline3}
\end{figure}

 \item {\bf A directed example}

 Figure~\ref{fig:directed} shows a directed graph $\matr{T}$ and the exponential graph $[\vec{\matr{P}}_{_{2}},\matr{T}]$
(here note that we have excluded all default values of the parameters involved in our notation).
\end{itemize}

 }\end{exm}

\begin{figure}[ht]
\begin{center}
\begin{tabular}{c c}
\includegraphics{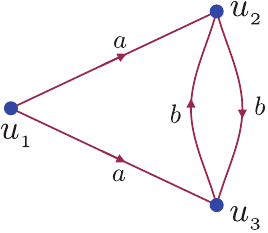} & \includegraphics{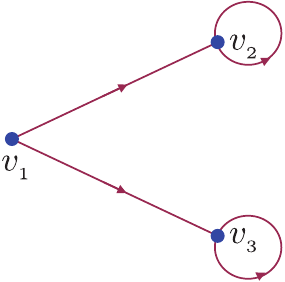} \\
\vspace{-0.5 cm}\cr {$(a)$} & {$(b)$}
\end{tabular}
\end{center}
\caption{$(a)$ The directed graph $\matr{T}$, \ \ $(b)$ The graph $[\vec{\matr{P}}_{{2}},\matr{T}]$.}
\label{fig:directed}
\end{figure}

 \subsection{The cylindrical construction}

 Let $\matr{C} \isdef \{\matr{C}^{^{j}}(\by^{^{j}},\bz^{^{j}}, \J ^{^{j}}) \ \ | \ \ j \in \nat[m]\}$ be a set
of $\Gamma$-coherent $(t,k)$-cylinders, where each $\matr{C}^{^{j}}$ is an $\ell_{_{j}}$-graph.
Also, let $\matr{G}$ be a $(\Gamma,m)$-graph labeled by the map $\ell_{_{\matr{G}}}$. Then the
{\it cylindrical product} of $\matr{G}$ and $\matr{C}$ is an amalgam that, intuitively, can
be described as the graph constructed by replacing each edge of $\matr{G}$
whose label is $(\gamma^-,j,\gamma^+)$, by a copy of the cylinder
$\matr{C}^{^{j}}$ twisted by $(\gamma^-,\gamma^+)$, while we identify (the vertices of the) bases of the cylinders that intersect at the position of
each vertex of $\matr{G}$. Formally, this construction denoted by $\matr{G} \cyl{\Gamma} \matr{C}$ can be defined as,
\begin{equation}
\matr{G} \cyl{\Gamma} \matr{C} \isdef \sum_{uv \in E(\matr{G})\ : \ \ell_{_{\matr{G}}}(uv)=(\gamma^u,j,\gamma^v)} \matr{C}^{^{j}}(\bu^{^{j}}\gamma^u,\bv^{^{j}}\gamma^v, (\gamma^u,\gamma^v)\J ^{^{j}}).
\end{equation}

 To be more descriptive, let us describe the construction in an algorithmic way as follows.
Let $V(\matr{G}) = \{v_{_{1}},v_{_{2}},\ldots,v_{_{n}}\}$.
Then the graph $\matr{G} \cyl{\Gamma} \matr{C}$ is constructed on $\matr{G}_{_{0}}$ as the
graph that is generated by the following algorithm,
\begin{itemize}
\item[$1.$]{{\bf Blowing up vertices:} (This step is just for clarification and may be omitted.)\\
Consider the empty graph $\matr{G}_{_{0}}$ with
$$V(\matr{G}_{_{0}}) \isdef \displaystyle{\bigcup_{j=1}^{n}} \ \{v_{_{1}}^{^{j}},v_{_{2}}^{^{j}},\ldots,v_{_{k}}^{^{j}}\},$$
and $E(\matr{G}_{_{0}})=\emptyset$. }
\item[$2.$]{{\bf Choosing the cylinders:} For every edge $v_{_{p}}v_{_{q}} \in E(\matr{G})$ with the label
$\ell_{_{\matr{G}}}(v_{_{p}}v_{_{q}})=(\gamma^-,j,\gamma^+)$, choose a twisted cylinder
$\matr{C}^{^{j}}(\by^{^{j}}\gamma^-,\bz^{^{j}}\gamma^+)$ in which the vertices of the bases are {\it not} identified
according to $\J ^{^j}$) to $\matr{G} \cyl{\Gamma} \matr{C}$ along with
all its vertices and edges yet. Re-mark the vertices in each base and construct the cylinder
$\matr{C}^{^{j}}(\bv^{^{p}},\bv^{^{q}})$ for which
\begin{itemize}
\item{Let $\bv^{^{p}}=\by^{^{j}}\gamma^-$,}
\item{Let $\bv^{^{q}}=\bz^{^{j}}\gamma^+$.}
\end{itemize}
Consider the set $\matr{G}^*$ of disjoint union of all these cylinders constructed for each edge.
}
\item[$3.$]{{\bf Identification:} Consider the amalgam constructed on $\matr{G}_{_{0}} \cup \matr{G}^*$ as a marked graph and identify
all vertices first with respect to markings and then with respect to $\J ^{^j}$'s. In the end, identify all
multiple edges that have the same label.}
\end{itemize}
Note that given any
edge $e \in E(\matr{G}\cyl{\Gamma} \matr{C})$, there exists an index $j \in \nat[m]$ such that $e \in E(\matr{C}^{^{j}})$, and one may assign a well defined labeling
$$\ell_{_{\matr{G}\cyl{\Gamma} \matr{C}}}(e) \isdef \ell_{_{j}}(e).$$
Consequently, $\matr{G}\cyl{\Gamma} \matr{C}$ is an $\ell_{_{\matr{G}\cyl{\Gamma} \matr{C}}}$-graph.

 Also, as in Definition~\ref{defin:exponentialgraph}, if $\matr{C}$ is a set
of $\Gamma$-coherent {\it symmetric} $(t,k)$-cylinders, then for any symmetric graph $\matr{G}$, one may talk about the {\it symmetric} cylindrical construct $\matr{G}\scyl{\Gamma} \matr{C}$. Note that if $\matr{C}$ is directed, then the symmetric construct $\matr{G}\scyl{\Gamma} \matr{C}$ is also a directed graph although $\matr{G}$ is symmetric.\\ \ \\
\noindent {\bf Notation.} If $V(\matr{G}) = \{v_{_{1}},v_{_{2}},\ldots,v_{_{n}}\}$, hereafter the vertex set of the cylindrical construct
$\matr{G}\cyl{\Gamma} \matr{C}$ is assumed to be $V(\matr{G}\cyl{\Gamma} \matr{C}) = \bigcup_{j=1}^m \{v_{_{1}}^{^{j}},v_{_{2}}^{^{j}},\ldots,v_{_{k}}^{^{j}}\}$
where the superscript refers to the index of the corresponding vertex in $V(\matr{G})$ and, moreover, we may refer to the list
$\bv^{^{j}} \isdef (v_{_{1}}^{^{j}},v_{_{2}}^{^{j}},\ldots,v_{_{k}}^{^{j}})$. Also, we assume that in each cylinder, the set of vertices that do not appear in the vertex sets of the bases are enumerated (say from $1$ to $r$), and for each fixed edge $e \in E(\matr{G})$, we may refer to the list
$\bu^{{e}} \isdef (u_{_{1}}^{{e}},u_{_{2}}^{{e}},\ldots,u_{_{r}}^{{e}})$ consisting of non-base vertices that appear in $\matr{G}\cyl{\Gamma} \matr{C}$ when the edge $e$ is replaced by the corresponding cylinder.
\begin{figure}[ht]
\centering{\begin{tabular}{c c}
\includegraphics[width=4cm]{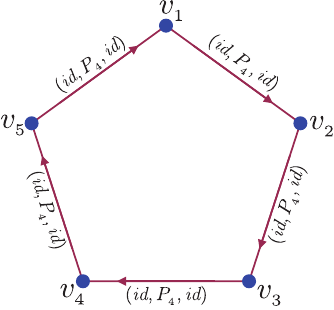} & \includegraphics[width=4cm]{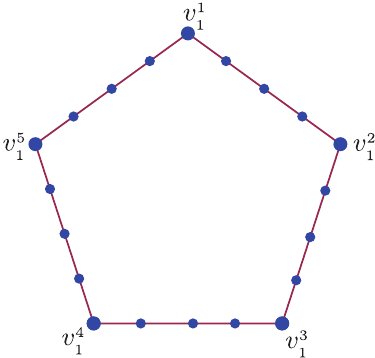} \\
\vspace{-0.5cm}\cr {$(a)$} & {$(b)$}
\end{tabular}}
\caption{$(a)$ An $(\Ident,1)$-graph $\matr{C}_{_{5}}$, $(b)$ The $4$-subdivision $\matr{C}_{_{5}}\cyl{} \matr{P}_{{4}}$.}
\label{fig:sub_prod}
\end{figure}

 \begin{exm}{ \label{exm:cylop} In this example we go through some basic and well-known special cases.
\begin{itemize}
\item{{\bf The identity relation}

 It is clear that the identity cylinder satisfies the identity property $\matr{G} \cyl{} \matr{I} = \matr{G}$ in both directed and symmetric cases (see Example~\ref{exm:basiccylinders}).
Note that, in general, cylindrical products using the identity cylinder gives rise to lifts of graphs (e.g. see \cite{MSS}
and references therein for the details, applications and background. Also see Figure~\ref{fig:cylc}.).
}

 \item{{\bf The replacement operation} (e.g. see \cite{HENE} and references therein)

 The cylindrical construction $\matr{G}\cyl{} \matr{C}$
for a $1$-cylinder $\matr{C}=(y,z)$ is the standard replacement operation.
}
\item{{\bf The fractional power}

 It is easy to see that $\matr{G}\cyl{} \matr{P}_{{n}}$ gives rise
to an $n$-subdivision of $\matr{G}$ which can be described as subdividing every edge of $\matr{G}$ by $n-1$ vertices (see Figure~\ref{fig:sub_prod} for an example of this case). Also, note that one may describe the fractional power
$\matr{G}^{^{\frac{m}{n}}}$ as follows (see \cite{haji} for the definition and more details),
$$\matr{G}^{^{\frac{m}{n}}} \isdef [\matr{P}_{{m}},\matr{G}\cyl{} \matr{P}_{{n}}].$$}

 \item{{\bf Fiber construction} \cite{fiber, fiber1},

 Let $\matr{M}$ be a fiber gadget (see Example~\ref{exm:basiccylinders}) with $W^{^a}$ and $W^{^b}$ (as bases), and let $\matr{G}$ be a graph. Also, for each vertex $v$ of $\matr{G}$, let $W^{^v}$ be a copy of $W^{^*}$, and for any edge $uv$ of $\matr{G}$,
let $\matr{M}_{_{uv}}$ be a copy of $\matr{M}$. Then, identifying $W^{^u}$ and $W^{^v}$ with the copies of $W^{^a}$ and $W^{^b}$, respectively, in $\matr{M}$, we get a graph that is the fiber construction $\matr{M}(\matr{G})$, and can be represented as
$$\matr{M}(\matr{G})=\matr{G}\cyl{}\matr{C}(\matr{M}).$$
}
\item{{\bf Pultr left adjoint construction} \cite{adjoint, ladjoint},

 Given a Pultr template $\tau = (\matr{P}, \matr{Q}, \eta_{_1}, \eta_{_2})$ (see Example~\ref{exm:basiccylinders}), the \textit{left Pultr functor}
$\Lambda_{\tau}$ is a digraph functor which send any digraph $\matr{G}$ to a digraph $\Lambda_{\tau}(\matr{G})$, whose
vertices are the copies $\matr{P}_u$ for any vertex $u\in V(\matr{G})$, and for the arcs,
$\Lambda_{\tau}(\matr{G})$ contains a copy $Q_{uv}$ of $Q$ for any arc
$uv\in E(\matr{G})$, where $\eta_{_1}[\matr{P}]$ with $\matr{P}_u$
and $\eta_{_2}[\matr{P}]$ with $\matr{P}_v$.
Considering the corresponding cylinder $\matr{C}$ whose bases are isomorphic to $\matr{P}$, one may verify that
$\Lambda_{\tau}(\matr{G}) = \matr{G} \cyl{} \matr{C}$.

 Note that the fiber gadget is a special instance of Pultr template, and fiber construction is a special instance of
Pultr left adjoint construction.
}
\item{{\bf The join operation}\\
The join operation $\matr{G} \nabla \matr{H}$ is a graph which is the union of two graph $\matr{G}$ and $\matr{H}$ with additional edges between any vertex of $\matr{G}$ and any vertex of $\matr{H}$.

 Let $\matr{G},\matr{H}$ be two graphs. Then one can show that if $\matr{H}_{\nabla}$ be a cylinder (see Figure ~\ref{fig:join}$(a)$ and ~\ref{fig:join}$(b)$) with $\matr{B}^-= \matr{H}\nabla y$, $\matr{B}^+= \matr{H}\nabla z$, $\J=\{(h,h) | h\in V(\matr{H}) \}$ and we have an extra edge $yz$, then $$\matr{G}\cyl{} \matr{H}_{\nabla} \simeq \matr{G} \nabla \matr{H}.$$
\begin{figure}[ht]
\centering{\begin{tabular}{ccc}
\includegraphics{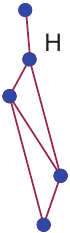} &
\includegraphics{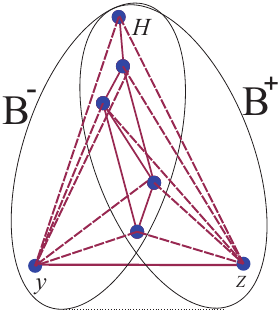}&
\includegraphics{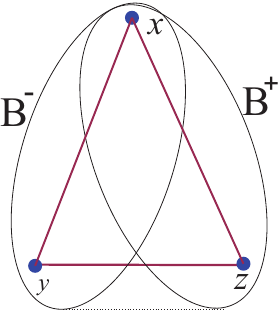} \\
\vspace{-0.5cm}\cr {$(a)$} & {$(b)$}& {$(c)$}
\end{tabular}}
\caption{$(a)$ A base graph $\matr{H}$, $(b)$ The join cylinder $\matr{H}_{\nabla}$, $(c)$ The cylinder $\triangle$. }
\label{fig:join}
\end{figure}}
\item{{\bf The universal vertex construction}

 One may verify that the result of the cylindrical construction $\matr{G}\cyl{} {\rm \triangle}$ (see Figure ~\ref{fig:join}$(c)$)
can be described as adding a universal vertex to the graph $\matr{G}$ (i.e. a new vertex that is adjacent to each vertex in $V(\matr{G})$).}

 \item{{\bf Deletion and contraction operations}

 Let $\matr{d}(y,z)$ and $c=\matr{P}_{_0}$ be deletion and contraction cylinders respectively, and $\matr{G}$ be a graph labeled by
$\{\matr{d},\matr{c},\matr{I}\}$ and define $\matr{C} \isdef \{\matr{d}(y,z), \matr{P}_{_0}, \matr{I}(y,z)\}$.
Then $\matr{G}\cyl{}\matr{C}$ is a graph obtained from $\matr{G}$ in which edges labeled by $\matr{c}$ are contracted
and edges labeled by $\matr{d}(y,z)$ are deleted. Hence, any minor is a cylindrical construct.}

 \item{{\bf Generalized loop cylinders}

 Let $\matr{C}$ be a $\Gamma$-cylinder, and also let $\matr{L}$ be a loop on a vertex $v$
with the label $(\lambda, \gamma)\in \Gamma \times \Gamma$. Then we define
$$Gl_{_{(\lambda, \gamma)}}(\matr{C}) \isdef \matr{L}\cyl{\Gamma}\matr{C}$$ which is a generalized loop cylinder.
Note that by definition, any generalized loop cylinder can be constructed as mentioned above.

 }

 \item{{\bf The role of twists in cylindrical construction}

 As depicted in Figure \ref{fig:cylc}, different labelings of $\matr{K}_{_3}$
in $\matr{K}_{_3} \cyl{\matr{S}_{_2}} \matr{I}_{_2}$
can lead to different cylindrical constructions. This example shows that the choice of labeling can even affect the connectivity of the product graph.

 \item{
One can verify that,
\begin{prepro}{\label{pro:plusu} For any $\Gamma$-graph $\matr{G}$, we have
\ \\
\begin{itemize}
\item[{\rm a)}]{For $(t,k)$-cylinders $\matr{A}$ and $\matr{B}$,
$$\matr{G}\cyl{\Gamma} (\matr{A}+\matr{B}) \simeq \matr{G}\cyl{\Gamma} \matr{A}+\matr{G}\cyl{\Gamma} \matr{B}.$$
}
\item[{\rm b)}]{For $(t,k)$-cylinder $\matr{A}$ and $(t',k')$-cylinder $\matr{B}$, $\Gamma_{_1} \leq \textbf{S}_{_{k}},\ \Gamma_{_2} \leq \textbf{S}_{_{k'}}$ and $\Gamma \isdef \Gamma_{_1}\sqcup \Gamma_{_2}$, then,
$$\matr{G}\cyl{\Gamma} (\matr{A}\Cup \matr{B}) \simeq (\matr{G}\cyl{\Gamma_{_2}} \matr{A}) \sqcup (\matr{G}\cyl{\Gamma_{_2}} \matr{B}),$$
where $\sqcup$ is the disjoint union of two graphs.}
\end{itemize}
}\end{prepro}

}

 }\end{itemize}}
\end{exm}
\begin{figure}[ht]
\centering{
\includegraphics[scale=.65]{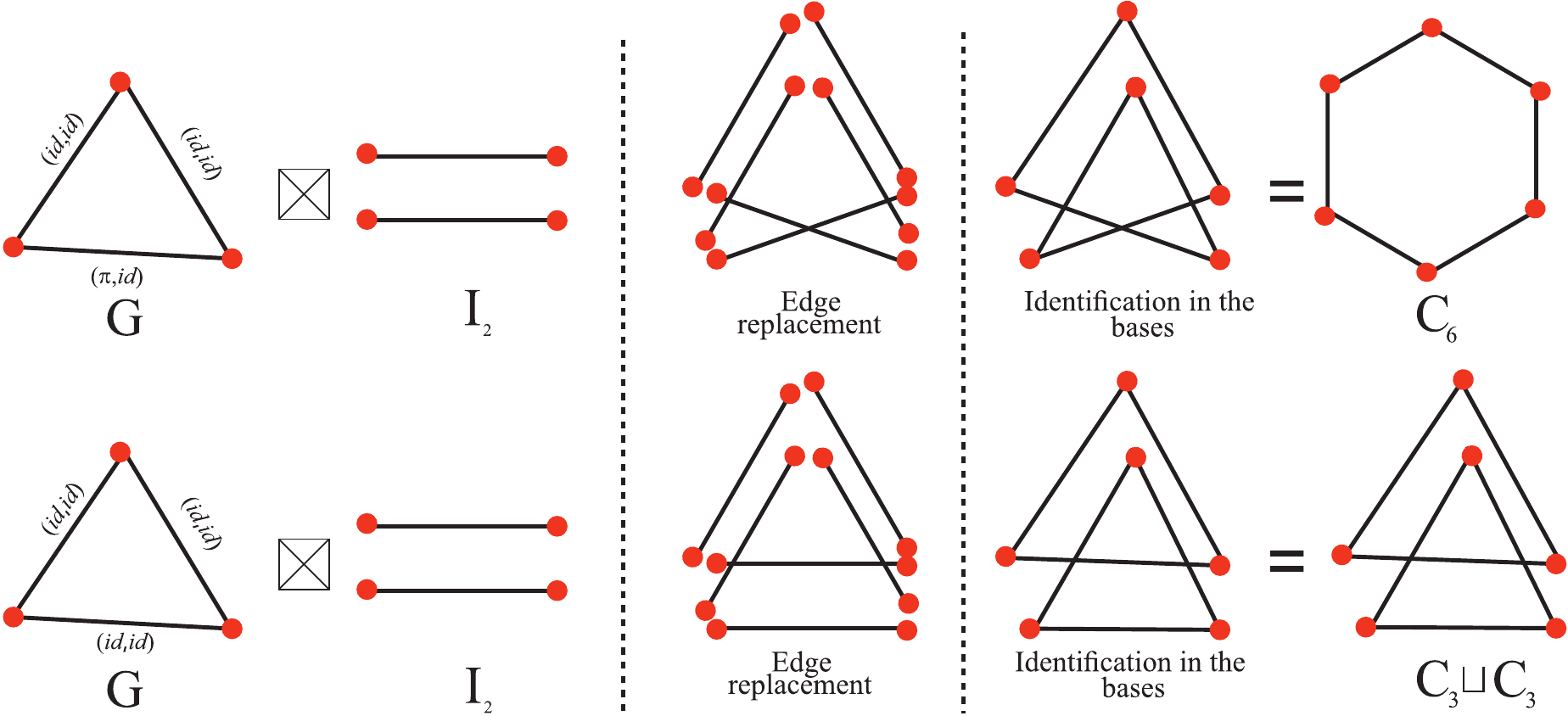}}
\caption{The role of twists in cylindrical construction.}
\label{fig:cylc}
\end{figure}
\begin{exm}{\label{exm:petersen}
In this example we elaborate on different descriptions of the Petersen graph using cylindrical constructions, where in this regard we will
also describe any {\it voltage graph construction} as a special case of a cylindrical construction. This specially shows that a graph may have many different descriptions as a cylindrical product.
\begin{itemize}
\item{{\bf Petersen graph and the $\sqcap$-cylinder}\\
Consider the complete graph $\widehat{\matr{K}}_{_{5}}$ of
Figure~\ref{fig:petersen_prod}$(a)$ as a $({\bf S_{_{2}}},1)$-graph, and note that the cylindrical construction
$(\widehat{\matr{K}}_{_{5}} \cyl{{\bf S_{_{2}}}} {\sqcap})$ gives rise to the Petersen graph as depicted in Figure~\ref{fig:petersen_prod}$(b)$.
\begin{figure}[ht]
\centering{
\begin{tabular}{c c}
\includegraphics{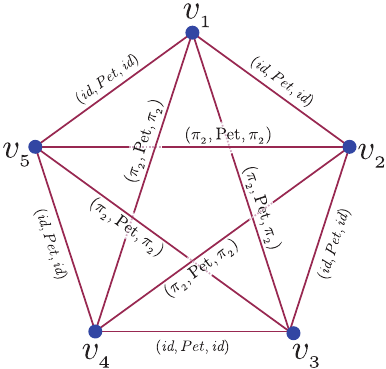} & \includegraphics{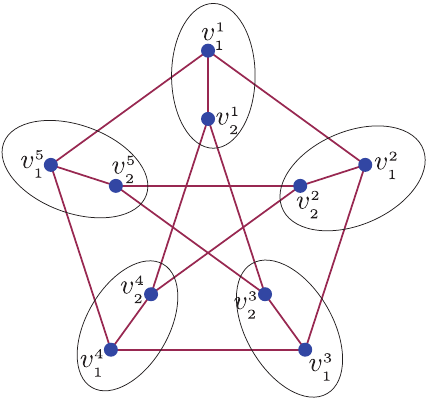} \\
\vspace{-0.5cm}\cr {$(a)$} & {$(b)$}
\end{tabular}
}
\caption{$(a)$ The base graph $\widehat{\matr{K}}_{_{5}}$, $(b)$ The Petersen graph as $\widehat{\matr{K}}_{_{5}} \cyl{{\bf S_{_{2}}}}{\sqcap}$.}
\label{fig:petersen_prod}
\end{figure}
}
\item{{\bf Petersen graph and the voltage graph construction} \\
Let $A$ be a set of $k$ elements and $\Omega$ be a group of $n$ elements that acts on $A$ from the right.
\begin{figure}[ht]
\begin{center}
\begin{tabular}{c c}
\includegraphics[width=4cm]{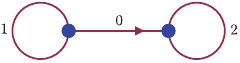} & \includegraphics[width=4cm]{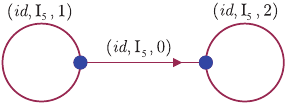} \\
\vspace{-0.5cm}\cr {$(a)$} & {$(b)$}
\end{tabular}
\end{center}
\caption{$(a)$ Petersen Voltage graph, $(b)$ The corresponding $(\Gamma,m)$-graph.}
\label{fig:petersen_voltage}
\end{figure}
Then it is clear that one may identify
the set $A$ as $A=\{a_{_{1}},a_{_{2}},\ldots,a_{_{k}}\}$ and $\Omega$ as a permutation group acting on the index set $\nat[k]$.\\
An {\it action voltage graph} is a (directed) labeled graph $(\matr{G},\ell_{_{\matr{G}}}: E(\matr{G}) \longrightarrow \Omega)$
whose edges are labeled by the elements of the group $\Omega$. The {\it derived graph} of
$\matr{G}$, denoted by $\tilde{\matr{G}}$, is a graph on the vertex set $V(\tilde{\matr{G}}) = V({\matr{G}}) \times A$, where
there is an edge between $(u,a_{_{i}})$ and $(v,a_{_{j}})$ if $e=uv \in E(\matr{G})$ and
$a_{_{j}}=a_{_{i}}\ell_{_{\matr{G}}}(e)$. \\
When $\Omega$ acts on itself on the right, then $A=\Omega$ and the action voltage graph is simply called the {\it voltage graph},
when we usually refer to the corresponding derived graph as the {\it voltage graph construct} in this case. It is also easy to verify that
Cayley graphs are among the well-known examples of voltage graphs (e.g. see \cite{HENE} for more on voltage graphs and related topics).\\
Now, we show that action voltage graphs and their derived graphs are cylindrical constructs.
For this, given an action voltage graph $(\matr{G},\ell_{_{\matr{G}}}: E(\matr{G}) \longrightarrow \Omega)$, we define the corresponding
$(\Omega,1)$-graph $\widehat{\matr{G}}$ as a graph with the same vertex and edge sets as $\matr{G}$, and the labeling
$$\forall\ e \in E(\widehat{\matr{G}}) \quad \ell_{_{\widehat{\matr{G}}}}(e) =
(\ident,\matr{I}_{_{k}},\rho) ~ \Leftrightarrow ~ \ell_{_{\matr{G}}}(e)=\rho,$$
where $\matr{I}_{_{k}}$ is the identity cylinder on $k$ vertices (see Figure~\ref{fig:voltage_cylinder}).
Then it is easy to see that $\tilde{\matr{G}}=\widehat{\matr{G}} \cyl{\Omega} \matr{I}_{_{k}}.$
Figure \ref{fig:petersen_voltage} shows the voltage graph and the corresponding $({\bf Z}_{_{5}},1)$-graph to construct the Petersen graph using this construction.}
\end{itemize}
\begin{figure}[ht]
\centering{
\begin{tabular}{c c c c}
\includegraphics[width=1.5cm]{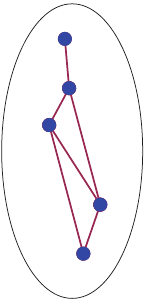} & \includegraphics[width=4cm]{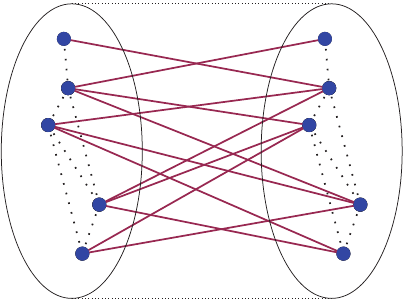} &\includegraphics[width=4cm]{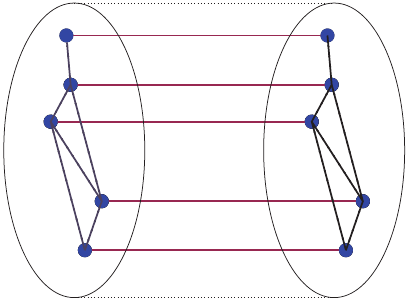} \\
\vspace{-0.5cm}\cr {$(a)$} & {$(b)$} & {$(c)$}
\end{tabular}}
\caption{$(a)$ The graph $\matr{G}$, $(b)$ The cylinder $\matr{C}_{_{\times}}(\matr{G})$, $(c)$ The cylinder $\matr{C}_{_{\Box}}(\matr{G})$. }
\label{fig:neps}
\end{figure}
}\end{exm}
\begin{exm}{\label{exm:neps}{\bf NEPS of simple graphs}\\
In this example we claim that any NEPS of simple graphs can be described as a cylindrical construction (e.g. see \cite{HENE} for the definition and more details).
We concentrate on the more classical cases of Cartesian product $\matr{G} \Box \matr{H} = {\rm NEPS}_{\{(0,1),(1,0)\}}(\matr{G},\matr{H})$ and the categorical product $\matr{G} \times \matr{H} = {\rm NEPS}_{\{(1,1)\}}(\matr{G},\matr{H}) $.
Note that, the basic idea of generalizing the whole construction to arbitrary components or arbitrary NEPS Boolean sets is straight forward by noting the fact that the adjacency matrix of a NEPS of graphs is essentially constructible through suitable tensor products, using identity matrix and the adjacency matrices of the components.

 For this consider a graph $\matr{G}$ with $V(\matr{G})=\{v_{_{1}}, \ldots ,v_{_{\nu}}\}$. Then define cylinders $\matr{C}_{_{\Box}}(\matr{G})(\by,\bz)$ and $\matr{C}_{_{\times}}(\matr{G})(\by,\bz)$ where the base of $\matr{C}_{_{\Box}}(\matr{G})(\by,\bz)$ is
isomorphic to $\matr{G}$, base of $\matr{C}_{_{\times}}(\matr{G})(\by,\bz)$ is an empty graph on $V(\matr{G})$,
and the edges are described as follows (see Figure~\ref{fig:neps}),
$$y_{_{i}}z_{_{j}} \in E(\matr{C}_{_{\Box}}(\matr{G})(\by,\bz)) \quad \Leftrightarrow \quad i=j,$$
$$y_{_{i}}z_{_{j}} \in E(\matr{C}_{_{\times}}(\matr{G})(\by,\bz)) \quad \Leftrightarrow \quad v_{_{i}} v_{_{j}}\in E(\matr{G}).$$
Then one may verify that
$$ \matr{G} \Box \matr{H} = \matr{H} \cyl{} \matr{C}_{_{\Box}}(\matr{G}) \quad {\rm and} \quad \matr{G} \times \matr{H}= \matr{H} \cyl{} \matr{C}_{_{\times}}(\matr{G}). $$
}\end{exm}
The following is a classical result in the literature (see \cite{HENE}) for the graph $\matr{H}^\matr{G} $ defined as,
$$V(\matr{H}^\matr{G})=\{f:V(\matr{G}) \rightarrow V(\matr{H})\},$$
$$fg\in E(\matr{H}^\matr{G}) \Leftrightarrow \forall vw\in E(\matr{G}) \quad f(v)g(w) \in E(\matr{H}).$$
\begin{pro}{\label{pro:ppower}
For any pair of simple graphs $\matr{H}$ and $\matr{G}$, we have
$\matr{H}^\matr{G} \simeq [\matr{C}_{\times}(\matr{G}),\matr{H}].$
}\end{pro}
\begin{proof}{
Let $V(\matr{G})=\{v_{_1}, v_{_2}, \dots , v_{_n}\}$ and $|V(\matr{H})|=m$. There is a canonical graph isomorphism between two constructions,
$$\theta : \matr{H}^{\matr{G}} \rightarrow [\matr{C}_{\times}(\matr{G}),\matr{H}],$$
where for any $f \in V(\matr{H}^{\matr{G}})$ we define
$$\theta(f)=(f(v_{_1}), f(v_{_2}), ...,f(v_{_n})) \in V([\matr{C}_{\times}(\matr{G}),\matr{H}]).$$
Now, if $fg \in V(\matr{H}^{\matr{G}})$, then by definition $\forall vw\in E(\matr{G}),\ f(v)g(w) \in E(\matr{H})$, hence in the $\matr{H}$ we can find a cylinder $\matr{C}_{\times}(\matr{G})$ which $\matr{B}^-=(f(v_{_1}), f(v_{_2}), ...,f(v_{_n}))$, and $\matr{B}^+=(g(v_{_1}), g(v_{_2}), ...,g(v_{_n}))$. Thus we have an edge $\theta(f) \theta(g) \in E([\matr{C}_{\times}(\matr{G}),\matr{H}]).$
So $\Theta$ is a graph homomorphism and it is easy to check that it is an graph isomorphism.
}\end{proof}

 \begin{figure}[ht]
\centering{\includegraphics[width=6.5cm]{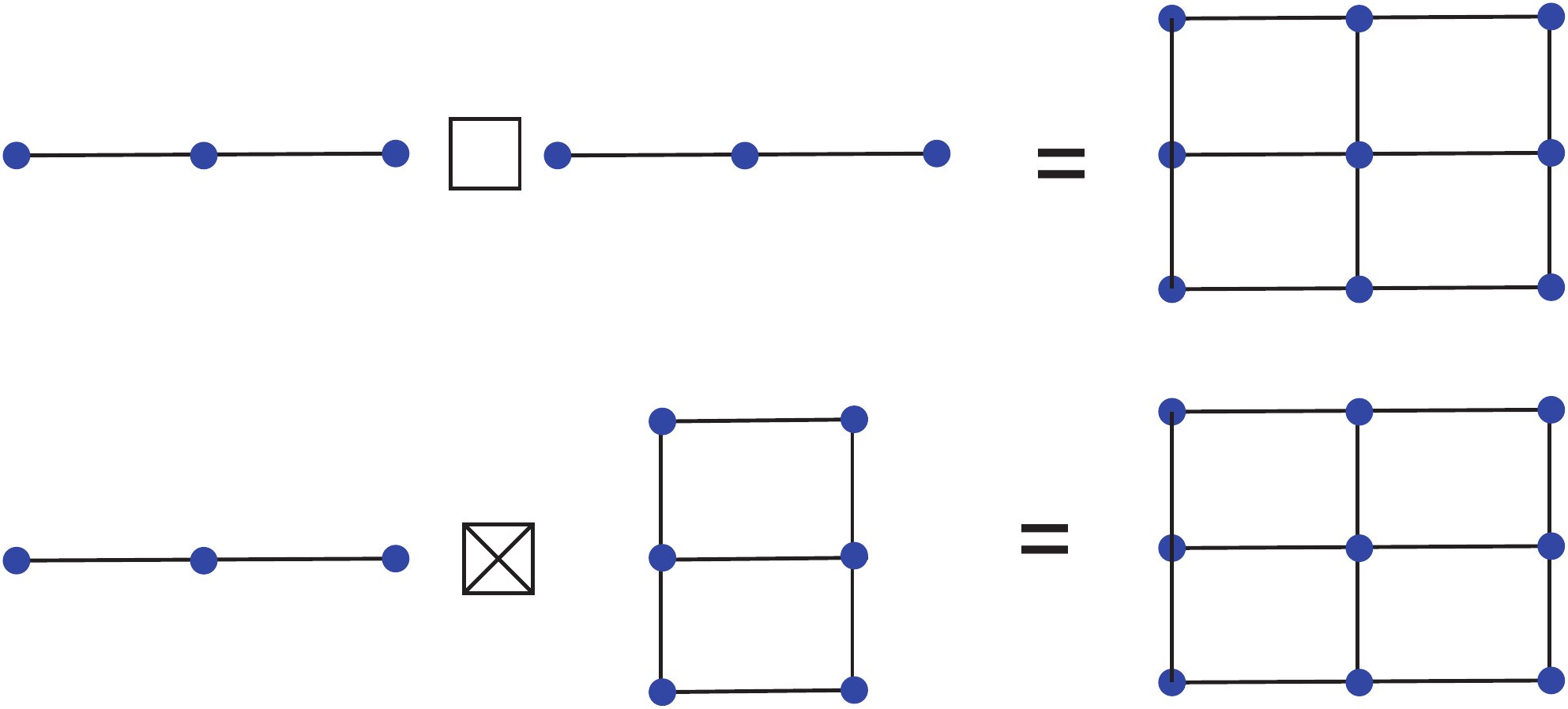}}
\caption{The Cartesian product of two 2-paths and the corresponding cylindrical construction.}
\label{fig:cartesian}
\end{figure}

 \begin{figure}[ht]
\centering{
\includegraphics[width=6.5cm]{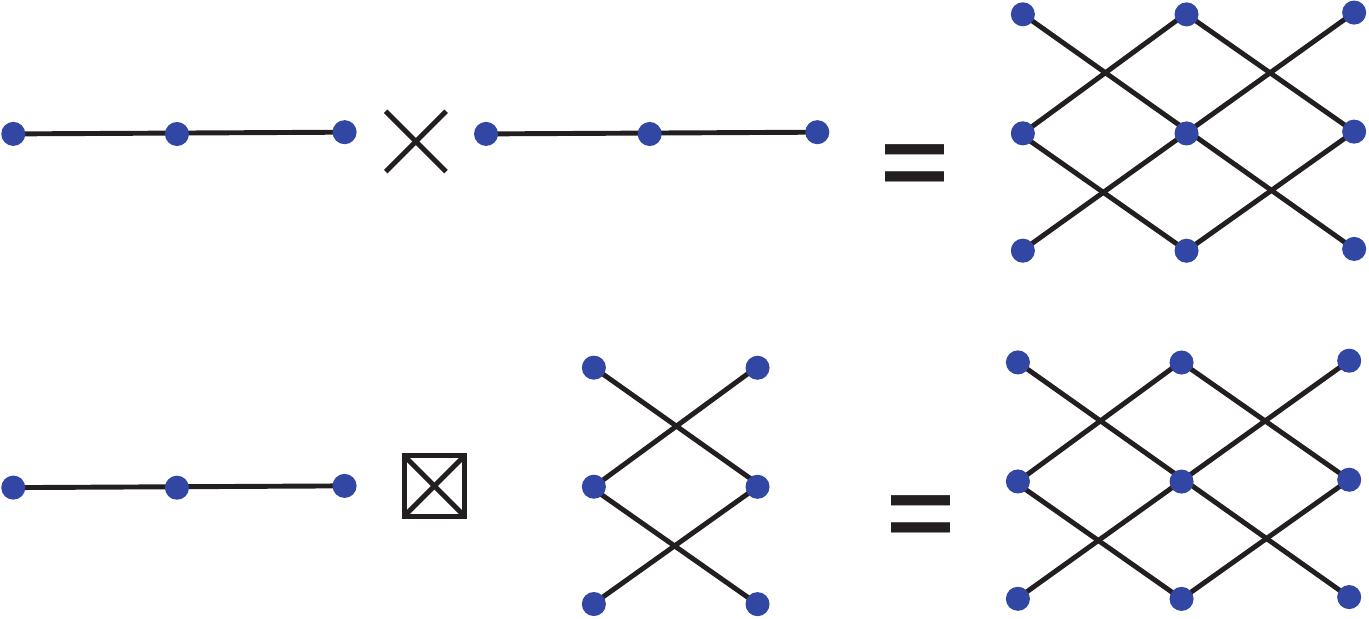}
}
\caption{The categorical product of two 2-paths and the corresponding cylindrical construction.}
\label{fig:catproduct}
\end{figure}
Note that similarly one may show that the strong product and the lexicographical product can be described by cylinders $\matr{C}_{\cyl{}}(\matr{G})$ and $\matr{C}_{\bullet}(\matr{G})$, respectively.
Their bases are isomorphic to $\matr{G}$, and their edges are described as follow (see \cite{imrich} and Figure~\ref{fig:strong}),
$$i=j \quad \mathrm{or} \quad v_{_{i}} v_{_{j}} \in E(\matr{G}) \Leftrightarrow \quad y_{_{i}}z_{_{j}} \in E(\matr{C}_{_{\bullet}}(\matr{G})(\by,\bz)) ,$$
$$i,j \in \nat[k]\quad \Leftrightarrow \quad y_{_{i}}z_{_{j}} \in E(\matr{C}_{_{\cyl{}}}(\matr{G})(\by,\bz)).$$
\begin{figure}[ht]
\centering{\includegraphics[width=9cm]{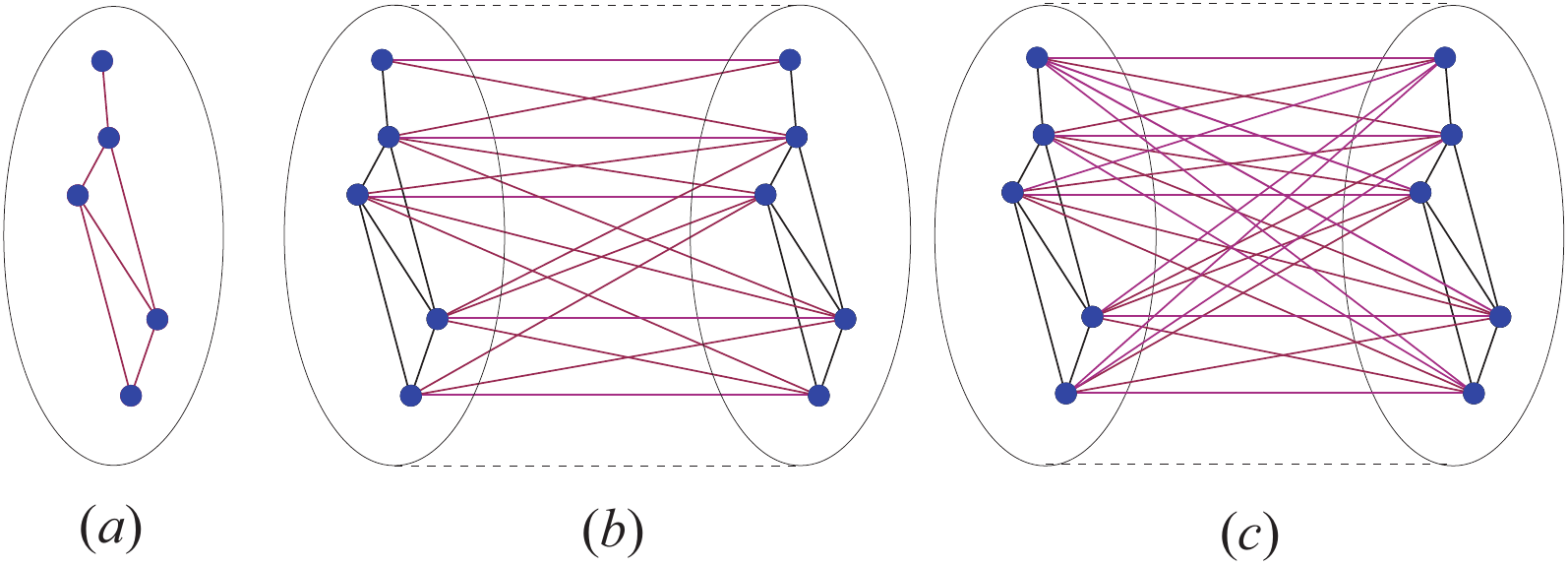}}
\caption{$(a)$ The graph $\matr{G}$, $(b)$ The lexicographical product cylinder $\matr{C}_{\bullet}(\matr{G})$, $(c)$ The strong product cylinder $\matr{C}_{\cyl{}}(\matr{G})$.}
\label{fig:strong}
\end{figure}

 \begin{exm}{{ \bf The zig-zag product of symmetric graphs} \\
In this example we show that the zig-zag product of graphs can also be presented as a cylindrical construct. We introduce two different presentations of this construction in what follows (e.g. see \cite{HENE} for more details and the background).\\
Let $(\matr{G},\ell_{_{\matr{G}}}: E(\matr{G}) \longrightarrow D^2)$ be a symmetric $d$-regular labeled graph on the vertex set $V(\matr{G})$, whose edges are labeled by the set $D^2$ where $D \isdef \nat[d]$ in a way that
$$\forall\ v \in V(\matr{G}), \quad \{\ell_{_{\matr{G}}}(e) \ \ | \ \ e^-=v\} = D.$$
Then, any such graph can be represented by a {\it rotation} map as
$${\rm Rot_{_{\matr{G}}}} : V(\matr{G}) \times D \longrightarrow V(\matr{G}) \times D,$$
for which
$${\rm Rot_{_{\matr{G}}}}(u,i)=(v,j) \ \ \Leftrightarrow \ \ \exists\ e=uv \in \ E(\matr{G}) \quad \ell_{_{\matr{G}}}(e)=(i,j).$$
Let $\matr{G}$ be a symmetric $d$-regular graph on the vertex set $V(\matr{G})$, and also, let $\matr{H}$ be a symmetric $k$-regular graph
on the vertex set $D$, respectively, given by rotation maps ${\rm Rot_{_{\matr{G}}}}$ and ${\rm Rot_{_{\matr{H}}}}$. Then the zig-zag product of $\matr{G}$ and $\matr{H}$, denoted by $\matr{G} \otimes_{_{z}} \matr{H}$ is a symmetric $k^2$-regular graph defined on the vertex set
$V(\matr{G}) \times D$ for which
$${\rm Rot_{_{\matr{G} \otimes_{_{z}} \matr{H}}}}((v,h),(i,j))=((u,l),(j',i')),$$
if and only if
$${\rm Rot_{_{\matr{H}}}}(h,i)=(h',i'), \quad {\rm Rot_{_{\matr{G}}}}(v,h')=(u,l'), \quad {\rm and} \ \ {\rm Rot_{_{\matr{H}}}}(l',j)=(l,j').$$
A {\it vertex transitive zig-zag product} $\matr{G} \otimes_{_{z}} \matr{H}$ is such a product when $\matr{H}$ is a vertex transitive graph.
\begin{figure}[ht]
\centering{
\begin{tabular}{c c}
\includegraphics[width=3.5cm]{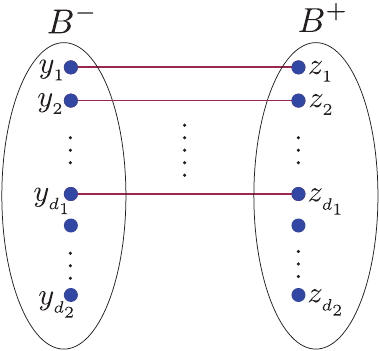} & \includegraphics[width=3.5cm]{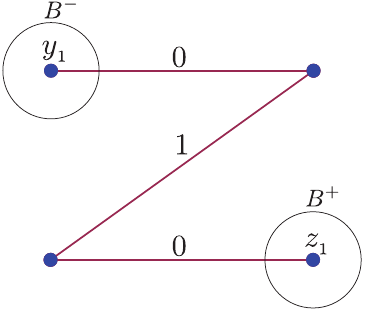}\\
\vspace{-0.5cm}\cr {$(a)$} & {$(b)$}
\end{tabular}}
\caption{$(a)$ The zig-zag product cylinder, $(b)$ The zig-zag quotient cylinder.}
\label{fig:zigzag_cylinder}
\end{figure}
\begin{itemize}
\item{{\bf The vertex transitive zig-zag product as a mixed construct}\\
Let $\matr{H}$ be a vertex transitive symmetric graph on the vertex set $V(\matr{H})=\{v_{_{1}},v_{_{2}},\ldots,v_{_{d}}\}$ given by the rotation map ${\rm Rot}_{_{\matr{H}}}$ and for each $i \in \nat[d]$ fix an automorphism of $\matr{H}$ as $\sigma_{_{i}}$ such that $\sigma_{_{i}}(v_{_{1}})=v_{_{i}}$ with $\sigma_{_{1}}$ equal to the identity automorphism. Let $\gamma_{_{i}} \in {\bf S}_{_{d}}$ be the permutation induced on the index set through $\sigma_{_{i}}$ and define
$$\Gamma \isdef \langle \gamma_{_{1}},\gamma_{_{2}},\ldots,\gamma_{_{d}} \rangle$$
as the subgroup of ${\bf S}_{_{d}}$ generated by the permutations $\gamma_{_{i}}$'s.
Also, define the cylinder
$\matr{C}_{_{\matr{H}}}^{^R}(\by,\bz)$ (see Figure~\ref{fig:zigzag_prod}) labeled by $\{0,1\}$ whose bases are two isomorphic copies of $\matr{H}$ marked by the vertices $\{y_{_{1}},y_{_{2}},\ldots,y_{_{d}}\}$ and $\{z_{_{1}},z_{_{2}},\ldots,z_{_{d}}\}$ and labeled by $0$, respectively, where the only other edge is the simple edge $y_{_{1}}z_{_{1}}$ labeled by $1$.\\
Moreover, let the graph $\matr{G}$ be defined through the rotation map
$${\rm Rot_{_{\matr{G}}}} : V(\matr{G}) \times D \longrightarrow V(\matr{G}) \times D,$$
and define the $(\Gamma,1)$-graph $\widehat{\matr{G}}$ as a graph with the same vertex and edge sets as $\matr{G}$
and with the following labeling,
$$\ell_{_{\widehat{\matr{G}}}}(uv)=(\gamma_{_{i}},\matr{C}_{_{\matr{H}}}^{^R},\gamma_{_{j}}) \quad \Leftrightarrow
\quad {\rm Rot_{_{\matr{G}}}}(u,i)=(v,j). $$
Now, the graph $\widehat{\matr{G}}\ \cyl{\Gamma} \matr{C}_{_{\matr{H}}}^{^R}$ is usually called the {\it replacement product} of the graphs $\matr{G}$ and $\matr{H}$. On the other hand, it is not hard to verify that the exponential graph $[\matr{Z},\widehat{\matr{G}} \cyl{\Gamma} \matr{C}_{_{\matr{H}}}^{^R}]_{_{\Gamma}}$ is actually the zig-zag product $\matr{G} \otimes_{_{z}} \matr{H}$, where $\matr{Z}$ is the zig-zag cylinder
depicted in Figure~\ref{fig:zigzag_prod}.
}
\item{{\bf The vertex transitive zig-zag product as a cylindrical product}\\
It is not hard to see that the above construction can also be described as a pure cylindrical product. For this, using the above setup, define
the cylinder $\matr{C}_{_{\matr{H}}}^{^Z}(\by,\bz)$ by adding edges to $\matr{C}_{_{\matr{H}}}^{^R}(\by,\bz)$ such that
$$y_{_{i}}z_{_{j}} \in E(\matr{C}_{_{\matr{H}}}^{^Z}) \quad \Leftrightarrow \quad
{\rm Rot_{_{\matr{H}}}}(y_{_{i}},l)=(y_{_{1}},l') \ \ {\rm and} \ \ {\rm Rot_{_{\matr{H}}}}(z_{_{1}},h)=(z_{_{j}},h').$$
Then one may verify that the cylindrical product $\widehat{\matr{G}}\ \cyl{\Gamma} \matr{C}_{_{\matr{H}}}^{^Z}$ is essentially the same as the zig-zag product $\matr{G} \otimes_{_{z}} \matr{H}$.
}
\end{itemize}
}\end{exm}
\begin{figure}[ht]
\begin{center}
\begin{tabular}{c c c }
\includegraphics{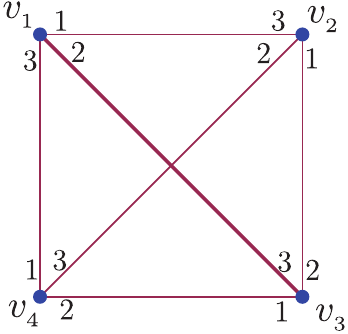} & \includegraphics{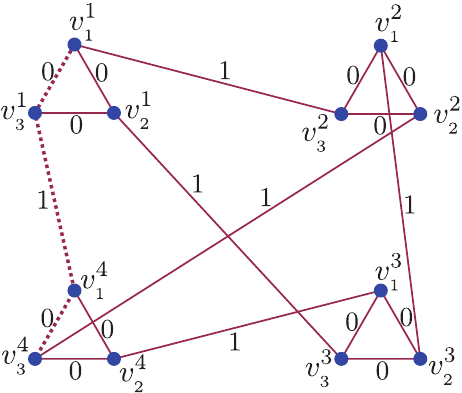} & \includegraphics{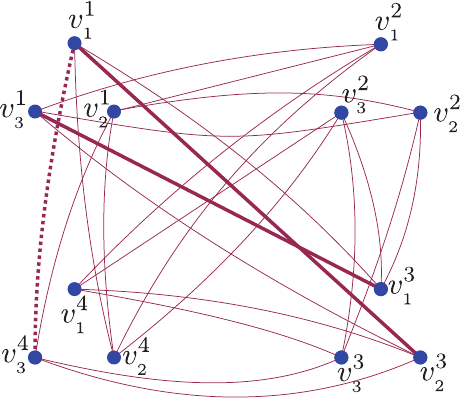}\\
\vspace{-0.5cm}\cr {$(a)$} & {$(b)$} & {$(c)$}
\end{tabular}
\end{center}
\caption{$(a)$ $\matr{G}_{_{1}}$ : zig-zag base graph in cylindrical product case, $(b)$ The zig-zag base graph in cylindrical quotient case, $(c)$ Resulting zig-zag graph. The bold lines are representing a copy of zig-zag product cylinder placed on the edge $(v_{_{1}},v_{_{3}})$ with appropriate label. The dashed line represents a homomorphism from zig-zag quotient cylinder to the zig-zag base graph.}
\label{fig:zigzag_prod}
\end{figure}

 Also, as a direct consequences of the definitions we have
\begin{prepro}\label{pro:cylinder}{
For any pair of reduced simple graphs $\matr{G}$ and $\matr{H}$, there exists a
vertex-surjective homomorphism $\sigma:\matr{G} \rightarrow \matr{H}$
if and only if there exists a coherent set of plain cylinders $\matr{C}=\{\matr{C}_i\}$, such that $\matr{G} \cong {\rm red}(\matr{H}_{_\ell} \cyl{} \matr{C})$, where $\matr{H}_{_\ell}$ is $\matr{H}$ properly labeled according to the label of cylinders.

 }\end{prepro}
\begin{proof}{
Consider the quotient graph and note that the graph induced on any two inverse images of $\sigma$ is a plain cylinder.
}\end{proof}

 This observation shows that the cylindrical construction generalizes the concept of a covering or a homomorphic pre-image in a fairly broad sense.
Hence, it is important to find out about the relationships between the construction and the homomorphism problem, which is the subject of the next section.

\section{Main duality theorem }\label{sec:main}
The tensor home dualities, although not observed widely, are classic in graph theory and at least goes back to the following result
of Pultr\footnote{Here one must be careful about the choice of categories (also see Section~\ref{sec:adjred} and \cite{adjoint}).} (see Examples~\ref{exm:exp} and \ref{exm:cylop}).

 \begin{alphthm}{{\rm \cite{ladjoint, pultr}} For any pultr template $\tau$, functors $\Lambda_{\tau}$ and $\Gamma_{\tau}$ are left and right adjoints in ${\bf Grph}_{_{\leq}}$, i.e. for any two graphs $\matr{G}, \matr{K}$ there exists a homomorphism $\Lambda_{\tau}(\matr{G}) \rightarrow \matr{K}$ if and only if there exists a homomorphism $\matr{G} \rightarrow \Gamma_{\tau}(\matr{K})$.
}\end{alphthm}

 In this section we state the fundamental theorem stating a duality between the exponential and cylindrical constructions, which shows that even in the presence of twists and cylinders with mixed bases one may also deduce a correspondence.
\begin{thm}{\label{thm:maino}\textbf{The fundamental duality}\\
Let $\matr{C}=\{\matr{C}^{^{j}}(\by^{^{j}},\bz^{^{j}},\J^{^{j}}) \ \ | \ \ j \in \nat[m]\}$ be a $\Gamma$-coherent set of $(t,k)$-cylinders, and also let $\matr{G}$ be a $(\Gamma,m)$-graph.
Then for any labeled graph $\matr{H}$,
\begin{itemize}
\item[{\rm a)}]{ $\lhom(\matr{G} \cyl{\Gamma} \matr{C},\matr{H}) \not = \emptyset \quad \Leftrightarrow \quad \chom(\matr{G},[\matr{C},\matr{H}]_{_{\Gamma}}) \not = \emptyset.$}
\item[{\rm b)}]{There exist a retraction
$$r_{_{\matr{G},\matr{H}}}: \lhom(\matr{G} \cyl{\Gamma} \matr{C},\matr{H}) \rightarrow \chom(\matr{G},[\matr{C},\matr{H}]_{_{\Gamma}})$$
and a section
$$s_{_{\matr{G},\matr{H}}}: \chom(\matr{G},[\matr{C},\matr{H}]_{_{\Gamma}}) \rightarrow \lhom(\matr{G} \cyl{\Gamma} \matr{C},\matr{H}),$$
such that
$r_{_{\matr{G},\matr{H}}} \circ s_{_{\matr{G},\matr{H}}} = {\bf 1}_{_{_{ \lhom(\matr{G} \cyl{\Gamma} \matr{C},\matr{H})}}}$, where ${\bf 1}_{_{_{ \lhom(\matr{G} \cyl{\Gamma} \matr{C},\matr{H})}}}$ is the identity mapping.}
\end{itemize}
}\end{thm}
\begin{proof}{\noindent ($\Rightarrow$) If $\lhom(\matr{G} \cyl{\Gamma} \matr{C},\matr{H}) \not = \emptyset$, then there exists a homomorphism $$\sigma=(\sigma_{_{V}},\sigma_{_{E}}) \in \lhom(\matr{G} \cyl{\Gamma} \matr{C},\matr{H}).$$
Assume that $V(\matr{G}) = \{v_{_{1}},v_{_{2}},\ldots,v_{_{n}}\}$ and
$$\forall\ i \in \nat[n] \quad \sigma_{_{V}}(\bv^i) \isdef \bu^i,$$
where $\bv ^i=(\bv^i_1, \bv^{i}_{2}, \dots , \bv^i _k)$ is the set of vertices that blow up at $v_i$, and $$\bu^i = (\sigma_{_V}(\bv^i_1), \sigma_{_V}(\bv^i_2), \dots , \sigma_{_V}(\bv^i_k)).$$ Now, define $r_{_{\matr{G},\matr{H}}}((\sigma_{_{V}},\sigma_{_{E}}))=(\sigma'_{_{V}},\sigma'_{_{E}})$, where $\sigma'_{_{V}}: V(\matr{G}) \longrightarrow V([\matr{C},\matr{H}]_{_{\Gamma}})$ and
$$\forall\ i \in \nat[n] \quad \sigma'_{_{V}}(v_{_i}) \isdef \langle \bu^i \rangle_{_{\Gamma}}.$$
Note that by the definition of the exponential graph construction,
for each $i \in \nat[n]$, there exists a unique $\alpha_{_{i}} \in \Gamma$ such that $\langle \bu^i \rangle_{_{\Gamma }}=\bu^i \alpha_{_{i}}$, hence we have
$$\forall\ i \in \nat[n] \quad \sigma'_{_{V}}(v_{_i}) \isdef \bu^i \alpha_{_{i}} .$$
Again, for each $e=v_{_i}v_{_j} \in E(\matr{G})$ with $\iota(e)=v_{_i}$, $\tau(e)=v_{_j}$ and $\ell_{_{\matr{G}}}(e)=(\gamma^-,t,\gamma^+)$, define,
$$\sigma'_{_{E}}(e) \isdef e'=\langle \bu^i \rangle\langle \bu^j \rangle,$$
such that
$$\iota(e')=\langle \bu^i \rangle_{_{\Gamma }}=\bu^i \alpha_{_{i}}, \ \ \tau(e')=\langle \bu^j \rangle_{_{\Gamma }}=\bu^j \alpha_{_{j}}, \ \
\ell_{_{[\matr{C},\matr{H}]_{_{\Gamma}}}}(e') \ = \ (\alpha_{_{i}}\gamma^-,t,\alpha_{_{j}}\gamma^+).$$
Now, we show that the mapping $\sigma'=(\sigma'_{_{V}},\sigma'_{_{E}})$ is well defined and is a $(\Gamma,m)$-homomorphism in $\chom(\matr{G},[\matr{C},\matr{H}]_{_{\Gamma}})$.
First, note that since $\ell_{_{\matr{G}}}(e)=(\gamma^-,t,\gamma^+)$, by cylindrical construction, the restriction of $\sigma$
to the cylinder on this edge $e=v_{_i}v_{_j}$, gives a homomorphism
$$\sigma^{e}=(\sigma_{_{V}}^{e},\sigma_{_{E}}^{e}) \in \lhom (\matr{C}^{^t}(\by^{^{t}}\gamma^-,\bz^{^{t}}\gamma^+,(\gamma^-, \gamma^+)\J^{^{t}}),\matr{H}),$$
such that
$$\sigma_{_{V}}^{e}(\by^{^{t}}\gamma^-)=\bu^i \quad {\rm and} \quad \sigma_{_{V}}^{e}(\bz^{^{t}}\gamma^+)=\bu^j.$$
Hence, by applying the twists $(\alpha_{_{i}} , \alpha_{_{j}})$ we obtain a new homomorphism
$$\tilde{\sigma}^{e}=(\tilde{\sigma}_{_{V}}^{e},\tilde{\sigma}_{_{E}}^{e}) \in \lhom (\matr{C}^{^t}(\by^{^{t}}\alpha_{_{i}}\gamma^-,\bz^{^{t}}\alpha_{_{j}}\gamma^+, (\gamma^-, \gamma^+)\J^{^{t}}),\matr{H}),$$
such that
$$\tilde{\sigma}_{_{V}}^{e}(\by^{^{t}}\alpha_{_{i}}\gamma^-)=\bu^i \alpha_{_{i}}=\langle \bu^i \rangle_{_{\Gamma }},
\quad \tilde{\sigma}_{_{V}}^{e}(\bz^{^{t}}\alpha_{_{j}}\gamma^+)=\bu^j \alpha_{_{j}}=\langle \bu^j \rangle_{_{\Gamma }},$$
and by the definition of the exponential graph, this shows that there is an edge $\tilde{e} \isdef \langle \bu^i \rangle_{_{\Gamma }}\langle \bu^j \rangle_{_{\Gamma }} \in E([\matr{C},\matr{H}]_{_{\Gamma}})$ with the label
$\ell_{_{[\matr{C},\matr{H}]_{_{\Gamma}}}}(\tilde{e}) \ = \ (\alpha_{_{i}}\gamma^-,t,\alpha_{_{j}}\gamma^+).$ This shows that $\sigma'$ is a graph homomorphism. \\
Also, by Proposition~\ref{CHOMLEMMA}, it is clear that
$$\sigma'=(\sigma'_{_{V}},\sigma'_{_{E}}) \in \chom(\matr{G},[\matr{C},\matr{H}]_{_{\Gamma}}).$$
\ \\
($\Leftarrow$)
Fix a homomorphism $\sigma'=(\sigma'_{_{V}},\sigma'_{_{E}}) \in \chom(\matr{G},[\matr{C},\matr{H}]_{_{\Gamma}})$ and assume that
$$\forall\ i \in \nat[n] \quad \sigma'_{_{V}}(v_{_i})=\langle \bu^i \rangle_{_{\Gamma }} \in V([\matr{C},\matr{H}]_{_{\Gamma}}).$$
Also, by Proposition~\ref{CHOMLEMMA}, we know that for each $i \in \nat[n]$ there exists $\alpha_{_{i}} \in \Gamma$ such that for any edge $e$ intersecting $v_{_{i}}$ we have,
$$\ell^{v_{_{i}}}_{_{\matr{G}}}(e)=\alpha_{_{i}}[\ell^{\langle \bu^i \rangle_{_{\Gamma }}}_{_{ [\matr{C},\matr{H}]_{_{\Gamma}} }}(\sigma'_{_{E}}(e))].$$
Again, fix an edge $e=v_{_i}v_{_j} \in E(\matr{G})$ with the label $\ell_{_{\matr{G}}}(e)=(\gamma^-,t,\gamma^+)$ and note that since $\sigma'$ is a homomorphism, then
$$\sigma'_{_{E}}(e) \in E([\matr{C},\matr{H}]_{_{\Gamma}}) \quad {\rm and}
\quad \ell_{_{[\matr{C},\matr{H}]_{_{\Gamma}}}}(e)=(\alpha_{_{i}}\gamma^-, t, \alpha_{_{j}}\gamma^+).$$
Therefore, by the definition of the exponential graph, we have a homomorphism
$$\tilde{\sigma}^{e}=(\tilde{\sigma}_{_{V}}^{e},\tilde{\sigma}_{_{E}}^{e}) \in \lhom (\matr{C}^{^t}(\by^{^{t}}\alpha_{_{i}}\gamma^-,\bz^{^{t}}\alpha_{_{j}}\gamma^+, (\gamma^-, \gamma^+)\J^{^{t}}),\matr{H}),$$
and consequently, we obtain a homomorphism
$$\sigma^{e}=(\sigma_{_{V}}^{e},\sigma_{_{E}}^{e}) \in \lhom (\matr{C}^{^t}(\by^{^{t}}\gamma^-,\bz^{^{t}}\gamma^+, (\gamma^-, \gamma^+)\J^{^{t}}),\matr{H}),$$
such that
$$\sigma_{_{V}}^{e}(\by^{^{t}}\gamma^-)=\bu^i \alpha_{_{i}} \quad {\rm and} \quad \sigma_{_{V}}^{e}(\bz^{^{t}}\gamma^+)=\bu^j \alpha_{_{j}}.$$
Now, by Proposition~\ref{CHOMLEMMA}, these homomorphisms are compatible at each vertex and we may define a global homomorphism
$$s_{_{\matr{G},\matr{H}}}((\sigma'_{_{V}},\sigma'_{_{E}})) = (\sigma_{_{V}},\sigma_{_{E}})\isdef \bigcup_{e \in E(\matr{G})} \sigma^{e} \in \lhom(\matr{G} \cyl{\Gamma} \matr{C},\matr{H}). $$
It remains to prove that $$\forall \sigma' \in \chom(\matr{G},[\matr{C},\matr{H}]_{_{\Gamma}}) \quad r_{_{\matr{G},\matr{H}}}(s_{_{\matr{G},\matr{H}}}(\sigma')) = \sigma',$$
but by definitions,
$$ r_{_{\matr{G},\matr{H}}}(s_{_{\matr{G},\matr{H}}}(\sigma'))(e)=r_{_{\matr{G},\matr{H}}}(\sigma^{e})(e)= \sigma'(e),$$
since $\langle \bu^i \alpha_i \rangle_{_{\Gamma }} = \langle \bu^i \rangle_{_{\Gamma}}$.
Also, for any $e \in E(\matr{G})$ whose label is $(\gamma^-,t,\gamma^+)$, by definitions
we know that $$r_{_{\matr{G},\matr{H}}}(s_{_{\matr{G},\matr{H}}}(\sigma'_{_{E}}))(e) = e',$$
where
$$\iota(e')=\langle \bu^i \rangle_{_{\Gamma }}=\bu^i \alpha_{_{i}}, \ \ \tau(e')=\langle \bu^j \rangle_{_{\Gamma }}=\bu^j \alpha_{_{j}}, \ \
\ell_{_{[\matr{C},\matr{H}]_{_{\Gamma}}}}(e') \ = \ (\alpha_{_{i}}\gamma^- ,t,\alpha_{_{j}}\gamma^+),$$
and by Proposition~\ref{CHOMLEMMA} it is clear that
$$r_{_{\matr{G},\matr{H}}}(s_{_{\matr{G},\matr{H}}}(\sigma'_{_{E}}))(e) = \sigma'_{_{E}}(e)$$
and consequently,
$$r_{_{\matr{G},\matr{H}}}(s_{_{\matr{G},\matr{H}}}(\sigma'_{_{E}})) = \sigma'_{_{E}}.$$
}\end{proof}
\begin{figure}[ht]
\includegraphics[width=11cm]{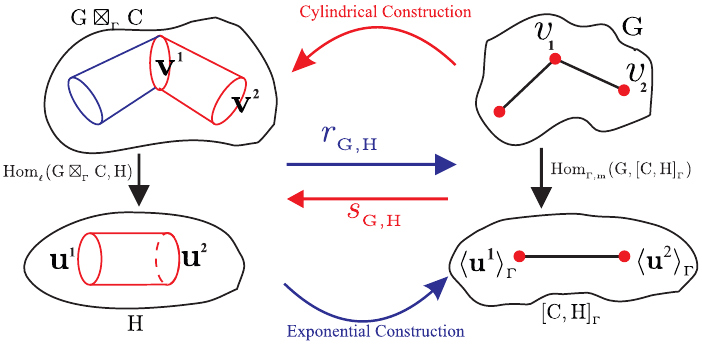}\centering{
\caption{Schematic diagram for the proof of Theorem \ref{thm:maino}. }}
\label{fig:main} \end{figure}
Similarly, one may prove the following version of this theorem for symmetric graphs and their homomorphisms.

 \begin{thm}\label{SMAINTHM}
Let $\matr{C}=\{\matr{C}^{^{j}}(\by^{^{j}},\bz^{^{j}},\J^{^{j}}) \ \ | \ \ j \in \nat[m]\}$ be a set
of $\Gamma$-coherent symmetric $(t,k)$-cylinders, and
also, let $\matr{G}$ be a symmetric $(\Gamma,m)$-graph.
Then for any labeled symmetric graph $\matr{H}$,
\begin{itemize}
\item[{\rm a)}]{ $\slhom(\matr{G} \scyl{\Gamma} \matr{C},\matr{H}) \not = \emptyset \quad \Leftrightarrow \quad \schom(\matr{G},[\matr{C},\matr{H}]^{^{s}}_{_{\Gamma}}) \not = \emptyset.$}
\item[{\rm b)}]{There exist a retraction
$$r_{_{\matr{G},\matr{H}}}: \slhom(\matr{G} \scyl{\Gamma} \matr{C},\matr{H}) \rightarrow \schom(\matr{G},[\matr{C},\matr{H}]^{^{s}}_{_{\Gamma}})$$
and a section
$$s_{_{\matr{G},\matr{H}}}: \schom(\matr{G},[\matr{C},\matr{H}]^{^{s}}_{_{\Gamma}}) \rightarrow \slhom(\matr{G} \scyl{\Gamma} \matr{C},\matr{H}) $$
such that
$r_{_{\matr{G},\matr{H}}} \circ s_{_{\matr{G},\matr{H}}} = {\bf 1}$, where ${\bf 1}$ is the identity mapping.}
\end{itemize}
\end{thm}

Here we go through some basic observations.

 \begin{pro}{\label{pro:basics1}
For any graph $\matr{G}$ and positive integers $r,s,t$ and $k$, we have
\begin{itemize}
\item[{\rm 1)}] {
$[\matr{P}_{_{2k-1}} ,\matr{C}_{_{2k+1}} ]\simeq \matr{K}_{_{2k+1}}.$
}
\item[{\rm 2)}] $(\matr{G}^r)^s\homeq \matr{G}^{rs}.$
\item [{\rm 3)}]For any two $(t,k)$-cylinders $\matr{A}$ and $\matr{B}$ and any graph $\matr{H}$, we have
$$[\matr{A}+\matr{B},\matr{H}] \rightarrow [\matr{A},\matr{H}] \times [\matr{B},\matr{H}],$$
where $\times$ is the categorical product of two graphs and amalgamation is according to the standard labeling of cylinders.
\end{itemize}
}\end{pro}

 \begin{proof}{
\
\begin{enumerate}
\item Note that for any two fixed vertices $x$ and $y$ in $\matr{C}_{_{2k+1}}$, there exists a path of odd length $d \leq 2k-1$ between $x$ and $y$ in $\matr{C}_{_{2k+1}}$.
\item Note that $(\matr{G}^r)^s \simeq [\matr{P}_s,[\matr{P}_r,\matr{G}]] $ and $(\matr{G}\boxtimes \matr{P}_s) \boxtimes \matr{P}_r \simeq (\matr{G}\boxtimes \matr{P}_r) \boxtimes \matr{P}_s \simeq \matr{G}\boxtimes \matr{P}_{rs}$. Hence, by
Theorem~\ref{thm:maino} for any graph $\matr{H}$ we have,
$$
\begin{array}{lll}
\hom(\matr{H},(\matr{G}^r)^s) \neq \emptyset &\Leftrightarrow
\hom(\matr{H},[\matr{P}_s,[\matr{P}_r,\matr{G}]]) \neq \emptyset & \Leftrightarrow
\hom(\matr{H}\cyl{}\matr{P}_s,[\matr{P}_r,\matr{G}]) \neq \emptyset
\\
& \Leftrightarrow
\hom((\matr{H}\cyl{}\matr{P}_s)\cyl{}\matr{P}_r,\matr{G}) \neq \emptyset
&\Leftrightarrow
\hom(\matr{H}\cyl{}\matr{P}_{rs},\matr{G}) \neq \emptyset
\\
&\Leftrightarrow
\hom(\matr{H},[\matr{P}_{rs},\matr{G}]) \neq \emptyset &\Leftrightarrow \hom(\matr{H},\matr{G}^{rs}) \neq \emptyset.
\end{array}
$$
Now use these equivalences once for $\matr{H}=(\matr{G}^r)^s$ to get $\hom((\matr{G}^r)^s,\matr{G}^{rs}) \neq \emptyset$ and once for $\matr{H}=\matr{G}^{rs}$, to get $\hom(\matr{G}^{rs},(\matr{G}^r)^s) \neq \emptyset$.
\item We use Theorem~\ref{thm:maino} and Proposition \ref{pro:plusu} as follows,
$$
\begin{array}{lll}
\matr{G} \rightarrow [\matr{A}+\matr{B},\matr{H}] & \xLeftrightarrow{\mathrm{Theorem~\ref{thm:maino}}} &\matr{G}\boxtimes(\matr{A}+\matr{B}) \rightarrow \matr{H}\\ & \xLeftrightarrow{\mathrm{Proposition \ \ref{pro:plusu}}}& \matr{G}\boxtimes \matr{A} +\matr{G}\boxtimes \matr{B} \rightarrow \matr{H}\\
&\xRightarrow{\mathrm{Pushout\ properties}} &\matr{G}\boxtimes \matr{B} \rightarrow \matr{H}\ \mathrm{and} \ \matr{G}\boxtimes \matr{A} \rightarrow \matr{H} \\
& \xLeftrightarrow{\mathrm{Theorem~\ref{thm:maino}}}& \matr{G} \rightarrow [\matr{A},\matr{H}]\ \mathrm{and} \ \matr{G} \rightarrow [\matr{B},\matr{H}] \\
& \xLeftrightarrow{\mathrm{Categorical\ product \ properties}}& \matr{G}\rightarrow [\matr{A},\matr{H}] \times [\matr{B},\matr{H}].
\end{array}
$$
Now, set $\matr{G}\isdef [\matr{A}+\matr{B},\matr{H}]$,
and consequently, $[\matr{A}+\matr{B},\matr{H}]\rightarrow [\matr{A},\matr{H}] \times [\matr{B},\matr{H}].$
\end{enumerate}
}\end{proof}

 Definitely one may consider a category of cylinders and maps between them and consider this category and its basic properties.
We will not delve into the details of this in this article.
\section{A categorical perspective}\label{sec:cat}
Let $\matr{C}=\{\matr{C}^{^{j}}(\by^{^{j}},\by^{^{j}},\J^{^{j}}) \ \ | \ \ j \in \nat[m]\}$ be a $\Gamma$-coherent set of $(t,k)$-cylinders.
We show that the cylindrical construction $- \cyl{\Gamma} \matr{C} : {\bf LGrph}(\Gamma,m) \longrightarrow {\bf LGrph}$ and the exponential graph
construction $[\matr{C}, -]_{_{\Gamma}} : {\bf LGrph} \longrightarrow {\bf LGrph}(\Gamma,m)$ introduce well-defined functors.
To see this,
For two objects $ \matr{G}$ and $ \matr{G}'$ in $ Obj({\bf LGrph}(\Gamma,m))$, where
$$V( \matr{G}) =\{v_{_1}, v_{_2}, . . . , v_{_n}\}, \quad V ( \matr{G}') =\{u_{_1}, u_{_2}, . . . , u_{_n}\}$$
consider the following definition of the cylindrical construction functor,
$$\forall\ \matr{G} \in Obj({\bf LGrph}(\Gamma,m)) \quad (- \cyl{\Gamma} \matr{C})(\matr{G}) \isdef \matr{G} \cyl{\Gamma} \matr{C},$$
$$\forall\ f \in \chom(\matr{G},\matr{G}') \quad (- \cyl{\Gamma} \matr{C})(f) \isdef f \cyl{\Gamma} \matr{C} \in \lhom(\matr{G} \cyl{\Gamma} \matr{C},\matr{G}' \cyl{\Gamma} \matr{C}),$$
defined as follows, where $ \{\alpha_{_{1}},\alpha_{_{2}},\ldots,\alpha_{_{n}}\} \subseteq \Gamma,$
is provided by Proposition~\ref{CHOMLEMMA} with respect to $f$,
$$(f \cyl{\Gamma} \matr{C})_{_{V}}(x) \isdef \left \{ \begin{array}{lll}

 \bu^{^j}_{_t}\alpha_{_{i}} & x=\bv^{^i}_{_t},\ f(\bv^{^i})=\bu^{^j} \quad &(\bv^{^j}\textrm{:blowed up vertices for }v_{i}\\

 & &\textrm{ \ } \bu^{^j} \textrm{: blowed up vertices for }u_{j}), \\
w^{e'}_{i} & x=w^e_{i},\ f(e)=e' \quad & \textrm{(Internal vertices of cylinders)},
\end{array}\right. $$
where as mentioned before $\bw^{{e}} \isdef (w_{_{1}}^{{e}},w_{_{2}}^{{e}},\ldots,w_{_{r}}^{{e}})$ consisting of internal vertices that appear in $\matr{G}\cyl{\Gamma} \matr{C}$ when the edge $e$ is replaced by the corresponding cylinder. Also (see Figure~\ref{functor}) we define
$$(f \cyl{\Gamma} \matr{C})_{_{E}}(e) \isdef \left \{ \begin{array}{ll}
u^{^{j}}_{_{\alpha_{_{i}}(l)}}u^{^{j}}_{_{\alpha_{_{i}}(l')}} & e=v^{^{i}}_{_{l}}v^{^{i}}_{_{l'}}, \ f(v_{_{i}})=u_{_{j}},\\
& \textrm{(Entire edge is in base of a cylinder),}\\
u^{^{j}}_{_{\alpha_{_{i}}(l)}}u^{^{j'}}_{_{\alpha_{_{i'}}(l')}} & e=v^{^{i}}_{_{l}}v^{^{i'}}_{_{l'}}, \ f(v_{_{i}}v_{_{i'}})=u_{_{j}}u_{_{j'}},\\
& \textrm{(Edges which adjacent to both bases of a cylinder),}\\
u^{^{j}}_{_{\alpha_{_{i}}(l)}}w & e=v^{^{i}}_{_{l}}w, \ f(v_{_{i}})=u_{_{j}},\\
& \textrm{(An internal vertex and a vertex in the initial base } \matr{B}^-),\\
wu^{^{j}}_{_{\alpha_{_{i}}(l)}} & e=wv^{^{i}}_{_{l}},\ f(v_{_{i}})=u_{_{j}},\\
&\textrm{(An internal vertex and a vertex in the terminal base } \matr{B}^+),\\
ww' \in E(\matr{C}^{^{u_{_{j}}u_{_{j'}}}}) & e=ww' \in E(\matr{C}^{^{v_{_{i}}v_{_{i'}}}}),\ f(v_{_{i}}v_{_{i'}})=u_{_{j}}u_{_{j'}},\\ & \textrm{(Edges which is not adjacent to any base vertices).}
\end{array}\right. $$
Strictly speaking, this describes $f \cyl{\Gamma} -$ as a mapping that acts as $f$ but identically on cylinders as generalized edges.
\begin{figure}[ht]
\centering{\includegraphics[width=9cm]{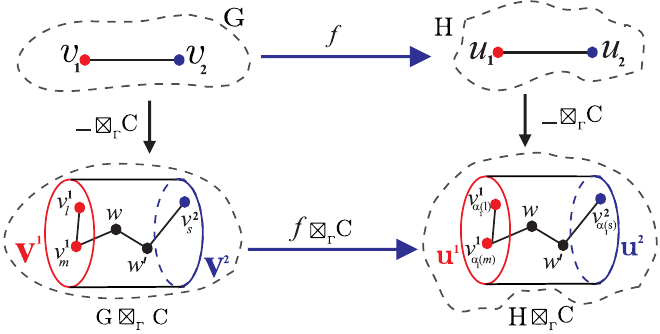}}
\caption{A commutative diagram to define $- \cyl{\Gamma} \matr{C}(f)$.}\label{functor}
\end{figure}
\begin{prepro}\label{pro:naturalcyl}
The map $- \cyl{\Gamma} \matr{C} : {\bf LGrph}(\Gamma,m) \longrightarrow {\bf LGrph}$ is a well defined functor.
\end{prepro}
\begin{proof}{
First, we should verify that for every $f \in \chom(\matr{G},\matr{G}')$, the map
$f \cyl{\Gamma} \matr{C}$ is actually a homomorphism in $\lhom(\matr{G} \cyl{\Gamma} \matr{C},\matr{G}' \cyl{\Gamma} \matr{C})$, but this is clear since by definition not only edges but cylinders are mapped to cylinders of the same type and moreover, if $e \in E(\matr{G} \cyl{\Gamma} \matr{C})$, then $(f \cyl{\Gamma} \matr{C})_{_{E}}(e)$ is an edge of $\matr{G}' \cyl{\Gamma} \matr{C}$. To see this,
we consider the first two cases as follows. The rest of the cases can be verified similarly.
\begin{itemize}
\item{{\bf $e=v^{^{i}}_{_{l}}v^{^{i}}_{_{l'}}, \ f(v_{_{i}})=u_{_{j}}$}.\\
In this case we know that $$(f \cyl{\Gamma} \matr{C})_{_{E}}(e)=u^{^{j}}_{_{\alpha_{_{i}}(l)}}u^{^{j}}_{_{\alpha_{_{i}}(l')}}$$
where $\ell^{^{v_{_{i}}}}_{_{\matr{G}}}(\tilde{e})=\gamma^-$ and $\ell^{^{u_{_{j}}}}_{_{\matr{G}'}}(f(\tilde{e}))=\alpha_{_{i}}\gamma^-$ for some edge $\tilde{e} \in E(\matr{G})$. But if $l=\gamma^-(l_{_{0}})$ then $\alpha_{_{i}}(l)=\alpha_{_{i}}\gamma^-(l_{_{0}})$ which shows that both edges $e$ and $u^{^{j}}_{_{\alpha_{_{i}}(l)}}u^{^{j}}_{_{\alpha_{_{i}}(l')}}$ correspond to the same edge in the cylinder $\matr{C}^{^{\ell^{*}_{_{\matr{G}}}(\tilde{e})}}$.
}
\item{{\bf $e=v^{^{i}}_{_{l}}v^{^{i'}}_{_{l'}}, \ f(v_{_{i}}v_{_{i'}})=u_{_{j}}u_{_{j'}}$}.\\
In this case we know that $$(f \cyl{\Gamma} \matr{C})_{_{E}}(e)=u^{^{j}}_{_{\alpha_{_{i}}(l)}}u^{^{j'}}_{_{\alpha_{_{i'}}(l')}},$$
where $\ell_{_{\matr{G}}}(v_{_{i}}v_{_{i'}})=(\gamma^-,t,\gamma^+)$ and $\ell_{_{\matr{G}}}(u_{_{j}}u_{_{j'}})=
(\alpha_{_{i}}\gamma^-,t,\alpha_{_{i'}}\gamma^+)$. But if we assume that $l=\gamma^-(l_{_{0}})$ and $l'=\gamma^+(l'_{_{0}})$ then $$\alpha_{_{i}}(l)=\alpha_{_{i}}\gamma^-(l_{_{0}}) \quad {\rm and} \quad
\alpha_{_{i'}}(l')=\alpha_{_{i'}}\gamma^+(l'_{_{0}}),$$
which shows that both edges $e$ and $u^{^{j}}_{_{\alpha_{_{i}}(l)}}u^{^{j'}}_{_{\alpha_{_{i'}}(l')}}$ correspond to the same edge in the cylinder $\matr{C}^{^{t}}$.
}
\end{itemize}
Also, one may easily verify that $1_{_{\matr{G}}} \cyl{\Gamma} \matr{C}=1_{_{\matr{G} \cyl{\Gamma} \matr{C}}}$ and that
$$(f \circ g) \cyl{\Gamma} \matr{C}=(f \cyl{\Gamma} \matr{C}) \circ (g \cyl{\Gamma} \matr{C}),$$ for every compatible pair of homomorphisms $f$ and $g$ (just attach two square like Figure~\ref{functor} side by side), which shows that $- \cyl{\Gamma} \matr{C}$ is a well-defined (covariant) functor.
}\end{proof}
On the other hand, for the exponential graph construction we have,
$$\forall\ \matr{H} \in Obj({\bf LGrph}) \quad ([\matr{C}, -]_{_{\Gamma}})(\matr{H}) \isdef [\matr{C},\matr{H}]_{_{\Gamma}},$$
$$\forall\ f \in \lhom(\matr{H},\matr{H'}) \quad ([\matr{C}, -]_{_{\Gamma}})(f) \isdef [\matr{C},f]_{_{\Gamma}} \in \chom([\matr{C},\matr{H}]_{_{\Gamma}},[\matr{C},\matr{H}']_{_{\Gamma}}),$$
where
$$\left([\matr{C},f]_{_{\Gamma}}\right)_{_{V}}(\bv) \isdef \langle f(\bv) \rangle_{_{\Gamma }}=f(\bv)\alpha_{_{\bv}},$$
and for $e=\bv\bw$ with $\ell_{_{[\matr{C},\matr{H}]_{_{\Gamma}}}}(e) \ = \ (\gamma^-,t,\gamma^+)$,
$$\left([\matr{C},f]_{_{\Gamma}}\right)_{_{E}}(e) \isdef \langle f(\bv) \rangle_{_{\Gamma}} \langle f(\bw) \rangle_{_{\Gamma}} \ \ {\rm with} \ \
\ell_{_{[\matr{C},\matr{H}']_{_{\Gamma}}}}(\langle f(\bv) \rangle_{_{\Gamma}} \langle f(\bw) \rangle_{_{\Gamma}})=(\alpha_{_{\bv}}\gamma^-,t,\alpha_{_{\bw}}\gamma^+).$$
\begin{prepro}\label{pro:naturalexp}
The map $[\matr{C}, -]_{_{\Gamma}} : {\bf LGrph} \longrightarrow {\bf LGrph}(\Gamma,m)$ is a well defined functor.
\end{prepro}
\begin{proof}{First, we should verify that for every $f \in \lhom(\matr{H},\matr{H'})$ the map $[\matr{C},f]_{_{\Gamma}}$ is actually a homomorphism in $\chom([\matr{C},\matr{H}]_{_{\Gamma}},[\matr{C},\matr{H}']_{_{\Gamma}})$.\\
For this, first note that by the definition of exponential graph construction and composition of homomorphisms, if $e=\bv\bw$ with $\ell_{_{[\matr{C},\matr{H}]_{_{\Gamma}}}}(e) \ = \ (\gamma^-,t,\gamma^+)$,
then $ \langle f(\bv) \rangle_{_{\Gamma}} \langle f(\bw) \rangle_{_{\Gamma}}$ is actually and edge of $[\matr{C},\matr{H}']_{_{\Gamma}}$ with the label $(\alpha_{_{\bv}}\gamma^-,t,\alpha_{_{\bw}}\gamma^+)$. Also, clearly by its definition and Proposition~\ref{CHOMLEMMA} this map is a homomorphism in $\chom([\matr{C},\matr{H}]_{_{\Gamma}},[\matr{C},\matr{H}']_{_{\Gamma}})$.\\
Moreover, one may verify that
$[\matr{C},1_{_{\matr{H}}}]_{_{\Gamma}}=1_{_{[\matr{C},\matr{H}]_{_{\Gamma}}}},$
and that
$$[\matr{C},f \circ g]_{_{\Gamma}}=[\matr{C},f]_{_{\Gamma}} \circ [\matr{C},g]_{_{\Gamma}},$$
for every compatible pair of homomorphisms $f$ and $g$, which shows that $[\matr{C},-]$
is a well-defined (covariant) functor.
}\end{proof} In what follows we prove that both maps $r_{_{\matr{G},\matr{H}}}$ and $s_{_{\matr{G},\matr{H}}}$ are natural with respect to indices $\matr{G}$ and $\matr{H}$. This shows that the pair $(r_{_{\matr{G},\matr{H}}},s_{_{\matr{G},\matr{H}}})$ is a weak version of an adjunct pair.
\begin{figure}[ht]
\centering{
\includegraphics[scale=.8]{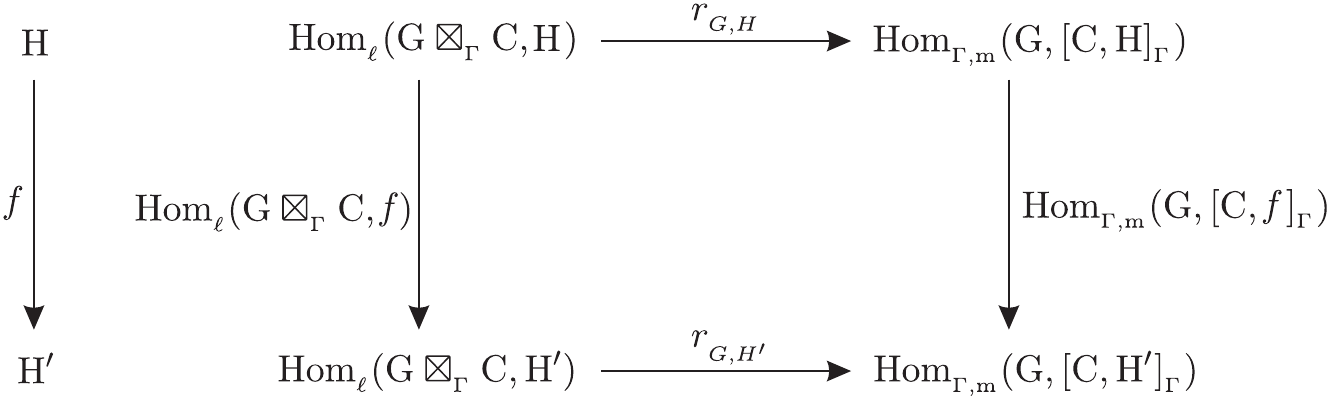}
}\caption{Naturality of $r_{_{\matr{G},\matr{H}}}$ with respect to $\matr{H}$ (see Theorem~\ref{thm:natural1}).}
\label{fig:diag1}
\end{figure}
\begin{thm}\label{thm:natural1}
Let $\matr{C}=\{\matr{C}^{^{j}}(\by^{^{j}},\by^{^{j}},\J^{^{j}}) \ \ | \ \ j \in \nat[m]\}$ be a $\Gamma$-coherent set of $(t,k)$-cylinders.
Then the retraction $r_{_{\matr{G},\matr{H}}}$ and the section $s_{_{\matr{G},\matr{H}}}$ introduced in Theorem~{\rm \ref{thm:maino}} are both natural with respect to $\matr{G}$ and $\matr{H}$.
\end{thm}
\begin{proof}{
In what follows we prove that $r_{_{\matr{G},\matr{H}}}$ is natural with respect to its second index $\matr{H}$, and that $s_{_{\matr{G},\matr{H}}}$ is natural with respect to its first index $\matr{G}$. The other two cases can be verified similarly.\\
For the first claim, assume that $V(\matr{G}) = \{v_{_{1}},v_{_{2}},\ldots,v_{_{n}}\}$ and fix homomorphisms
$$f \in \lhom(\matr{H},\matr{H}')\quad {\rm and} \quad \sigma=(\sigma_{_{V}},\sigma_{_{E}}) \in \lhom(\matr{G} \cyl{\Gamma} \matr{C},\matr{H}),$$
and let
$$r_{_{\matr{G},\matr{H}}}(\sigma)=\sigma' \in \chom(\matr{G},[\matr{C},\matr{H}]_{_{\Gamma}}), \quad {\rm and} \quad
\sigma'' \isdef \chom(\matr{G},[\matr{C},f]_{_{\Gamma}})(\sigma'),$$
(see the proof of Theorem~\ref{thm:maino} and the diagram depicted in Figure~\ref{fig:diag1}). Then, by definitions we have
$$
\begin{array}{ll}
\forall\ i \in \nat[n]\ \ \sigma_{_{V}}(\bv^{^i}) = \bu^{^i} & \Rightarrow \ \ \forall\ i \in \nat[n]\ \ \sigma'_{_{V}}(v_{_{i}}) = \langle \bu^{^i} \rangle_{_{\Gamma}}\\
&\\
&\Rightarrow \ \ \forall\ i \in \nat[n]\ \ \sigma''_{_{V}} (v_{_{i}}) = \langle f(\bu^{^i}) \rangle_{_{\Gamma}}.\\
\end{array}
$$
On the other hand, let
$$\tilde{\sigma} \isdef \lhom(\matr{G} \cyl{\Gamma} \matr{C},f)(\sigma) \quad {\rm and} \quad \tilde{\tilde{\sigma}} \isdef r_{_{\matr{G},\matr{H}'}}(\tilde{\sigma}).$$
Then, again by definitions,
$$
\begin{array}{ll}
\forall\ i \in \nat[n]\ \ \sigma_{_{V}}(\bv^{^i}) = \bu^{^i} & \Rightarrow \ \ \forall\ i \in \nat[n]\ \ \tilde{\sigma}_{_{V}}(\bv^{^i}) = f(\bu^{^i}) \\
&\\
&\Rightarrow \ \ \forall\ i \in \nat[n]\ \ \tilde{\tilde{\sigma}}_{_{V}} (v_{_{i}}) = \langle f(\bu^{^i}) \rangle_{_{\Gamma}},\\
\end{array}
$$
which clearly shows that $\sigma''_{_{V}}=\tilde{\tilde{\sigma}}_{_{V}}$.
The equality for the edge-maps also follows easily in a similar way, and
consequently, the diagram of Figure~\ref{fig:diag1} is commutative and $r_{_{\matr{G},\matr{H}}}$ is natural with respect to its second
\begin{figure}[ht]
\centering{
\includegraphics[scale=.8]{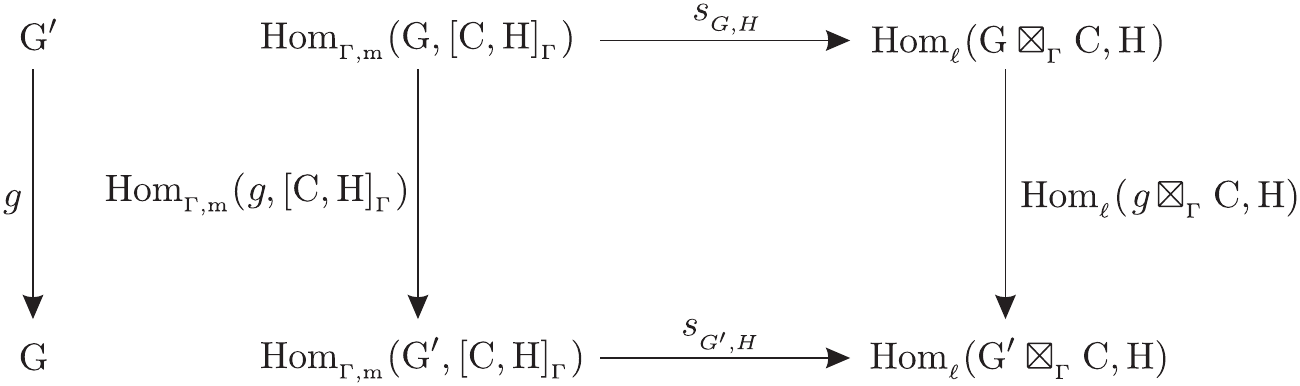}
}
\caption{Naturality of $s_{_{\matr{G},\matr{H}}}$ with respect to $\matr{G}$ (see Theorem~\ref{thm:natural1}).}
\label{fig:diag2}
\end{figure}
index $\matr{H}$.\\
For the second claim, assume that $V(\matr{G}') = \{w_{_{1}},w_{_{2}},\ldots,w_{_{m}}\}$, fix homomorphisms
$$g \in \chom(\matr{G}',\matr{G}), \ \ \quad {\rm and} \quad \sigma' \in \chom(\matr{G},[\matr{C},\matr{H}]_{_{\Gamma}}),$$
with the mapping $\beta$ such that
$$\forall\ i \in \nat[m]\ \ g(w_{_{i}}) \isdef v_{_{\beta(i)}},$$
and let $\{\alpha_{_{1}},\alpha_{_{2}},\ldots,\alpha_{_{n}}\} \subseteq \Gamma$ be obtained by Proposition~\ref{CHOMLEMMA}. Also, define,
$$\sigma \isdef s_{_{\matr{G},\matr{H}}}(\sigma') \in \lhom(\matr{G} \cyl{\Gamma} \matr{C},\matr{H}) \quad {\rm and} \quad \sigma'' \isdef \lhom(g \cyl{\Gamma} \matr{C},\matr{H})(\sigma). $$
(see the proof of Theorem~\ref{thm:maino} and the diagram depicted in Figure~\ref{fig:diag2}). Then, by definitions for all $i \in \nat[n]$ we may choose $\bu^{^i}$ such that for all $j \in \nat[m]$,
$$
\begin{array}{ll}
\sigma'_{_{V}}(v_{_{i}}) = \langle \bu^{^i} \rangle_{_{\Gamma}}= \bu^{^i} \alpha_{_{i}}& \Rightarrow \ \ \sigma_{_{V}}(\bv^{^i}) = \bu^{^i}\\
&\\
&\Rightarrow \ \ \sigma''_{_{V}} (\bw^{^j}) = \sigma_{_{V}}(g(\bw^{^j}))=\sigma_{_{V}}(\bv^{\beta(j)})=\bu^{\beta(j)}.\\
\end{array}
$$
On the other hand, let
$$\tilde{\sigma} \isdef \chom(g,[\matr{C},\matr{H}]_{_{\Gamma}})(\sigma') \quad
{\rm and} \quad \tilde{\tilde{\sigma}} \isdef s_{_{\matr{G}',\matr{H}}}(\tilde{\sigma}).$$
Then, for all $i \in \nat[n]$ and $j \in \nat[m]$,
$$
\begin{array}{ll}
\sigma'_{_{V}}(v_{_{i}}) = \langle \bu^{^i} \rangle_{_{\Gamma}} = \bu^{^i} \alpha_{_{i}}& \Rightarrow \ \ \tilde{\sigma}_{_{V}}(w_{_{j}}) = \sigma'_{_{V}}(g(w_{_{j}}))=\sigma'_{_{V}}(v_{_{\beta(j)}})=\langle \bu^{\beta(j)} \rangle_{_{\Gamma}}=\bu^{\beta(j)} \alpha_{_{\beta(j)}}\\
&\\
&\Rightarrow \ \ \tilde{\tilde{\sigma}}_{_{V}}(\bw^{^j}) = s_{_{\matr{G}',\matr{H}}}(\tilde{\sigma})(\bw^{^j})=\bu^{^{\beta(j)}},\\
\end{array}
$$
which shows that $\tilde{\tilde{\sigma}}_{_{V}}=\sigma''_{_{V}}$.
The equality for the edge-maps also follows easily in a similar way, and
consequently, the diagram of Figure~\ref{fig:diag2} is commutative and $s_{_{\matr{G},\matr{H}}}$ is natural with respect to its first
index $\matr{G}$.
}\end{proof}

\begin{pro}{\label{pro:basics2}
The standard line-graph construction is not cylindrical, i.e. if $L(\matr{H})$ is the standard line-graph of the fixed graph $\matr{H}$, then there is no $\Gamma$-coherent set of cylinders $\matr{C}$ for which the following holds,
$$\forall\ \matr{H'}\homeq \matr{H}, \quad [\matr{C},\matr{H'}]_{_{\Gamma}} \homeq L(\matr{H'}). $$
}\end{pro}
\begin{proof}{
Assume that there exists $\matr{C}$ such that
$$\forall\ \matr{H'}\homeq \matr{H}, \quad [\matr{C},\matr{H'}]_{_{\Gamma}} \homeq L(\matr{H'}). $$
Also, let $\chi(L(\matr{H}))=n$, and let $\matr{H'} \isdef \matr{H}(v)+\matr{St}_{_{n+1}}(v)$ be the graph obtained by identifying the central vertex $v$ of the star graph on $n+1$ vertices, $\matr{St}_{_{n+1}}$, with an arbitrary vertex of $\matr{H}$.
Clearly, $\matr{K}_{_{n+1}}\rightarrow L(\matr{H'})$.

 Note that $\matr{H}\homeq \matr{H'}$ and $L(\matr{H'}) \nrightarrow L(\matr{H})$, i.e.
$[\matr{C},\matr{H'}]_{_{\Gamma}} \nrightarrow [\matr{C},\matr{H}]_{_{\Gamma}},$
that contradicts the functoriality of the exponential construction.
}\end{proof}

\begin{prepro}{
Let $\matr{G}$ be a $\Gamma$-graph and $\matr{C}$ be a $\Gamma$-coherent cylinder, also let $\matr{H}$ be a labeled graph, then,
\begin{enumerate}
\item[{\rm (1)}] If $\sigma \in \lhom(\matr{G}\cyl{\Gamma} \matr{C}, \matr{H})$ is vertex-injective, then $r_{_{\matr{G},\matr{H}}}(\sigma) \in \ghom(\matr{G},[ \matr{C}, \matr{H}]_{_\Gamma})$ is also vertex-injective.
\item[{\rm (2)}]
The homomorphism $r_{_{\matr{G},\matr{H}}}(\sigma) \in \ghom(\matr{G},[ \matr{C}, \matr{H}]_{_\Gamma})$ is not necessarily vertex-surjective,
$($even if $\sigma \in \lhom(\matr{G}\cyl{\Gamma} \matr{C}, \matr{H})$ is an isomorphism$)$.
\end{enumerate}
}\end{prepro}
\begin{proof}{\
\begin{enumerate}
\item[{\rm (1)}]
Let $\sigma \in \lhom(\matr{G}\cyl{\Gamma} \matr{C}, \matr{H})$ be vertex-injective and $r_{_{\matr{G},\matr{H}}}(\sigma) \in \ghom(\matr{G},[ \matr{C}, \matr{H}]_{_\Gamma})$.
If $r_{_{\matr{G},\matr{H}}}(\sigma)(v_{_i})=<\sigma(\bv_{_i})>_{_\Gamma}$, where $i\in \{1,2\}$, then
$\sigma(\bv_{_1})$ and $\sigma(\bv_{_2})$ have not any vertex in common in $\matr{H}$, so
$<\sigma(\bv_{_i})>_{_\Gamma}\neq <\sigma(\bv_{_i})>_{_\Gamma}$, which means that $r_{_{\matr{G},\matr{H}}}(\sigma)(v_{_1})\neq r_{_{\matr{G},\matr{H}}}(\sigma)(v_{_2})$.

 \item[{\rm (2)}]
Just set $\matr{G}=\matr{C}_{_{r}},\ \matr{C}=\matr{P}_{_{t}}$ and $\matr{H}=\matr{C}_{_{rt}}$, where $r,t>1$, then it is easy to check that $id \in \lhom(\matr{G}\cyl{\Gamma} \matr{C}, \matr{H})= \lhom(\matr{C}_{_{rt}},\matr{C}_{_{rt}})$ is an isomorphism,

 Now $[\matr{P}_{_{t}},\matr{C}_{_{rt}}]$ has order $rt$ and the order of $\matr{C}_{_{r}}$ is $r$. Hence, there is not any surjective homomorphism in $\hom(\matr{G},[ \matr{C}, \matr{H}])=\hom(\matr{C}_{_{r}}, [\matr{P}_{_{t}},\matr{C}_{_{rt}}])$.
\end{enumerate}
}\end{proof}

\section{Adjunctions and reductions}\label{sec:adjred}

 Note that Theorem~\ref{thm:maino} and naturalities proved in previous section imply that we have an adjunction
$- \cyl{\Gamma} \matr{C} \dashv [\matr{C}, -]_{_{\Gamma}}$
between ${\bf LGrph}_{_{\leq}}(\Gamma,m)$ and ${\bf LGrph}_{_{\leq}}$, however, this correspondence is not necessarily an adjunction between ${\bf LGrph}(\Gamma,m)$ and ${\bf LGrph}$.

 In this section we consider conditions under which a class of cylinders reflects
an adjunction between ${\bf LGrph}(\Gamma,m)$ and ${\bf LGrph}$, that will be called {\it tightness}.
We will also consider weaker conditions called {\it upper} and {\it lower} closedness and will consider some examples
and consequences related to computational complexity of graph homomorphism problem (e.g. see \cite{HENE} for the background and some special cases).

 \begin{lem}\label{lem:tightnes}
Let $\matr{C}=\{\matr{C}^{^{j}}(\by^{^{j}},\by^{^{j}},\J^{^{j}}) \ \ | \ \ j \in \nat[m]\}$ be a set
of $\Gamma$-coherent $(t,k)$-cylinders, and $\matr{H}$ be a labeled graph. Then the following statements are equivalent,
\begin{itemize}
\item[{\rm a)}]{For any $(\Gamma,m)$-graph $\matr{G}$ the retraction $r_{_{\matr{G},\matr{H}}}$ is a one-to-one map $($i.e. is invertible with $r^{-1}_{_{\matr{G},\matr{H}}}=s_{_{\matr{G},\matr{H}}})$. }
\item[{\rm b)}]{For any $(\Gamma,m)$-graph $\matr{G}$ we have
$$|\lhom(\matr{G} \cyl{\Gamma} \matr{C},\matr{H})| = |\chom(\matr{G},[\matr{C},\matr{H}]_{_{\Gamma}})|.$$}
\item[{\rm c)}]{For any $(\Gamma,m)$-graph $\matr{G}$ and any $\sigma \in \chom(\matr{G},[\matr{C},\matr{H}]_{_{\Gamma}})$
there is a unique $\rho \in \lhom(\matr{G} \cyl{\Gamma} \matr{C},\matr{H})$ such that $[\matr{C},\rho]_{_{\Gamma}} \circ \eta_{_{\matr{G}}}=\sigma$.
}
\end{itemize}
\end{lem}
\begin{proof}{
$(\mathrm{a} \Leftrightarrow \mathrm{b})$: Considering the retraction
$$r_{_{\matr{G},\matr{H}}}: \lhom(\matr{G} \cyl{\Gamma} \matr{C},\matr{H}) \rightarrow \chom(\matr{G},[\matr{C},\matr{H}]_{_{\Gamma}})$$
and the section
$$s_{_{\matr{G},\matr{H}}}: \chom(\matr{G},[\matr{C},\matr{H}]_{_{\Gamma}}) \rightarrow \lhom(\matr{G} \cyl{\Gamma} \matr{C},\matr{H}). $$
It is easy to see that the retraction is one-to-one if and only if the section is so and this is equivalent to
$$|\lhom(\matr{G} \cyl{\Gamma} \matr{C},\matr{H})| = |\chom(\matr{G},[\matr{C},\matr{H}]_{_{\Gamma}})|.$$
$(\mathrm{a} \Leftrightarrow \mathrm{c})$: Let $\rho \isdef r_{_{\matr{G},\matr{H}}} (\sigma). $
Then by definitions
$$
\begin{array}{ll}
\forall v_{_{i}} \in V(\matr{G}) \ \ \rho_{_{V}}(\bv^{^i}) = \bu^{^i}& \Rightarrow \forall v_{_{i}} \in V(\matr{G}) \ \ [\matr{C},\rho]_{_{\Gamma}}
(\langle \bv^{^i} \rangle_{_{\Gamma}} ) = \langle \bu^{^i} \rangle_{_{\Gamma}}\\
&\\
&\Rightarrow \forall v_{_{i}} \in V(\matr{G}) \ \ [\matr{C},\rho]_{_{\Gamma}} \circ \eta_{_{\matr{G}}} (v_{_{i}}) = \langle \bu^{^i} \rangle_{_{\Gamma}},
\end{array}
$$
which means that $ [\matr{C},\rho]_{_{\Gamma}} \circ \eta_{_{\matr{G}}}$ and $\sigma$ act the same on vertices, and consequently, since they are homomorphisms, they also have the same action on edges.\\
Now, if the retraction is one-to-one, then there is a unique map $\rho \isdef r_{_{\matr{G},\matr{H}}} (\sigma) $ for which
$[\matr{C},\rho]_{_{\Gamma}} \circ \eta_{_{\matr{G}}}=\sigma$ and vice versa.
}\end{proof}

 Note that for any $(\Gamma,m)$-graph, $\matr{G}$, we have
$$\chom(\matr{G},[\matr{C},\matr{G} \cyl{\Gamma} \matr{C}]_{_{\Gamma}}) \not = \emptyset,$$
since
$$\eta_{_{\matr{G}}} \isdef r_{_{\matr{G},\matr{G} \cyl{\Gamma} \matr{C}}}(1_{_{\matr{G} \cyl{\Gamma} \matr{C}}}) \in
\chom(\matr{G},[\matr{C},\matr{G} \cyl{\Gamma} \matr{C}]_{_{\Gamma}}),$$
and moreover, it is easy to see that $\eta : {\bf 1}_{_{{\bf LGrph(\Gamma,m)}}} \Rightarrow [\matr{C}, -]_{_{\Gamma}} \circ (- \cyl{\Gamma} \matr{C})$,
is a natural transformation.

 Similarly, for any labeled graph $\matr{H}$ we have $\lhom([\matr{C},\matr{H}]_{_{\Gamma}} \cyl{\Gamma} \matr{C},\matr{H}) \not = \emptyset$,
since $$\varepsilon_{_{\matr{H}}} \isdef s_{_{[\matr{C},\matr{H}]_{_{\Gamma}},\matr{H}}}(1_{_{[\matr{C},\matr{H}]_{_{\Gamma}}}}) \in
\lhom([\matr{C},\matr{H}]_{_{\Gamma}} \cyl{\Gamma} \matr{C},\matr{H}),$$
and moreover, it is easy to see that
$\varepsilon : {\bf 1}_{_{{\bf LGrph}}} \Rightarrow (- \cyl{\Gamma} \matr{C}) \circ [\matr{C}, -]_{_{\Gamma}}$,
is a natural transformation.

 \begin{defin}{{\bf Tightness}\\
A $\Gamma$-coherent set of $(t,k)$-cylinders $\matr{C}=\{\matr{C}^{^{j}}(\by^{^{j}},\by^{^{j}},\J^{^{j}}) \ \ | \ \ j \in \nat[m]\}$
is said to be {\it tight} with respect to a labeled graph $\matr{H}$, if it satisfies any one of the equivalent conditions of
Lemma~\ref{lem:tightnes}.\\
A $\Gamma$-coherent set of $(t,k)$-cylinders $\matr{C}$ is said to be {\it lower-closed} with respect to a $(\Gamma,m)$-graph $\matr{H}$ if
$$\chom([\matr{C},\matr{H} \cyl{\Gamma} \matr{C}]_{_{\Gamma}},\matr{H}) \not = \emptyset.$$
Also, $\matr{C}$ is said to be {\it upper-closed} with respect to a labeled graph $\matr{G}$ if
$$\lhom(\matr{G},[\matr{C},\matr{G}]_{_{\Gamma}} \cyl{\Gamma} \matr{C}) \not = \emptyset.$$
}\end{defin}

 It is clear that a $\Gamma$-coherent set of $(t,k)$-cylinders $\matr{C}$ is {\it lower-closed} with respect to a $(\Gamma,m)$-graph $\matr{H}$
if and only if $[\matr{C},\matr{H} \cyl{\Gamma} \matr{C}]_{_{\Gamma}} \homeq \matr{H}$,
and similarly, $\matr{C}$ is {\it upper-closed} with respect to a labeled graph $\matr{G}$ if and only if
$[\matr{C},\matr{G}]_{_{\Gamma}} \cyl{\Gamma} \matr{C} \homeq \matr{G}$.
Also, it should be noted that if a cylinder is {\it strong} in the sense of \cite{HENE}, i.e.
\begin{itemize}
\item{A $(t,k)$-cylinder $\matr{C}$ is {\it strong} if for any irreflexive $(\Gamma,m)$-graph $\matr{H}$, and any homomorphism
$\sigma: \matr{C} \longrightarrow \matr{H} \cyl{\Gamma} \matr{C}$, the homomorphic image $\sigma(\matr{C})$ is contained in some cylinder $\matr{C}$.}
\end{itemize}
then any strong cylinder is lower-closed with respect to all irreflexive $(\Gamma,m)$-graphs.

 \begin{exm}{\label{exm:square}
\begin{enumerate}
\item {\it $\square$ is lower-closed}\\
If $|E(\matr{G})|=m$, then $[ \square ,\matr{G}\boxtimes \square ] \simeq \matr{G} \cup \matr{I}_{m}$, where the cylinder
$\square$ is depicted in Figure~\ref{fig:square}$(a)$. Therefore, $ \square$ is lower-closed with respect to any graph.
\begin{figure}[ht]
\centerline{
\includegraphics[width=7cm]{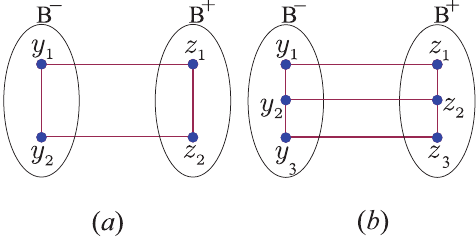}}
\caption{$(a)$ The cylinder $\square$, \ $(b)$ The cylinder $\boxminus$.}\label{fig:square}
\end{figure}
\item {\it Any non-empty bipartite cylinder is upper-closed with respect to any bipartite graph.}\\
Let $ \matr{C}$ and $ \matr{G}$ be non-empty and bipartite. Therefore, $ \matr{C} \rightarrow \matr{G}$, and $[\matr{C},\matr{G}]$ and consequently $[\matr{C},\matr{G}]\cyl{}\matr{C}$ is non-empty. Thus, since $\matr{G}$ is bipartite $\matr{G} \rightarrow [\matr{C},\matr{G}]\cyl{}\matr{C}$.
\item
\begin{itemize}
\item{ {\it $\matr{C}_{\times}(\matr{G})$ is tight with respect to any graph $\matr{H}$.}\\
Since $|\matr{Hom}(\matr{G}\times \matr{F}, \matr{H})| = |\matr{Hom}(\matr{F}, \matr{H}^\matr{G})|$ (e.g. see \cite{HENE}),
we have
$$|\matr{Hom}(\matr{H}\cyl{} \matr{C}_{\times}(\matr{G}), \matr{F})| = |\matr{Hom}(\matr{H}, [\matr{C}_{\times}(\matr{G}), \matr{H}] )|,$$ and consequently $\matr{C}_{\times}(\matr{G})$ is tight.
}
\item{ {\it $\matr{C}_{\times}(\matr{G})$ is upper-closed with respect to $\matr{H}$, if and only if $\matr{G} \rightarrow \matr{H}$.}\\
Since $\hom(\matr{G},\matr{G}^{\matr{H}}) \neq \emptyset$, we have
$$\hom(\matr{H}, [\matr{C}_{\times}(\matr{G}),\matr{H}]\cyl{} \matr{C}_{\times}(\matr{G})) \neq \emptyset \Leftrightarrow \hom(\matr{H},\matr{H}^{\matr{G}}\times \matr{H}) \neq \emptyset.$$
But $|\matr{Hom}(\matr{H}, \matr{H}^\matr{G}\times \matr{G})|=|\matr{Hom}(\matr{H}, \matr{H}^\matr{G})|. |\matr{Hom}(H,\matr{G})|$ (e.g. see \cite{HENE}), and by assuming $\hom(\matr{H},\matr{G}) \neq \emptyset$, we have $\matr{Hom}(\matr{H}, \matr{H}^\matr{G}\times \matr{G}) \neq \emptyset$. It is easy to check the reverse implication.
}

 \item{ {\it $\matr{C}_{\times}(\matr{G})$ is not lower-closed with respect to a simple graph $\matr{H}$, if $\matr{G} \rightarrow \matr{H}$.}\\
If $\matr{G} \rightarrow \matr{H}$ then $\matr{G} \rightarrow \matr{G} \times \matr{H}$. Also, $G \approx Gl(\matr{C}_{\times}(\matr{G}))$, and consequently $Gl(\matr{C}_{\times}(\matr{G})) \rightarrow \matr{G} \times \matr{H}$. Thus,
$$(\matr{G} \times \matr{H})^\matr{G} \simeq [\matr{C}_{\times}(\matr{G}),\matr{H} \cyl{} \matr{C}_{\times}(\matr{G})]=[\matr{C}_{\times}(\matr{G}), \matr{G} \times\matr{H}]$$
has a loop (consider the fact that if $Gl(\matr{C})\rightarrow \matr{H}$, then $[\matr{C}, \matr{H}]$ has a loop), and
since $\matr{H}$ does not have any loop,
$$\hom([\matr{C}_{\times}(\matr{G}),\matr{H} \cyl{} \matr{C}_{\times}(\matr{G})], \matr{H})=\emptyset.$$
}
\end{itemize}

 The cylinder $\matr{C}_{\times}(\matr{G})$ shows that a tight cylinder may neither be lower nor upper-closed.

 \end{enumerate}
}\end{exm}
\begin{prelem}
Let $\matr{C}$ be a $\Gamma$-coherent set of $(t,k)$-cylinders. Then,
\begin{itemize}
\item[{\rm a)}]{If $\matr{C}$ is lower-closed with respect to a $(\Gamma,m)$-graph, $\matr{H}$, then
$$\lhom(\matr{G} \cyl{\Gamma} \matr{C},\matr{H} \cyl{\Gamma} \matr{C}) \not = \emptyset \quad \Leftrightarrow \quad
\chom(\matr{G},\matr{H}) \not = \emptyset.$$}
\item[{\rm b)}]{If $\matr{C}$ is upper-closed with respect to a labeled graph $\matr{G}$, then
$$\chom([\matr{C},\matr{G}]_{_{\Gamma}},[\matr{C},\matr{H}]_{_{\Gamma}}) \not = \emptyset \quad \Leftrightarrow \quad
\lhom(\matr{G},\matr{H}) \not = \emptyset.$$}
\end{itemize}
\end{prelem}
\begin{proof}{
To prove (a) by Theorem~\ref{thm:maino} and the definition we have,
$$\begin{array}{ll}
\lhom(\matr{G} \cyl{\Gamma} \matr{C},\matr{H} \cyl{\Gamma} \matr{C}) \not = \emptyset& \Leftrightarrow \ \
\chom(\matr{G},[\matr{C},\matr{H} \cyl{\Gamma} \matr{C}]_{_{\Gamma}}) \not = \emptyset\\
& \Rightarrow \ \ \chom(\matr{G},\matr{H}) \not = \emptyset.\\
\end{array}
$$
The inverse implication is clear by Proposition~\ref{pro:naturalcyl}.
Similarly, to prove (b) we have,
$$
\begin{array}{ll}
\chom([\matr{C},\matr{G}]_{_{\Gamma}},[\matr{C},\matr{H}]_{_{\Gamma}}) \not = \emptyset& \Leftrightarrow \ \
\lhom([\matr{C},\matr{G}]_{_{\Gamma}} \cyl{\Gamma} \matr{C},\matr{H}) \not = \emptyset\\
& \Rightarrow \ \ \chom(\matr{G},\matr{H}) \not = \emptyset.\\
\end{array}
$$
The inverse implication is clear by Proposition~\ref{pro:naturalexp}.
}\end{proof}
\begin{pro}{\label{pro:con}
Suppose that $\matr{C}$ is a connected cylinder, then $\matr{C}$ is not tight with respect to a graph $\matr{H}$ if and only if it is not tight with respect to at least one of its connected components.
}\end{pro}
\begin{proof}{
If we decompose $\matr{H}$ to the connected components $$\matr{H}=\matr{H}_{_1}\cup \matr{H}_{_2} \cup . . . \cup \matr{H}_{_s},$$
we can see that
\begin{align*}
|\matr{Hom}(\matr{G}\cyl{}\matr{C},\matr{H})|&=
\sum_{i=1}^s|\matr{Hom}(\matr{G}\cyl{}\matr{C},\matr{H}_i)| \\
&\geq \sum_{i=1}^s|\matr{Hom}(\matr{G},[\matr{C},\matr{H}_i])| \\
&=|\matr{Hom}(\matr{G},\cup_{i=1}^s[\matr{C},\matr{H}_i])|= |\matr{Hom}(\matr{G},[\matr{C},\matr{H}])|.
\end{align*}
Therefore, since for all $i\in\{1,2,...,s\}$ we have
$|\matr{Hom}(\matr{G}\cyl{}\matr{C},\matr{H}_i)|\geq |\matr{Hom}(\matr{G},[\matr{C},\matr{H}_i])|,$
the strict inequality
$$|\matr{Hom}(\matr{G}\cyl{}\matr{C},\matr{H})|> |\matr{Hom}(\matr{G},[\matr{C},\matr{H}])|$$
holds for $\matr{H}$ if and only if the strict inequality
$$|\matr{Hom}(\matr{G}\cyl{}\matr{C},\matr{H}_i)|> |\matr{Hom}(\matr{G},[\matr{C},\matr{H}_i])|$$
holds for one of the indices $i\in\{1,2,...,s\}$.
}\end{proof}

 \begin{exm}{ \textbf{Tightness of paths.}\\
\textbf{{\small 1- Odd paths}}
\begin{itemize}
\item[{\rm (1)}] {\it Let $k>1$ and $\matr{H}$ be a simple graph in which $\matr{C}^1,\matr{C}^2, \dots \matr{C}^s$ are $s$ distinct odd cycles of lengths $c_1,c_2, . . . , c_s$ with $c_i \leq 2k+1$ for all $1 \leq i \leq s$. If $2\sum_{i=1}^s c_i>|V(\matr{H})|$, then $\matr{P}_{_{2k+1}}$ is not tight with respect to $\matr{H}$.}

 \indent Note that if $\matr{C}_{_{r}}$ and $\matr{C}_{_{s}}$, are two odd cycles with $s\leq r$, then there exist at least $2s$ distinct homomorphisms from $\matr{C}_{_{r}}$ to $\matr{C}_{_{s}}$, (mark vertices of $\matr{C}_{_{s}}$ and $\matr{C}_{_{r}}$ in circular order and for any $i\in \{1,2, . . . ,s\}$, consider the homomorphisms $1 \mapsto i, 2\mapsto i+1, 3\mapsto i+2, ...$ and $1 \mapsto i, 2\mapsto i-1, 3\mapsto i-2, ...$).

 Now, consider two distinct odd cycles $\matr{C}^1,\matr{C}^2$ of lengths $c_1,c_2\leq 2k+1$ in $\matr{H}$,
and note that since any homomorphism to an odd cycle is a vertex-surjective map, for odd $r$
$$\matr{Hom}(\matr{C}_{_{r}},\matr{C}^1) \cap \matr{Hom}(\matr{C}_{_{r}},\matr{C}^2) = \emptyset.$$
Therefore, if $\matr{L}$ is a loop on one vertex, then
$$|\matr{Hom}(\matr{L}\cyl{}\matr{P}_{_{2k+1}},\matr{H})|=|\matr{Hom}(\matr{C}_{_{2k+1}},\matr{H})|\geq 2\sum_{i=1}^s c_i.$$
On the other hand, since $[\matr{P}_{_{2k+1}}, \matr{H}]$ has at most $|V(\matr{H})|$ loops we have,
$$|\matr{Hom}(\matr{L}, [\matr{P}_{_{2k+1}}, \matr{H})])|\leq |V(\matr{H})|,$$
and consequently,
$$|\matr{Hom}(\matr{L}\cyl{}\matr{P}_{_{2k+1}},\matr{H})|\geq 2\sum_{i=1}^s c_i>|V(\matr{H})|\geq |\matr{Hom}(\matr{L},[\matr{P}_{_{2k+1}},\matr{H}])|.$$

 \item[{\rm (2)}] {\it ${\matr{P}_{_{2k+1}}}$ is lower-closed with respect to any simple graph $\matr{G}$.}

 By \cite{haji} we have $\hom(\matr{G}^{\frac{2k+1}{2k+1}},\matr{G})\neq \emptyset$.
On the other hand, we have
$$\hom([\matr{P}_{_{2k+1}},\matr{G}\cyl{} \matr{P}_{_{2k+1}}],\matr{G})=\hom(\matr{G}^{\frac{2k+1}{2k+1}},\matr{G}),$$
and consequently, $\matr{P}_{_{2k+1}}$ is lower-closed with respect to $\matr{G}$.
\item[{\rm (3)}] {\it ${\matr{P}_{_{2k+1}}}$ is not upper-closed with respect to $\matr{G}$, if $\matr{G}$ has odd girth less than or equal to $2k+1$.}

 We have
$$\hom(\matr{G},[\matr{P}_{_{2k+1}},\matr{G}]\cyl{} \matr{P}_{_{2k+1}})=\hom(\matr{G}, (\matr{G}^{2k+1})^{\frac{1}{2k+1}}),$$
but the smallest odd cycle of $(\matr{G}^{2k+1})^{\frac{1}{2k+1}}$ is $2k+1$, and consequently, the odd girth of $\matr{G}$ is less than the odd girth of
$(\matr{G}^{2k+1})^{\frac{1}{2k+1}}$. Hence, there is not any homomorphism from $\matr{G}$ to
$(\matr{G}^{2k+1})^{\frac{1}{2k+1}}$.
\end{itemize}

 \textbf{{\small 2- Even paths}}
\begin{itemize}
\item[{\rm (1)}] {\it The cylinder $\matr{P}_{_{2k}}$ is tight with respect to $\matr{H}$, if and only if $\matr{H}$ is a matching.}

 By Proposition~\ref{pro:con}, we may assume that $\matr{H}$ is connected.
Also, if $\matr{H}$ is a connected graph that has a vertex of degree $2$, then we have $2|E(\matr{H})|>|V(\matr{H})|$.\\
If $\matr{L}$ is a loop on a single vertex, since there exist at least two homomorphisms from an even cycle to an edge, we have
$$|\matr{Hom}(\matr{L}\cyl{}\matr{P}_{_{2k}},\matr{H})|=|\matr{Hom}(\matr{C}_{_{2k}},\matr{H})| \geq 2|E(\matr{H})|.$$
Also, $|\matr{Hom}(\matr{L},[\matr{P}_{_{2k}},\matr{H}])|\leq |V(\matr{H})|$, because the worst case is that
$[\matr{P}_{_{2k}},\matr{H}]$ has a loop on each vertex.
Therefore, when $2|E(\matr{H})|>|V(\matr{H})|$,
$$|\matr{Hom}(\matr{L}\cyl{}\matr{P}_{_{2k}},\matr{H})|\geq 2|E(\matr{H})|>|V(\matr{H})|\geq |\matr{Hom}(\matr{L},[\matr{P}_{_{2k}},\matr{H}])|.$$

 If $\matr{H}$ does not have any vertex of degree $2$, it is a matching and by
Proposition~\ref{pro:con} it is enough to assume that $\matr{H}=\matr{K}_{_2}$.

 Note that if $\matr{F}$ is a bipartite graph, then $ |\matr{Hom}(\matr{F},\matr{K}_{_2})|=2^c$, where $c$ is the number of connected components of $\matr{F}$.
For any graph $\matr{G}$ with $c$ connected components, $\matr{G}\cyl{}\matr{P}_{_{2k}}$ is bipartite with $c$ connected components.
Now, since $[\matr{P}_{_{2k}},\matr{K}_{_2}]=\matr{L}\cup \matr{L}$, there exists $2^c$ homomorphisms from $\matr{G}$ to $[\matr{P}_{_{2k}},\matr{K}_{_2}]$, and consequently,
$$\hom(\matr{G}, [\matr{P}_{_{2k}},\matr{K}_{_2}])=2^c=\hom(\matr{G}\cyl{}\matr{P}_{_{2k}},\matr{K}_{_2}).$$
\item[{\rm (2)}] {\it $\matr{P}_{_{2k}}$ is not upper-closed with respect to any non-bipartite graph.}

 Note that for any non-bipartite graph $\matr{G}$, the construction $[\matr{P}_{_{2k}},\matr{G}] \cyl{} \matr{P}_{_{2k}}$ is bipartite, and consequently,
$$\hom(\matr{G},[\matr{P}_{_{2k}},\matr{G}] \cyl{}\matr{P}_{_{2k}}) = \emptyset.$$

 \item[{\rm (3)}] {\it $\matr{P}_{_{2k}}$ is not lower-closed with respect to any loop free graph.}

 Note that for any non-empty graph $\matr{G}$, $[\matr{P}_{_{2k}},\matr{G} \cyl{}\matr{P}_{_{2k}}]$ has a loop,
and consequently, for a loop free graph $\matr{G}$, we have
$$\hom(\matr{G},[\matr{P}_{_{2k}},\matr{G} \cyl{}\matr{P}_{_{2k}}]) = \emptyset.$$
\end{itemize}

 }\end{exm}

\begin{cor}{
$\matr{P}_{_{2k+1}}$ is not tight with respect to any graph with the following spanning subgraphs,
\begin{itemize}
\item $\matr{C}_{_{2r+1}}$ for $r\leq k$.
\item A graph $\matr{H}$, in which at least half of its vertices lie in an odd cycle of length less than $2k+1$.
\end{itemize}
}\end{cor}

 Applications of indicator/replacement constructions to study the computational complexity of the homomorphism problem was initiated in
\cite{16} and now is well-known in the related literature (see \cite{5,HENE,7,11,15,16,fiber1,18,trans}). In what follows we will consider
some basic implications for the cylindrical/exponential construction as a generalization.

 \begin{prb} {\bf The $\matr{H}$-homomorphism problem } {\rm ($\matr{H}$-\hom)}\\
Given a labeled graph $\matr{G}$, does there exist a homomorphism $\sigma: \matr{G} \longrightarrow \matr{H}$?\\
$($i.e. find out whether $\lhom(\matr{G},\matr{H}) \not = \emptyset$.$)$
\end{prb}
\begin{prb} {\bf The $\#\matr{H}$-homomorphism problem } {\rm ($\#\matr{H}$-\hom)}\\
Given a labeled graph $\matr{G}$, compute the number of homomorphisms $\sigma: \matr{G} \longrightarrow \matr{H}$?\\
$($i.e. find out the number $|\lhom(\matr{G},\matr{H})|$.$)$
\end{prb}
To begin, note that as an immediate consequence of Theorem~\ref{thm:maino} we have,
\begin{cor}{\label{cor:np} If $\ \leq^{^{P}}_{_{m}}$ stands for the many-to-one polynomial-time reduction, then
\begin{center}
{\rm $[\matr{C},\matr{H}]_{_{\Gamma}}$-\hom} $\ \leq^{^{P}}_{_{m}} \ $ {\rm $\matr{H}$-\hom}.
\end{center}
}\end{cor}
\begin{proof}{
The reduction is exactly the cylindrical construction which is clearly a polynomial-time computable mapping of graphs.
}\end{proof}

 The following corollary is a direct consequence of Lemma~\ref{lem:tightnes}.
\begin{cor} If $\ \leq^{^{P}}_{_{\#}}$ stands for the parsimonious polynomial-time reduction, and $\matr{C}$ is tight with respect to $\matr{H}$, then
\begin{center}
{\rm $\#[\matr{C},\matr{H}]_{_{\Gamma}}$-\hom} $\ \leq^{^{P}}_{_{\#}} \ $ {\rm $\#\matr{H}$-\hom}.
\end{center}
\end{cor}

\begin{cor}{\label{cor:np1}
Let $n\geq 3$ be the odd girth of the simple graph $\matr{G}$. If there exists an odd path of length less than $n$ between each pair of vertices
in $\matr{G}$, then $\matr{G}-\hom$ is $\matr{NP}$-complete.
}\end{cor}
\begin{proof}{
By the assumptions, $[\matr{P}_{_{n-2}},\matr{G}]\simeq \matr{K}_{_{|V(\matr{G})|}}$ and by Corollary~\ref{cor:np} we have
\begin{center}
{\rm $[\matr{P}_{_{n-2}},\matr{G}]$-\hom} $\simeq$ {\rm $\matr{K}_{_{|V(\matr{G})|}}$-\hom} $\ \leq^{^{P}}_{_{m}}$ {\rm $\matr{G}$-\hom}.
\end{center}
But, since $|V(\matr{G})|\geq 3$, $\matr{G}-\hom$ is $\matr{NP}$-complete.
}\end{proof}

 \begin{cor}{\label{cor:ex}
For following graphs, $\matr{G}${\rm-\hom} is $\matr{NP}$-complete.
\begin{itemize}
\item $\matr{G}=\matr{C}_{_n}$ for odd $n$.
\item $\matr{G}=$ the Petersen graph.
\item $\matr{G}=$ Coxeter graph.
\item $\matr{G}=\matr{C}_{_n}^r$ for odd $n$ and $r|n-2$.
\end{itemize}
}\end{cor}
\begin{proof}{
The first three cases follow from Corollary~\ref{cor:np1}.
For the last one, note that $\matr{C}_{_n}${\rm-\hom} is $\matr{NP}$-complete for odd $n$. Also by Theorem~\ref{thm:maino}, we have
$$\begin{array}{lll}
\matr{G} \rightarrow [\matr{P}_{_{n-2}},\matr{C}_{_n}]\simeq \matr{K}_{_n} & \leq^{^{P}}_{_{m_{}}} & \matr{G}\cyl{}\matr{P}_{_{n-2}} \rightarrow \matr{C}_{_n}\\
& \leq^{^{P}}_{_{m}} \quad &
(\matr{G}\cyl{}(\matr{P}_{_\frac{n-2}{r}}\cyl{}\matr{P}_{_r}) \rightarrow \matr{C}_{_n}\\
&\leq^{^{P}}_{_{m}} \quad &
(\matr{G}\cyl{}\matr{P}_{_\frac{n-2}{r}})\cyl{}\matr{P}_{_r}\rightarrow \matr{C}_{_n}\\
& \leq^{^{P}}_{_{m}} \quad & \matr{G}\cyl{}\matr{P}_{_\frac{n-2}{r}}\rightarrow [\matr{P}_{_r}, \matr{C}_{_n}]\\
& \leq^{^{P}}_{_{m}} \quad &
\matr{G}\rightarrow [\matr{P}_{_r}, \matr{C}_{_n}].
\end{array}$$
}\end{proof}

 Let ${\cal H}$ be a class of graphs. Then the {\it cylindrical closure } of ${\cal H}$, denoted by $\Psi({\cal H})$, is defined as
$$\Psi({\cal H}) \isdef \{ \matr{G} \ : \ \exists \matr{C} \ \ [\matr{C},\matr{G}] \in {\cal H} \}.$$
Considering the identity cylinder one gets ${\cal H} \subseteq \Psi({\cal H})$, and moreover, the {\it universal closure}
of ${\cal H}$ is defined as
$$\overline{{\cal H}} \isdef \displaystyle{\bigcup_{n=0}^{\infty}} \Psi^{n}({\cal H})=\displaystyle{\lim_{n \to \infty}} \Psi^{n}({\cal H}),$$
where $\Psi^{0}({\cal H}) \isdef {\cal H}$. Clearly, by Corollary~\ref{cor:np}, if the $\matr{H}$-homomorphism problem is
$\matr{NP}$-complete for all $\matr{H} \in {\cal H}$ then this problem is also $\matr{NP}$-complete for all
$\matr{H} \in \overline{{\cal H}}$. To prove a Dichotomy Conjecture in a fixed category of graphs it is necessary and sufficient
to find a class ${\cal H}$ such that the $\matr{H}$-homomorphism problem is $\matr{NP}$-complete for all
$\matr{H} \in \overline{{\cal H}}$ and it is polynomial time solvable for all $\matr{H} \not \in \overline{{\cal H}}$
(e.g. see \cite{5,HENE,7,11,15,16,fiber1,18,trans} for the background and results in this regard for simple and directed graphs
as well as CSP's). It is interesting to consider possible generalizations of indicator/replacement/fiber construction techniques
to the case of cylindrical constructions or consider dichotomy and no-homomorphism results in the category of labeled graphs.


 \ \\

 \noindent
{\large \bf Appendix: On the existence of pushouts} \ \\

 This section contains subtleties about the existence of pushouts in different categories of graphs. The contents are classic
and have been included to clarify the details and avoid misunderstandings.

\begin{defin}{\textbf{Serre-Stallings category of Graphs, ${\bf StGrph}$}\\
This is a category of graphs usually used in algebraic topology and dates back at least to the contributions of Serre and Stallings \cite{Serre, St} where the existence of pushouts in this category was first addressed in \cite{St}.

The category is denoted by ${\bf StGrph}$ and a graph $(\matr{F},V, E)$ is an object of this category with
two sets $V$ (containing vertices) and $E$ (containing edges), and two functions $\iota_{_\matr{F}}:E\rightarrow V$
and $s_{_\matr{F}}:E\rightarrow E$. For each edge $e\in E$, there is an inverse edge $s(e)=\overline{e}\in E$, and an initial vertex $\iota(e)\in V$ where the terminal vertex defined as $\tau(e)=\iota(s(e))$ satisfying $\overline{\overline{e}}=e$ and $\overline{e}\neq e$.

 A ${\bf StGrph}$-{\it homomorphism} $(\sigma_{_{V}},\sigma_{_{E}})$ between two graphs $(\matr{G},V, E)$ and $(\matr{H},V',E')$ is a pair of maps that preserve structures i.e.
for two maps $\sigma_{_{V}} :V\rightarrow V'$ and $\sigma_{_{E}} : E\rightarrow E'$, we have
$$\iota_{_{\matr{H}}} \circ \sigma_{_{E}} = \sigma_{_{V}} \circ \iota_{_{\matr{G}}},$$
$$s_{_{\matr{H}}} \circ \sigma_{_{E}} = \sigma_{_{E}} \circ s_{_{\matr{G}}}.$$
}\end{defin}

 If one fixes orientations for graphs $\matr{G}$ and $\matr{H}$ (i.e. choosing just one edge from any pair $\{e,\overline{e}\}$) and call the corresponding directed graphs $O(\matr{G})$ and $O(\matr{H})$, respectively, then assuming the existence of an ordinary directed graph homomorphism from $O(\matr{G})$ to $O(\matr{H})$, then clearly, one can canonically construct a graph homomorphism from $\matr{G}$ to $\matr{H}$ in ${\bf StGrph}$. Hence, if there exists orientations
for $\textbf{StGraph}$-graphs $\matr{G}, \matr{H}_1$ and $\matr{H}_2$ such that two graph homomorphisms $\alpha: \matr{G}\rightarrow \matr{H}_1$ and $\beta: \matr{G}\rightarrow \matr{H}_2$ preserve the corresponding orientations, then their pushout exists. 

 By \cite{St} for two graph homomorphisms $\alpha: \matr{G}\rightarrow \matr{H}_1$ and $\beta: \matr{G}\rightarrow \matr{H}_2$ in ${\bf StGrph}$, their pushout exists, if and only if there exists orientations for $\matr{G}, \matr{H}_1$ and
$\matr{H}_2$ such that $\alpha$ and $\beta$ preserve such orientations.

 \begin{exm}{\label{pushout} \textbf{Non-existence of pushout in ${\bf StGrph}$}\\
Consider $\matr{G}=\matr{C}_7$ and $\matr{H}_1=\matr{H}_2=\matr{C}_3$. We mark vertices of $\matr{G}$ with $\{1,2,3,4,5,6.7\}$ and vertices of $\matr{C}_3$ by $\{1,2,3\}$ in a circular configuration (see Figure \ref{fig:push}). Now, let
$$\alpha_{_V}=\{(1,1), (2,2), (3,3), (4,1),(5,2), (6,3), (7,2) \},$$
$$\beta_{_V}=\{(1,1), (2,2), (3,1), (4,2),(5,1), (6,2), (7,3) \},$$
where edges are mapped canonically (see Figure \ref{fig:push}).

 We claim that we can't find orientations for $\matr{G}$, $\matr{H}_1$ and $\matr{H}_2$ which preserved by $\alpha$ and $\beta$.
To see this, assume that we have orientations $O_{\matr{G}}, O_{\matr{H}_1}$ and $O_{\matr{H}_2}$ preserved by
$\alpha$ and $\beta$. We write $O_{\matr{G}}(ij)=1$, if the orientation of the edge $ij$ in $O_{\matr{G}}$ is from $i$ to $j$ in ${\matr{G}}$ and $O_{\matr{G}}(ij)=-1$ vice versa.

 Now, since $\alpha$ preserves orientations, we have $O_{\matr{G}}(12)=O_{\matr{G}}(45)$, and by orientation preserving property of $\beta$, we also have $O_{\matr{G}}(12)=-O_{\matr{G}}(45)$ which is a contradiction.

 On the other hand, one can see that if there exists a pushout object $(\matr{Q},\ j^1: \matr{H}_1 \rightarrow \matr{Q},\ j^2: \matr{H}_2\rightarrow \matr{Q})$, then
$$\alpha_{_E}(12)=\alpha_{_E}(45) \Rightarrow j^1_{_E}\alpha_{_E}(12)= j^1_{_E}\alpha_{_E}(45),$$
$$\beta_{_E}(21)=\beta_{_E}(45) \Rightarrow j^2_{_E}\beta_{_E}(21)= j^2_{_E}\beta_{_E}(45).$$
Hence, by commutativity of the pushout diagram we should have $j^1_{_E}\alpha_{_E}(45)=j^2_{_E}\beta_{_E}(45)$, hence $j^1_{_E}\alpha_{_E}(12)=j^2_{_E}\beta_{_E}(21)$.
But since the graph homomorphisms should preserve structure, we should also have
$$j^2_{_E}\beta_{_E}(21)=j^2_{_E}\beta_{_E}(\overline{12})=\overline{j^2_{_E}\beta_{_E}(12)}=\overline{j^1_{_E}\alpha_{_E}(21)}=\overline{j^2_{_E}\beta_{_E}(21)},$$
which is a contradiction since we have a rule in ${\bf StGrph}$ indicating that $\overline{j^2_{_E}\beta_{_E}(21)}\neq j^2_{_E}\beta_{_E}(21)$.
}\end{exm}
\begin{figure}[ht]
\centering{\includegraphics[width=9cm]{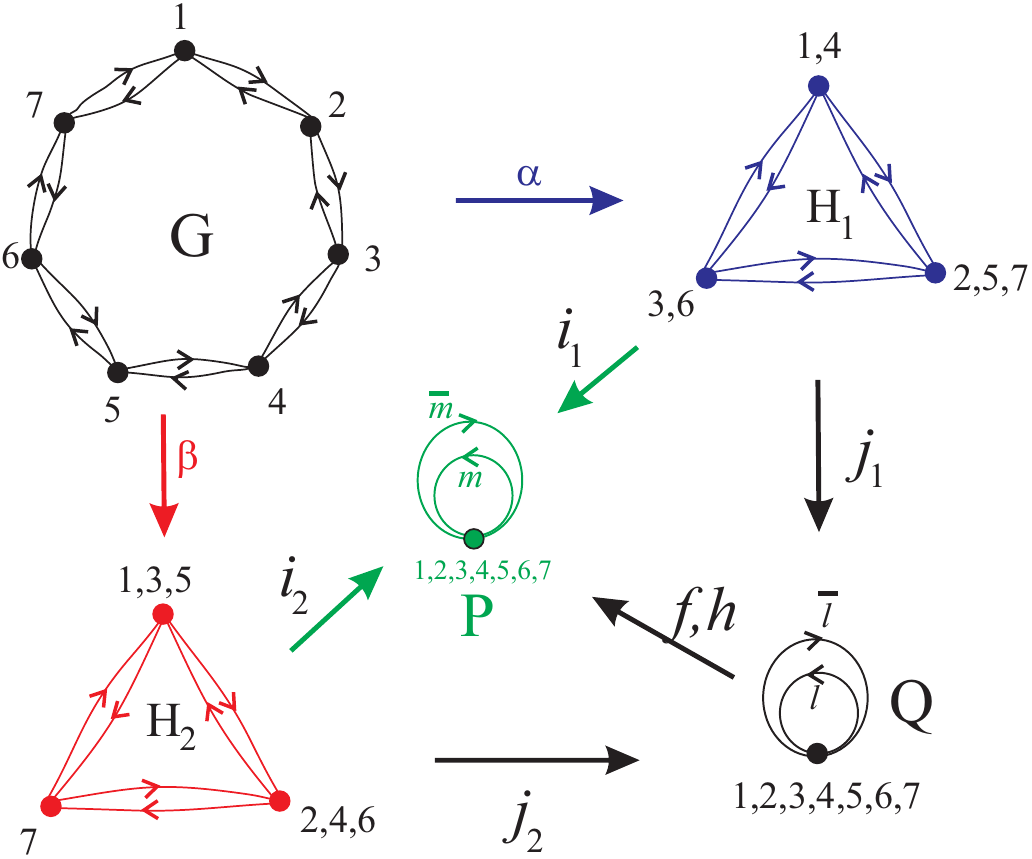}}
\caption{Push out diagram in ${\bf SGrph}$, where there exist two homomorphisms between two object which commute the diagram.}
\label{fig:push}
\end{figure}

Note that it is straight forward to verify that pushouts exist in  ${\bf Grph}$ and ${\bf SymGrph}$, since each homomorphism can be 
decomposed into vertex and edge maps for which the pushout exist and are compatible. Hence, the pushout of $\alpha$ and $\beta$ in the above example exists and is equal to a simple loop. 

Also, let ${\bf SGrph}$ be a the subcategory of  ${\bf Grph}$ whose objects are as those of ${\bf StGrph}$. 
Then, it is not hard to see that the pushout of $\alpha$ and $\beta$ in Example~\ref{pushout} doesn't exists in ${\bf SGrph}$, either. This can be verified since the pushout object, if it would have existed,  should have been a single vertex with two directed loops $\{l,\overline{l}\}$, where the image of $j_{_E}^1$ and $j_{_E}^2$ must be a singleton and without loss of generality we can assume that $\matr{Image}(j_{_E}^1)=\matr{Image}(j_{_E}^2)=\{l\}$. Therefore, considering two 
pushout objects (cones) $(\matr{P}, i^1: \matr{H}_1 \rightarrow \matr{P}, i^2: \matr{H}_2\rightarrow \matr{P})$ and 
$(\matr{Q}, j^1: \matr{H}_1 \rightarrow \matr{Q}, j^2: \matr{H}_2\rightarrow \matr{Q})$ for which the image of $i^1_{_E}$ and $i^2_{_E}$ is a single loop which we denote by $m$. Hence, we have two homomorphisms $h,f:\matr{Q} \rightarrow \matr{P}$, where
 $$h_{_E}=\{(l,m), (\overline{l},m)\}, \qquad f_{_E}=\{(l,m), (\overline{l},\overline{m})\}.$$
Both $h$ and $f$ satisfying the pushout diagram commutativity condition which is a contradiction.

Striktly speaking, in ${\bf SymGrph}$ there exist only a unique homomorphism between a pair of loops but in ${\bf SGrph}$, there exists more than one homomorphism between a pair of loops.

\begin{pro}{ Let $F:{\bf SGrph}\rightarrow {\bf SymGrph}$ be the forgetful functor, that sends each (directed) graph to its corresponding simple graph. Then $F$ is bijective on objects, but is not faithful on objects endowed with loops or multiple edges.
}\end{pro}


\begin{thebibliography}{10}

 \bibitem{lifts}
{\sc A.~Amit and N.~Linial}, {\em Random graph coverings I: general theory and graph connectivity}, Combinatorica, {\bf 22} (2002), 1--18.

 \bibitem{arxiv}
{\sc A.~Daneshgar, M.~Hejrati, M. Madani}, {\em Cylindrical Graph Construction (definition and basic properties)}, available at arXiv:1406.3767v1 [math.CO] 14 Jun, 2014.
\bibitem{5}
{\sc A. Bulatov and P. Jeavons}, {\em Algebraic Structures in Combinatorial Problems}, Submitted.
\bibitem{7}
{\sc A. Bulatov, P. Jeavons, A. Krokhin}, {\em Constraint Satisfaction Problems and Finite Algebras}, In: Proceedings of the 27th International Coll. on Automata, Languages and Programming, ICALP’00, LNCS {\bf 1853}, Springer, (2000),~272--282.

 \bibitem{11}
{\sc T. Feder and M. Vardi}, {\em The Computational Structure of Monotone
Monadic SNP and Constraint Satisfaction: a Study through Datalog and Group Theory}, SIAM J. Comput., {\bf 28}, ~no. 1, (1999),~57--104.

 \bibitem{adjoint}
{\sc J.~Foniok and C.~Tardif}, {\em Adjoint Functors in Graph Theory}, arXiv:1304.2215v1, 2013.
\bibitem{ladjoint}
{\sc J.~Foniok and C.~Tardif}, {\em Digraph functors which admit both left and right adjoints}, arXiv:1304.2204v1, 2013.
\bibitem{ghebleh}
{\sc M.~Ghebleh}, {\em Theorems and Computations in Circular Colourings of Graphs}, Ph.D. Thesis, Simon Fraser University, 2007.

 \bibitem{imrich}
{\sc R.~Hammack, W.~Imrich, S.~Klav\v{z}ar}, {\em Handbook of Product Graphs}, Second Edition, CRC Press, 2011.

 \bibitem{HRG}
{\sc A.~Habel}, {\em Hyperedge replacement: grammars and languages}, Spriger-Verlag, 1992.

 \bibitem{haji}
{\sc H.~Hajiabolhassan and A.~Taherkhani}, {\em Graph powers and graph homomorphisms}, The Electronic Journal of Combinatorics,
{\bf 17} (2010), R17.

 \bibitem{15}
{\sc P. Hell}, {\em From Graph Colouring to Constraint Satisfaction: There and Back Again}, In M. Klazar, J. Kratochv´ıl, M. Loebl, J. Matouˇsek, R. Thomas, P. Valtr (eds.) Topics in Discrete Mathematics, Springer Verlag, (2006), ~407--432.

 \bibitem{18}
{\sc P.~Hell and J.~Ne\v{s}et\v{r}il}, {\em Colouring, Constraint Satisfaction, and Complexity}, Comp. Sci. Review 2, {\bf 3} (2008), 143--163.
\bibitem{HENE}
{\sc P.~Hell and J.~Ne\v{s}et\v{r}il}, {\em Graphs and Homomorphisms}, Oxford Lecture Series in Mathematics and its Applications, Oxford University Press, Oxford, 2004.

 \bibitem{16}
{\sc P.~Hell and J.~Ne\v{s}et\v{r}il}, {\em On the complexity of H-colouring}, J. Combin. Theory B, {\bf 48}, (1990),~92--100.
\bibitem{trans}
{\sc P.~Hell and J.~Ne\v{s}et\v{r}il, X.~Zhu}, {\em Duality and Polynomial Testing of Tree Homomorphisms}, Transactions of the American Mathematical Society, 4, {\bf 348}, (1996),~1281--1297.

 \bibitem{LAWI00}
{\sc E.~R.~Lamken and R.~M.~Wilson}, {\em Decompositions of edge-colored complete graphs}, J. Combinatorial Theory
Series A, {\bf 89}, (2000), ~149--200.

 \bibitem{MSS}
{\sc A.~W.~Marcus, D.~A.~Spielman and N.~Srivatava}, {\em Interlacing families I: bipartite Ramanujan graphs of all degrees}, 	arXiv:1304.4132, 2013.

 \bibitem{Abbass}
{\sc A.~Mehrabian} , {\em Personal communication}, 2007.


\bibitem{fiber}
{\sc J.~Ne\v{s}et\v{r}il, M.~Siggers, L.~Z\'{a}dori,} , {\em Combinatorial Proof that Subprojective Constraint Satisfaction Problems are NP-Complete}, In: MFCS 2007, Lecture Notes in Computer Science, 4708, 2007.


\bibitem{fiber1}
{\sc J.~Ne\v{s}et\v{r}il, M.~Siggers, L.~Z\'{a}dori,} , {\em A Combinatorial Constraint Satisfaction Problem Dichotomy Classification Conjecture}, Eur. J. Comb., 31.1. 2010,~280--296.

 \bibitem{pultr}
{\sc A.~Pultr}, {\em The right adjoints into the categories of relational systems},
Reports of the Midwest Category Seminar, IV (Berlin), Lecture Notes in Mathematics, Vol. {\bf 137}, Springer, 1970, ~100--113.

 \bibitem{St}
{\sc J.~R.~Stallings}, {\em Topology of finite graphs}, Inventiones Mathematicae {\bf 71} (1983), ~551--565.

 \bibitem{Serre}
{\sc J.~P.~Serre}, {\em Trees}, Translated by: John Stillwell, Springer-Verlag, 1980.

 \end{thebibliography}
 \end{document}